\documentclass[reqno,11pt]{amsart}
\usepackage{amsmath,amssymb,amsthm,graphicx,a4wide,enumerate,url}
\usepackage[small,bf]{caption} 
\usepackage{color}

\newtheorem{theorem}{Theorem}[section]
\newtheorem{lemma}[theorem]{Lemma}
\newtheorem{proposition}[theorem]{Proposition}
\newtheorem{corollary}[theorem]{Corollary}

\theoremstyle{definition}
\newtheorem{definition}[theorem]{Definition}

\theoremstyle{remark}
\newtheorem{remark}[theorem]{Remark}

\numberwithin{equation}{section}

\begin{document}
\parskip.9ex

\title[Spatial Manifestations of Order Reduction in RK Methods for IBVPs]
{Spatial Manifestations of Order Reduction in Runge-Kutta Methods for Initial Boundary Value Problems}

\author[R. R. Rosales]{Rodolfo Ruben Rosales}
\address[Rodolfo Ruben Rosales]
{Department of Mathematics \\ Massachusetts Institute of Technology \\
77 Massachusetts Avenue \\ Cambridge, MA 02139}
\email{rrr@math.mit.edu}

\author[B. Seibold]{Benjamin Seibold}
\address[Benjamin Seibold]
{Department of Mathematics \\ Temple University \\
1805 North Broad Street \\ Philadelphia, PA 19122}
\email{seibold@temple.edu}
\urladdr{http://www.math.temple.edu/\~{}seibold}

\author[D. Shirokoff]{David Shirokoff}
\address[David Shirokoff]
{Department of Mathematical Sciences \\ New Jersey Institute of Technology \\ University Heights \\ Newark, NJ 07102}
\email{david.g.shirokoff@njit.edu}
\urladdr{https://web.njit.edu/\~{}shirokof}

\author[D. Zhou]{Dong Zhou}
\address[Dong Zhou]
{Department of Mathematics \\ California State University Los Angeles \\
5151 State University Drive \\ Los Angeles, CA 90032}
\email{dzhou11@calstatela.edu}


\begin{abstract}
This paper studies the spatial manifestations of order reduction that occur when time-stepping initial-boundary-value problems (IBVPs) with high-order Runge-Kutta methods. For such IBVPs, geometric structures arise that do not have an analog in ODE IVPs: boundary layers appear, induced by a mismatch between the approximation error in the interior and at the boundaries. To understand those boundary layers, an analysis of the modes of the numerical scheme is conducted, which explains under which circumstances boundary layers persist over many time steps. Based on this, two remedies to order reduction are studied: first, a new condition on the Butcher tableau, called weak stage order, that is compatible with diagonally implicit Runge-Kutta schemes; and second, the impact of modified boundary conditions on the boundary layer theory is analyzed.
\end{abstract}

\maketitle

\newcommand{\hfillL}{\hspace*{0mm}\hfill}
\newcommand{\hfillR}{\hfill\hspace*{0mm}}
\newcommand{\Srm}{\textrm{\S}}
\newcommand{\TheoremEnd}{~\hfill$\clubsuit\/$}
\newcommand{\matrixM}{derivative coefficient matrix}
\newcommand{\MatrixM}{Derivative Coefficient Matrix}
\newcommand{\matrixMtilde}{derivative coefficient matrix}
\newcommand{\MatrixMtilde}{Derivative Coefficient Matrix}
\newcommand{\veci}[1]{\vec{#1}}					
\newcommand{\vecii}[1]{\underbar{\mbox{\boldmath$#1$}}}		
\newcommand{\vecipower}[2]{\veci{#1}^{\,#2}}			
\newcommand{\veciipower}[2]{\vecii{#1}^{\,#2}}			
\newcommand{\vecc}{{\veci{c}}} 					
\newcommand{\vece}{\veci{e}} 					
\newcommand{\diffop}{\mathcal{L}} 				
\newcommand{\bdryop}{\mathcal{B}} 				
\newcommand{\dt}{{\Delta t}} 					
\newcommand{\ef}{\psi}						
\newcommand{\EF}{\Psi}						
\newcommand{\mbc}[1]{\vec{g}_{\text{MBC}#1}} 			
\newcommand{\embc}{\vec{g}_{\text{EMBC}}} 			
\newcommand{\sov}[1]{\veci{\tau}^{\,(#1)}} 			
\newcommand{\soq}{\tilde{q}}					
\newcommand{\dimV}{\sigma }					
\newcommand{\projP}{V}
\newcommand{\projPperp}{V_{\perp}}
\newcommand{\mytilde}[1]{\tilde{#1}}
\newcommand{\BlockM}{\Gamma}				
\newcommand{\SchurBlock}{S_{\Gamma}} 		
\newcommand{\MatD}{D}						
\newcommand{\MatB}{C}						
\newcommand{\tempw}{\gamma}					
\newcommand{\VecErr}{h}						
\newcommand{\zeromode}{\psi_0}				
\newcommand{\Erralpha}{\alpha}				

\newcommand{\err}{\epsilon}					
\newcommand{\errs}{\epsilon}				
\newcommand{\lte}{\delta}					
\newcommand{\ltes}{\delta}					
\newcommand{\ltel}{\delta_{\rm 0}}					
\newcommand{\errl}{\epsilon_{\rm 0}}				
\newcommand{\ampk}[1]{H_{#1}(\dt)}
\hyphenation{Mas-sachusetts}

\newcommand{\rsol}{\Phi^{m}}
\newcommand{\hsol}{\Psi}
\newcommand{\esol}{\tilde{\Delta}}
\newcommand{\mpow}{m}
\newcommand{\powfail}{\kappa}


\section{Introduction}
Runge-Kutta (RK) methods advance a time-dependent differential equation forward in time by means of multiple stages. Each stage corresponds to one right hand side evaluation or solve, and appropriate linear combinations of those evaluations generate a high order of accuracy. Two particular advantages of RK schemes over alternatives, such as multistep schemes, are their locality in time and their stability properties \cite{HairerNorsettWanner1993}. In particular, for stiff problems, many types of high-order implicit RK (IRK) methods exist that are A-stable.

Drawbacks of RK methods are their computational cost per time step, as well as order reduction: when applied to certain stiff problems, the observed order of accuracy of the numerical solution may be lower than the (formal) order of the scheme. While order reduction can be rationalized for ordinary differential equations (ODEs) in terms of stiff limits \cite{ProtheroRobinson1974, WannerHairer1991}, for initial boundary value problems (IBVPs) geometric features in the spatial error play a key role. The specific focus of this paper is: (a)~a modal analysis and geometric (via singular perturbation theory) understanding of the global-in-time spatial error, including the accuracy of gradients; (b)~the impact of weak stage order (WSO)---a new condition on RK schemes that remedies order reduction---and modified boundary conditions on the spatial error and by what means these properties remedy or alleviate order reduction.
Specifically, we consider problems of the form
\begin{equation}\label{eq:IBVP}
\begin{cases}
u_t  = \diffop u + f & \mbox{in}\quad\Omega\times(0,t_\text{f}),\\
\mathcal{B}u  = g    & \mbox{on}\quad\partial\Omega\times[0,t_\text{f}], \\
u  = u_0             & \mbox{on}\quad\Omega\times\{t=0\},
\end{cases}
\end{equation}
where $\diffop$ is a linear differential operator, and $\bdryop$ is a boundary operator.  Most of the presentation/analysis in this paper focuses on \eqref{eq:IBVP} with Dirichlet boundary conditions (b.c.), a linear, second-order operator $\diffop$ (e.g., diffusion), and $\Omega = (0, 1)$, because those simple situations suffice to establish the fundamental spatial manifestation of order reduction. However, the order reduction phenomenon arises similarly in higher dimensions, for other types of b.c.\ (see \Srm\ref{sec:outlook}), and other differential operators (see \Srm\ref{ssec:NumExamples_Airy}), albeit with additional effects that are not studied here. Note, though, that many of the structural results and techniques developed in this paper (particularly \S\ref{sec:WeakStageOrder}) transfer to more general situations.

Order reduction for IBVPs incurs some fundamental differences to the stiff ODE case, most prominently: (i)~time discretizations of \eqref{eq:IBVP} are formally infinitely stiff (i.e., eigenvalues of $\diffop$ may be arbitrarily large in magnitude); (ii)~for IBVPs, spatial derivatives of the solution may be important and also exhibit order reduction; and (iii)~boundary conditions play a crucial role in the manifestation of order reduction for IBVPs. In particular, the naive thing to do for a RK method is to impose the b.c.\ for the PDE at the times $t_i\/$ associated with the stages, i.e., $u_i = g_i = g(t_i)\/$ in the case of Dirichlet b.c.. These \emph{conventional b.c.}\ let the error vanish at the boundary, yet lead to the paradoxical situation that for IBVPs, RK schemes may \emph{lose} accuracy because the approximation is \emph{too accurate} near the boundary. As we will show below, the effect of conventional b.c.\ will give rise to a singularly perturbed problem for the spatial numerical error and generate boundary layers (BLs).

A crucial property of the order reduction phenomenon studied here is that the loss of convergence order is caused solely by the time discretization. Therefore, the analysis in this paper focuses on semi-discrete problems, where only time is discretized, but space is left continuous; and likewise, all numerical examples are conducted with an extremely fine spatial resolution. This is feasible, as we restrict to schemes that are unconditionally stable when they are applied to problem \eqref{eq:IBVP}. The restriction to the semi-discrete case has an important implication: the order reduction phenomenon cannot be simply overcome by the choice of a specific spatial discretization; any spatial discretization that converges (as $\Delta x\to 0$) to the semi-discrete limit will encounter the order reduction phenomenon studied here.

\subsection{A Simple Example IBVP}
\label{ssec:example_order_loss}
Here we demonstrate (a)~that order reduction can occur with straightforward schemes (e.g., DIRK), applied to simple problems (e.g., the 1D heat equation); and (b)~how it manifests spatially. The only aspect that is strictly needed is that the problem has time-dependent forcing or b.c.; autonomous problems do not incur order reduction (see \cite{SanzSernaVerwer1989a, SanzSernaVerwer1989b, SanzVerwerHundsdorfer1986} for fully discrete schemes; \cite{OstermannRoche1992} for discrete-in-time schemes).

Consider the IBVP \eqref{eq:IBVP} with $\diffop = \partial_{xx}\/$, $\Omega=(0,1)\/$, and forcing $f$, Dirichlet b.c.~$g$, and initial conditions (i.c.)~$u_0$ chosen so that the exact solution is $u(x\/,\,t) = \cos(t)\/$. We discretize the problem in space, on a uniform grid with $10000\/$ points, using standard second order centered differences (so that spatial errors are negligible relative to temporal errors). Finally, the resulting system is advanced forward in time using standard first to fourth order DIRK schemes (see Appendix~\ref{app:listofRK} for the schemes used).

\begin{table}
	\begin{tabular}{|l|cccc|}
		\hline
		& DIRK1=BE & DIRK2 & DIRK3 & DIRK4 \\
		\hline\hline
		convergence order of $u$        & 1 & 2   & 2   & 2		\\
		convergence order of $u_x$      & 1 & 1.5 & 1.5 & 1.5	\\
		convergence order of $u_{xx}$   & 1 & 1   & 1   & 1		\\
		convergence order of $u_{xxx}$  & 1 & 0.5 & 0.5 & 0.5	\\
		\hline
	\end{tabular}
	\vspace{-.4em}
	\caption{Observed convergence order (in time) for DIRK~1 to 4. DIRK~1 is backward Euler; DIRK~2 to 4 can be found in Appendix~\ref{app:listofRK}.}
	\label{tab:observed_convg_order}
\end{table}

Table~\ref{tab:observed_convg_order} shows the resulting convergence orders for the solution and its spatial derivatives (which are frequently important in IBVPs for body forces, Dirichlet-to-Neumann maps, etc.), all measured in the maximum norm. Backward Euler (BE=DIRK1) shows no order reduction in function value or derivatives. DIRK2 shows a reduction of half an order per spatial derivative ($u_x$ converges with $O(\dt^{1.5})$; $u_{xx}$ with $O(\dt)$, etc.). More severe order reduction arises for DIRK3 and DIRK4: they converge at the same orders as DIRK2. This highlights two important messages. First, order reduction can arise already in very simple problems. Second, it manifests in two ways: (i)~spatial derivatives may be less accurate than function values; and (ii)~schemes of order higher than two may drop to second order (less for spatial derivatives). A geometric explanation for these observations follows.

\subsection{Geometric Explanation of Order Reduction via Boundary Layers}
\label{ssec:orderloss_lte}
The cause for the observations in Table~\ref{tab:observed_convg_order} can be illustrated by studying the shape of the truncation errors. Figure~\ref{fig:errorshape_dirk123} shows the local (single time step) and global (fixed final time) errors in space, for the 1D heat equation problem considered in \Srm\ref{ssec:example_order_loss}, using backward Euler (DIRK1), DIRK2, and DIRK3, respectively. In each panel, results for three choices of $\dt$ are shown, with successive ratios of $2$. For \emph{all} schemes, boundary layers appear locally. However, for DIRK1 the boundary layers vanish globally, while for DIRK2 and DIRK3 they persist globally.

The error in the interior of the domain always scales like the order of the method, but the boundary layer amplitudes scale like $O(\dt^2)\/$. For DIRK3, this results in an order reduction of 1 for $u$. Moreover, any boundary layer has a thickness of $O(\sqrt{\dt})$, resulting in a/an (additional) reduction of half an order per spatial derivative.

\begin{figure}
	\hfillL
	\includegraphics[width=0.44\textwidth]{%
		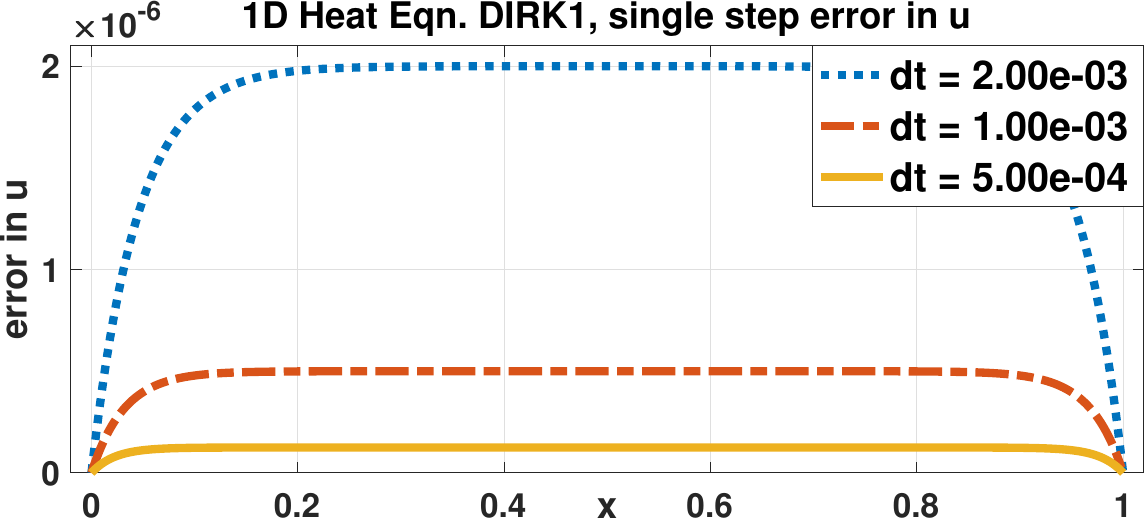}
	\hfill
	\includegraphics[width=0.44\textwidth]{%
		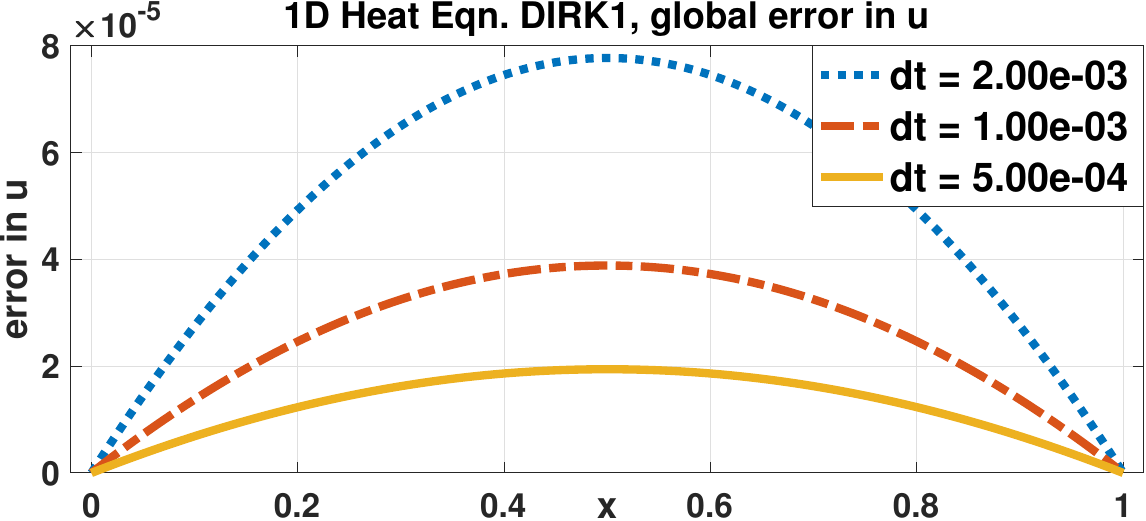}
	\hfillR\\
	\hfillL
	\includegraphics[width=0.44\textwidth]{%
		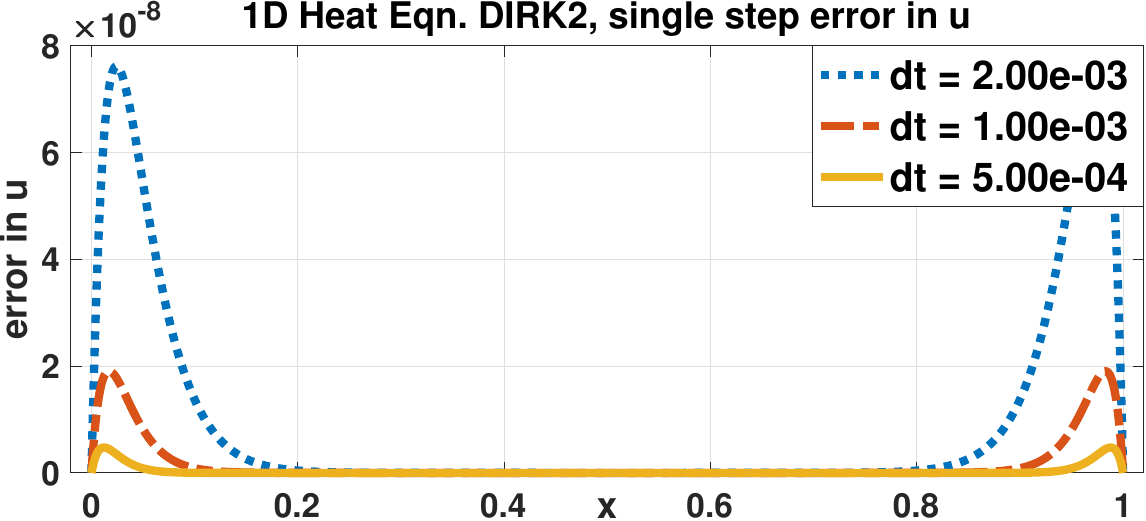}
	\hfill
	\includegraphics[width=0.44\textwidth]{%
		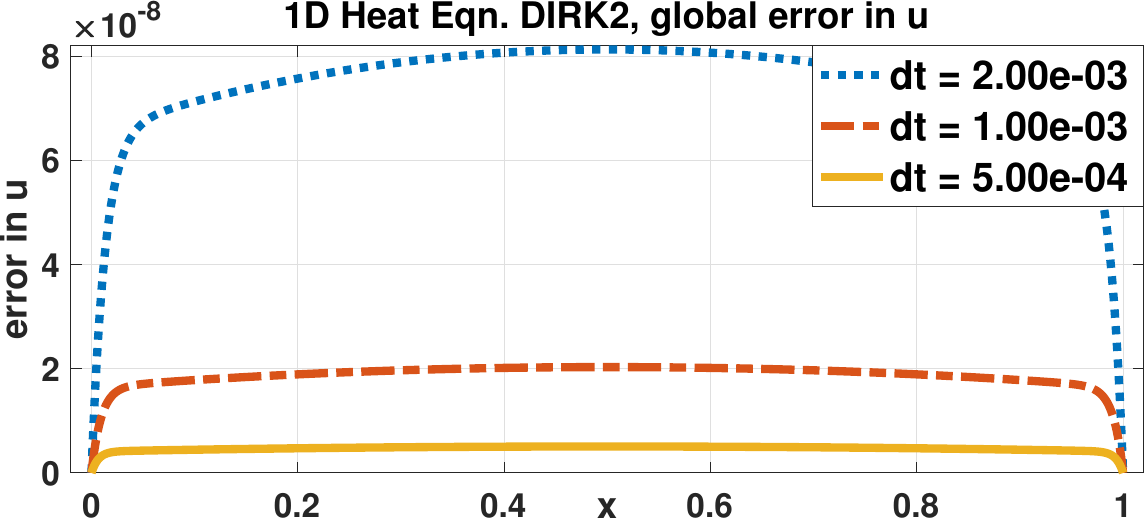}
	\hfillR\\
	\hfillL
	\includegraphics[width=0.44\textwidth]{%
		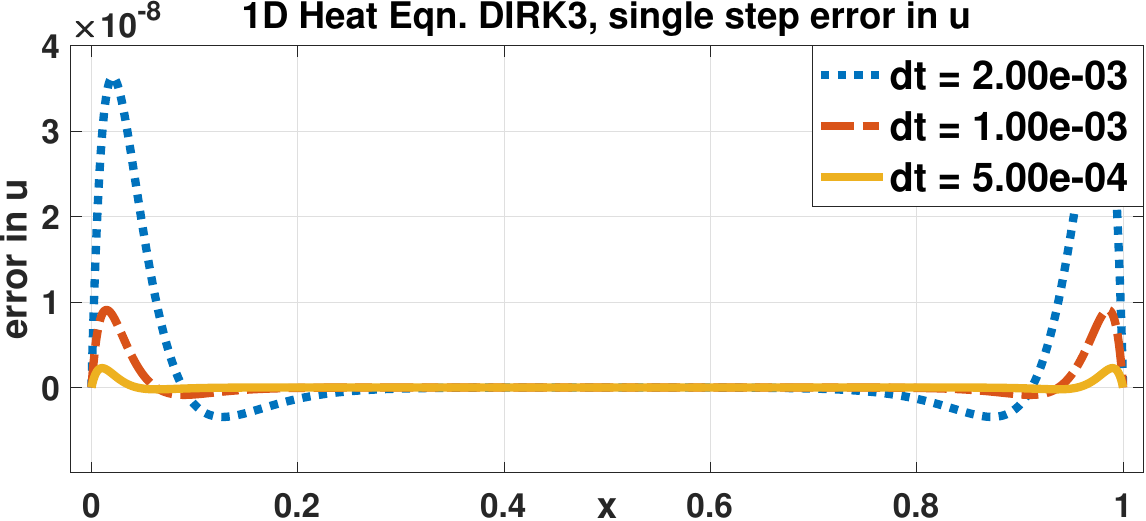}
	\hfill
	\includegraphics[width=0.44\textwidth]{%
		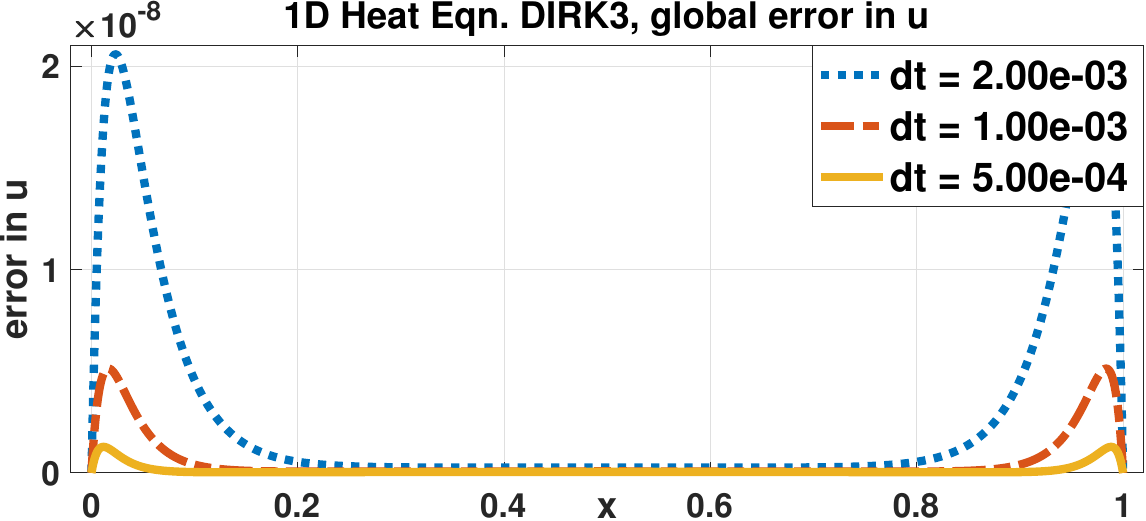}
	\hfillR
	\caption{Local (left) and global (right) errors as functions of $x\/$ for
		BE (top), DIRK2 (middle), and DIRK3 (bottom) with three $\dt\/$ choices.}
	\label{fig:errorshape_dirk123}
\end{figure}

Why boundary layers arise in the approximation error can be understood as follows. Every stage of a DIRK scheme is a backward Euler-type solve. Therefore, it is useful to first examine the BE scheme, applied to the 1D heat equation in the unit interval
\begin{equation}\label{eq:backwardEuler_step}
u^{n+1} - \dt\,u^{n+1}_{x\/x} = u^n + \dt\,f^{n+1}
\quad\mbox{for}\quad 0 < x < 1\/,
\end{equation}
with a smooth forcing, and conventional Dirichlet b.c.\ applied, i.e., $u^{n+1} = g^{n+1}\/$ for $x\in\{0,1\}$. Let $u^{*}$ be the exact solution. Then the approximation error at time $t_{n+1}$, defined as $\errl^{n+1}=u^{n+1}-u^*(t_{n+1})$, satisfies the BVP
\begin{equation}\label{eq:BackwardEulerStepError}
\errl^{n+1} - \dt\, \partial_{xx}\errl^{n+1} = \errl^n + F^{n+1}
\quad\text{for}\quad 0 < x < 1\/,
\end{equation}
with homogeneous Dirichlet b.c.. Here $F^{n+1} = - (u^*(t_{n+1})-u^*(t_n)-\dt\,u^*_{xx}(t_{n+1})-\dt\,f^{n+1})=O(\dt^2)$.

Problem \eqref{eq:BackwardEulerStepError} is a singularly perturbed BVP, where the time step $\dt$ is the small parameter. Standard boundary layer (BL) theory~\cite{BenderOrszag1978} implies that, generally, the solution exhibits a BL of thickness $O(\sqrt{\dt})\/$, and amplitude determined by $\errl^n + F^{n+1}$. If the i.c.\ are captured exactly, i.e., $\errl^0 = 0$, the BL amplitude is $O(\dt^2)$. This explains exactly the top left panel in Figure~\ref{fig:errorshape_dirk123}.

Higher order DIRK schemes combine intermediate stage solutions, each of which arises from a backward Euler-type (and thus singularly perturbed) problem. The BLs of the stage solutions do generally not cancel out, thus yielding a composite layer in the numerical approximation. How the composition of boundary layers from the stages works, and why they may vanish globally in time (cf.~BE), is the focus of the study in \Srm\ref{sec:orderloss_gte}; reducing the size of the BL is the focus of \Srm\ref{sec:WeakStageOrder} and \Srm\ref{sec:MBC}.

\subsection{Current Paper in Context of Prior Research}
\label{ssec:PriorResearch}

\subsubsection{Error analysis for RK order reduction in IBVPs.}
Early work \cite{BrennerCrouzeixThomee1982} highlighted that RK order conditions are in general not sufficient to ensure classical ($p$-th order) convergence for IBVP. For linear constant coefficient PDEs, they observed that additional conditions on the RK scheme, referred to as ``strictly accurate'' (equivalent to stage order) of order $p-1$, were sufficient to obtain a $p$-th order global error.

Subsequent studies in the 1980s \cite{Verwer1986, SanzSernaVerwer1989b, SanzVerwerHundsdorfer1986, SanzSernaVerwer1989a} examined the fully discrete (space and time) error incurred by RK methods applied to method of lines discretizations for IBVPs. The studies revealed that for IBVP with time dependent boundary conditions:
(i)~RK schemes may drop to second order temporal error \cite{Verwer1986, SanzSernaVerwer1989b};
(ii)~the global error may be smaller than the local truncation error \cite{Verwer1986, SanzSernaVerwer1989b, SanzVerwerHundsdorfer1986};
(iii)~order reduction may be improved if the solution happens to satisfy additional (but not natural) compatibility conditions at the boundaries \cite{SanzSernaVerwer1989b, SanzVerwerHundsdorfer1986, SanzSernaVerwer1989a}. For instance, \cite{Verwer1986} derived global $\ell^2$ error convergence rates for linear problems with time-dependent Dirichlet b.c.\ and demonstrated numerically that 3rd and 4th order DIRK schemes perform no better than 2nd order schemes. Many of these results are also included in \cite[Chapter~II.2]{HundsdorferVerwer2003}.

Rigorous error analyses for RK schemes applied to linear PDEs in the semi-discrete setting (in time only, space continuous) \cite{AlonsoPalencia2003, GonzalezOstermann1999, LubichOstermann1993, OstermannRoche1992} have also shown: (i)~the RK convergence order may be limited by the scheme's stage order \cite{AlonsoPalencia2003, OstermannRoche1992}; (ii)~the order (reduction) is in general fractional and depends on the b.c.\ and regularity of the solution \cite{OstermannRoche1992}; and (iii)~for parabolic equations, full convergence order is attained sufficiently far from boundaries \cite{LubichOstermann1995} and that order reduction is localized to the boundaries. Similar estimates have also been given for quasilinear parabolic \cite{LubichOstermann1995quasilinear} and nonlinear equations \cite{OstermannThalhammer2002}. Order reduction in the context of singular perturbation problems, and the interaction between the time step and the small parameter, have been examined in \cite{BoscarinoQiuRusso2018}.

Prior work has focused primarily on RK convergence rates quantifying the size of the error norms, and providing qualitative information on the error incurred in RK schemes (i.e., \cite{LubichOstermann1995} and \cite[Chapter~II.2]{HundsdorferVerwer2003} show errors are localized to domain boundaries).
\emph{The work here is complementary:}
we do not focus on direct estimates that bound the norms of the RK spatial error, but rather on the shape of the spatial error and its implications on the accuracy of quantities of practical interest (e.g., derivatives of the solution at the boundary). The key result in Theorem~\ref{thm:Meigenvalues} characterizes the RK spatial error as a singular perturbation problem with boundary layers. When that error is measured in norms, the results of prior work are recovered---however, the shape of the error provides additional important insight into how to remove order reduction, and what limitations stand in the way of removing it.

\subsubsection{Avoiding order reduction in IBVPs.}
There are several known approaches for remedying order reduction, with this work focusing on the following two.
\begin{enumerate}[i)]
\item Modified boundary conditions. Approaches to overcome order reduction for explicit RK schemes, applied to advective problems without forcing, have been proposed based on modifying the intermediate stage b.c.~\cite{CarpenterBottliebAbarbanelDon1995, AbarbanelGottliebCarpenter1996}. Further improvements have been developed in the context of conservation laws \cite{Pathria1997}. For the linear problems \eqref{eq:IBVP}, \cite{Alonso2002, AlonsoCano2004} derived high order modified b.c.\ and proved that they remedy order reduction to arbitrary order (\cite{Alonso2002} for autonomous $\diffop$ and \cite{AlonsoCano2004} for time-dependent $\diffop(t)$). Those papers also provide convergence results for $\diffop u$, demonstrating that derivatives of $u$ can be less accurate than $u$. In \S\ref{sec:MBC}, we provide a more general approach to deriving modified b.c.\ based on insights gained from viewing the RK error as a singular perturbation problem.
\item Time stepping coefficients with extra conditions. High stage order is the most straightforward condition on RK coefficients that will avoid order reduction---yet it is restrictive, and not compatible with high order DIRK schemes (see \cite[Chapter~IV.15]{HairerWanner1999}, and \S\ref{ssec:IRKbasics}). For Rosenbrock-Wanner (ROW) methods applied to linear problems, \cite{Scholz1989} devised conditions weaker than stage order that alleviate order reduction (cf.~\cite{OstermannRoche1993}). For RK schemes applied to linear IBVPs, similar conditions were stated in \cite{OstermannRoche1992} (see Remark~\ref{rem:OstermannRoche}), but no corresponding RK schemes were provided. In \S\ref{sec:WeakStageOrder} we introduce new conditions on the Butcher tableau, referred to as \emph{weak stage order} (WSO), that are sufficient for, yet simpler than, the conditions in \cite{OstermannRoche1992}. We then obtain new schemes that satisfy WSO and demonstrate that they alleviate order reduction. Note that the WSO conditions may appear formally equivalent to the ROW method conditions in \cite[Equations (3.11')]{OstermannRoche1993}; however, they apply to RK methods and are derived under a more general condition that does not assume an SDIRK structure as in \cite{OstermannRoche1993}. We also note that conditions in \cite{Rang2014} provide, in general, a subset of the WSO conditions that may improve convergence in the stiff limit. Note that related work \cite{KetchesonSeiboldShirokoffZhou2019} further develops (and characterizes the limitations of) WSO schemes satisfying the eigenvector criterion introduced in Def.~\ref{def:eig_criterion} below; however, \cite{KetchesonSeiboldShirokoffZhou2019} does not discuss the spatial manifestations of order reduction.
\end{enumerate}
In addition to the above approaches, the works \cite{SanzVerwerHundsdorfer1986} (for explicit schemes) and \cite{CalvoPalencia2002} (implicit schemes) provide a conceptually simple, yet practically complicated, methodology for avoiding order reduction: decompose the solution into one part from an IBVP that does not exhibit order reduction, and another part obtained directly from the data. We do not examine the spatial manifestation of such approaches here.

Finally, methods equivalent to multistage methods are prone to order reduction. Specifically, deferred correction methods \cite{Minion2003, BoscarinoQiuRusso2018} (see also \cite{MinionSaye2018}, and references within, for a review of order reduction in deferred correction and Gauss quadrature methods) exhibit order reduction since they can be recast as RK methods \cite{GottliebKetchesonShu2009, KetchesonWaheed2014}; similarly for extrapolation methods \cite{KetchesonWaheed2014}, Runge-Kutta-Nystr{\"o}m methods \cite{AlonsoCanoMoreta2005}, and ROW methods \cite{Scholz1989}, etc.. Multistep methods, as implied by the error estimates in \cite{Lubich1991} and \cite[Chapter IV.15]{HairerWanner1999}, do \emph{not} exhibit order reduction.

This paper is organized as follows. In \Srm\ref{sec:orderloss_gte}, based on a characterization of the spatial behavior of the global-in-time error we show that the error arises as a singular perturbation problem with BLs. Sections~\ref{sec:WeakStageOrder} and~\ref{sec:MBC} focus on remedies to order reduction: in \Srm\ref{sec:WeakStageOrder} the concept of \emph{weak stage order} is introduced, which (i)~makes the BLs that affect the final result as accurate as the scheme's order, and (ii)~is compatible with diagonally IRK (DIRK) schemes; and in \Srm\ref{sec:MBC}, the impact of \emph{modified boundary conditions} (MBC) on BLs and order reduction is studied.  Numerical results are shown in \Srm\ref{sec:NumericalExamples}, demonstrating the spatial manifestations of order reduction, and its remedies, in various examples. Generalizations are discussed in \Srm\ref{sec:outlook}.

\section{Boundary Layers in the Global Error and Order Reduction}
\label{sec:orderloss_gte}
As seen in the prior section, RK schemes can yield singular behavior, such as BLs, in the numerical solution---and this behavior serves as a root mechanism of order reduction for IBVPs. However, the existence of a singular perturbation problem in each RK stage does not strictly imply the formation of a BL, and order reduction, in the global truncation error. For example, DIRK2 and DIRK3 can produce BLs in the global error, while backward Euler (BE) does not.  This section provides an analysis that characterizes the global error behavior in space, and derives conditions under which order reduction does or does not occur.

\subsection{Review of Implicit Runge-Kutta Time-Stepping for IBVP}
\label{ssec:IRKbasics}
Here we briefly collect the key notation and results for implicit RK schemes used in the paper. The time-stepping coefficients for a general RK scheme may be represented by the Butcher notation
\begin{equation*}
	\begin{array}{c|c}
		\vecc & A\ \\ \hline
		& \vecipower{b}{T}
	\end{array}
	\,=\,
	\begin{array}{c|cccc}
		c_1 	& a_{11} &  \cdots & a_{1s}\\
		\vdots	& \vdots &         & \vdots\\
		c_s		& a_{s1} &  \cdots & a_{ss}\\ \hline
		& b_{1}  &  \cdots & b_{s}
	\end{array}\/,
\end{equation*}
where the entries of $\vecc$ are the row sums of $A$, i.e., $\vecc=A\vece$, and $\vece = (1, \ldots, 1)^T$ denotes the $s$-dimensional vector of all ones.

Diagonally implicit Runge-Kutta (DIRK) schemes are an important sub-class of implicit RK schemes. For these schemes the matrix $A$ is lower triangular; and has non-vanishing diagonal entries when $A$ is non-singular. DIRK schemes are particularly simple, because the stages can be solved sequentially, with each solve being a small modification of a backward Euler step.

An unconditionally stable RK scheme (so that the semi-discrete limit is justified) applied to the IBVP \eqref{eq:IBVP} takes the form of a BVP problem for the stage values
\begin{equation}\label{eq:dirk_step_intermediatestages}
	u^{n+1}_i = u^n + \dt \sum_{j = 1}^s a_{ij}
	\left(\diffop u^{n+1}_j + f^{n+1}_j\right)
	\quad\text{with b.c.\ for }u^{n+1}_i\/,
\end{equation}
followed by an explicit update rule for the new value
\begin{equation}\label{eq:dirk_step_finalstage}
	u^{n+1} = u^n + \dt\sum_{j=1}^s b_j \left( \diffop u^{n+1}_j+f^{n+1}_j\right)\/,
\end{equation}
for which no b.c.\ are required. Here $\dt>0$ is the time step, and $u^n$ denotes the numerical solution at time $t_n = n\dt$. The $i$-th stage solution $u^{n+1}_i$ is associated with time $t_n+c_i\dt$, as are the corresponding forcing terms $f^{n+1}_i = f(x,t_n+c_i\dt)$.

A scheme is said to have \emph{(classical) order $p\/$} if for sufficiently smooth solutions, the error obtained from a single RK step is $O(\dt^{p+1})$ (cf.~\cite{HairerNorsettWanner1993}).  This imposes a set of constraints on the RK coefficients, known as the \emph{order conditions} \cite{ButcherBook2008, HairerNorsettWanner1993}. Since we consider linear differential operators, we list here the RK order conditions for linear problems (see \cite[Chapter~II.2]{HairerNorsettWanner1993} for nonlinear problems)
\begin{equation}\label{eq:ordercondition}
	\vecipower{b}{T}A^{j}\,\vecipower{c}{k} = \frac{1}{(j+k+1)\dots(k+1)}
	\quad\text{for}\quad
	0\leq j+k \leq p-1\;\mbox{and}\; j\/,\,k\geq 0\/.
\end{equation}
Here a power of a vector applies to each component, i.e., $\vecipower{c}{k} = ((c_1)^k,\dots,(c_s)^k)^T$. For example, first order schemes require $\vecipower{b}{T}\vece = 1$, second order schemes additionally require $\vecipower{b}{T}\vecc = \tfrac{1}{2}$, while the third order conditions impose two further constraints: $\vecipower{b}{T}\vecipower{c}{2} = \tfrac{1}{3}$ and $\vecipower{b}{T} A \vecc = \tfrac{1}{6}$.
We now introduce the \emph{stage order residuals}, and the definition of stage order \cite[Chapter~IV.5]{WannerHairer1991}:
\begin{equation}\label{eq:StageOrderVector}
	\textrm{(Stage order residuals)} \quad\quad
    \sov{j} = A\vecipower{c}{j-1}-\frac{1}{j}\vecipower{c}{j}\/,
    		\quad j=1,2,\dots.
\end{equation}
\begin{definition}[Stage order] \label{def:StageOrder}
	Condition $B(\tilde{p})$: let $\tilde{p}$ be the largest number such that the quadrature condition holds
	$\vecipower{b}{T}\vecipower{c}{j-1} = j^{-1}$ for $j = 1\ldots \tilde{q}$. Condition $C(\tilde{q})$: let
	$\tilde{q}$ be the largest number such that $\sov{j} = \vec{0}$ for $j = 1\ldots \tilde{q}$.
	The \emph{stage order} of a RK scheme is $q = \min\{\tilde{p}, \tilde{q}\}$.
\end{definition}
It is well-known that schemes with high stage order avoid order reduction in stiff ODEs. Unfortunately, high stage order is a restrictive property for DIRK schemes:
\begin{remark} \label{rem:q_equal_1}
(Stage order in DIRKs)
DIRK schemes with nonzero diagonal entries are limited to stage order $q=1$~\cite{HairerNorsettWanner1993}. Moreover, DIRK schemes with singular $A$ may have stage order $q=2$ (but not higher) \cite{KetchesonSeiboldShirokoffZhou2019}. Examples are EDIRK schemes, such as Crank-Nicolson, or TR-BDF2
\cite{BankCoughranFichtnerGrosseRoseSmith1985}.
\TheoremEnd
\end{remark}

The stage order residuals $\sov{j}$ for $1\leq j \leq q$ become important later, even when they are nonzero. Their significance makes use of the following orthogonality property, which follows immediately from the order conditions \eqref{eq:ordercondition}:
\begin{proposition} \label{prop:StageOrderVectorOrthog}
For a $p$-th order RK scheme, the stage order residuals satisfy $\vecipower{b}{T}A^j \sov{k} = 0$ for $1\leq j+k \leq p-1$, with $j\geq 0$ and $k\geq 1$. In particular, $\vecipower{b}{T}\sov{k} = 0$ for $1\leq k \leq p-1$.
\end{proposition}
Lastly we introduce notation relevant to numerical stability. The \emph{stability function} of the RK scheme is given by $R(\zeta) = 1+\zeta\vecipower{b}{T}(I - \zeta\,A)^{-1}\vece$. The value $R(\zeta)$ measures the growth $u^{n+1}/u^n$ in one step $\dt$, when applying the RK scheme to the test equation $u^\prime(t) = \lambda u$, where $\zeta = \lambda\dt$.
A RK scheme is called \emph{A-stable}, if it is stable for all stable solutions of $u^\prime(t) = \lambda u$ (i.e., $|R(\zeta)|\leq 1$ for $\mbox{Re}(\zeta)\leq 0$);
and \emph{L-stable} if also $R(\zeta)\rightarrow0$ as $\zeta\rightarrow-\infty$.
A RK scheme is called \emph{stiffly accurate} \cite{ProtheroRobinson1974}, if the last row of $A$ equals the vector $\vecipower{b}{T}$, i.e., if $a_{sj} = b_j$ for $j=1,\dots,s$.
A stiffly accurate RK scheme with invertible coefficient matrix $A$, which is A-stable, is also L-stable \cite{WannerHairer1991}, seen by evaluating the $\zeta\to -\infty$ limit of the stability function.

\subsection{Equations for the Approximation Error}
\label{ssec:eqnerrspace}
In this subsection we derive equations for the discrete-in-time RK error incurred by the IBVP \eqref{eq:IBVP}. The full analysis using the Mellin/$z$-transform has been presented in \cite{LubichOstermann1993, LubichOstermann1995}, for normed-error estimates. Our focus is to set the stage for the study of BLs via asymptotic analysis for singular perturbation problems \cite{BenderOrszag1978,kevorkian1996multiple}.
Using asymptotics, we show that order reduction (OR) is restricted to certain regions in space, and that elsewhere no OR occurs.  For simplicity, we restrict the presentation to periodic-in-time solutions of the error equations, because those suffice to capture crucial OR mechanisms. It is important to emphasize that \emph{we do not claim that order reduction happens solely for periodic solutions,} but rather that periodic solutions suffice to provide an intuitive/geometrical visualization of OR for PDEs. Nevertheless, as shown in Appendix~\ref{app:periodicsol}, under some conditions the periodic solutions contain the full OR phenomenon.

As pointed out earlier, the asymptotic analysis in this subsection is for the case where $\diffop$ in \eqref{eq:IBVP} is a second-order operator, with $\Omega = (0, 1)$ and Dirichlet b.c., i.e., $u = g$ on $\partial \Omega$. However, the concepts generalize to other differential operators and boundary conditions. The numerical examples focus on the one-dimensional case as well.

Below, let $W_e(\theta)\/$ denote the wedge in the left complex half plane defined by: $\lambda \in W_e(\theta) \Longleftrightarrow |\mbox{arg}(-\lambda)| < \theta\/$.  We further assume that both the PDE and the scheme are well-defined and stable, in the following sense:
\begin{equation}\label{eq:AssumptionsLS}
\begin{tabular}{l@{~}l}
(a) & \begin{minipage}[t]{0.78\textwidth}
There are constants $K$, $0 <\theta_1< \pi/2\/$ such that: For any $\lambda \notin W_e(\theta_1)\/$  the operator $(I-\lambda\/\diffop\/)$ with homogeneous boundary conditions has a uniformly bounded inverse: $\| (I - \lambda \diffop)^{-1}\,u\|_{L^{\infty}} \leq K \, \| u\|_{L^{\infty}}$;
\end{minipage}\\ \\[-.9em]
(b) & \begin{minipage}[t]{0.78\textwidth}
The scheme's stability region includes $W_e(\theta_2)\/$, for some $\theta_2 > \theta_1\/$.
\end{minipage}\\ \\[-.9em]
(c) & \begin{minipage}[t]{0.78\textwidth}The eigenvalues of $A\/$ have non-negative real parts.\end{minipage}\\ \\[-.9em]
(d) & \begin{minipage}[t]{0.78\textwidth} There are constants $\delta_d, c_d > 0$ such that: For any complex $|z| < \delta_d$, the matrix $\vece\/\vecipower{b}{T}+ z \/A$ is diagonalizable, and the family of eigenvector matrices $T(z)$ can be selected so that their condition number satisfies $\|T(z)\| \, \|T^{-1}(z)\| < c_d$ for $|z| < \delta_d$.\end{minipage}\\ \\[-.9em]
(e) & \begin{minipage}[t]{0.78\textwidth} The matrix $A$ is invertible, and $\vecipower{b}{T}A^{-1}\vece\neq 0$.\end{minipage}\\
\end{tabular}
\end{equation}
Condition (a) is a property of the operator $\diffop\/$ only; condition (b) is a property of the scheme in relation to the operator $\diffop\/$; and conditions (c--e) are properties solely of the numerical scheme.
Condition (c) guarantees that the scheme equations (\ref{eq:dirk_step_intermediatestages}--\ref{eq:dirk_step_finalstage}), equivalently \eqref{eq:dirk_step}, have a well defined solution at each step. It is rather natural for RK schemes, and most commonly used methods (incl.~most DIRK, Gauss, Radau, Lobatto) satisfy it.
Condition (d) is a technical assumption on the RK scheme that will be used to estimate and bound the numerical errors. Requiring a uniform bound on the condition number of $T(z)$ avoids a situation in which two of the eigenvectors (i.e., columns of $T(z)$) become parallel as $z\rightarrow 0$. Condition (d) may be alternatively stated using perturbation theory, via conditions on $A$ and an $(s-1) \times (s-1)$ matrix determined by $A$ and $\vec{b}$. For brevity, however, we leave (d) in its current form. Due to its important role below, we introduce notation for the subspace spanned by the vectors orthogonal to $\vec{b}$:
\begin{equation*}
\vec{b}^\perp = \{\vec{v}:\vecipower{b}{T}\vec{v} = 0\}.
\end{equation*}

Condition (e) is also a technical assumption on the RK scheme. Most commonly used RK schemes satisfy $\vecipower{b}{T}A^{-1}\vece\neq 0$ (in particular, all stiffly accurate schemes do so, because $\vecipower{b}{T}A^{-1}\vece = 1$). However, some schemes do violate it, for example the 2-stage 4-th order Gauss method \cite{HairerNorsettWanner1993}.  Note that some unconditionally stable schemes, such as EDIRK schemes \cite{KvaernoNorsettOwren1996}, also do not have invertible $A$.
Condition (a) is required to estimate the magnitude of the RK numerical error; it is also a numerical stability condition because it guarantees that the spectrum of $\diffop\/$ is contained within $W_e(\theta_1)$.  Condition (a) is satisfied when $\diffop$ is (strongly) elliptic and \eqref{eq:IBVP} is parabolic \cite[Chapter 2]{Pazy1983} (e.g., heat equation)\footnote{In fact, (\ref{eq:AssumptionsLS}a) is closely related to the condition required for solutions of \eqref{eq:IBVP} to be defined by an analytic semigroup.}.  In fact, the inverses of differential operators $(I-\lambda \diffop)$ are generally given by Green's functions, which are singular at the spectrum of $\diffop$ and continuous functions of $\lambda$ away from the spectrum. Specifically, when $(I - \lambda \diffop)^{-1}$ can be written in terms of Green's functions, then condition (a) requires that the Green's function be uniformly bounded, in $L^1$, for $\lambda\/$ outside $W_e(\theta_1)\/$.  Condition (a) does not hold for dispersive equations where $\diffop$ has eigenvalues on the imaginary axis.

We first formulate equations for the RK error, and then examine the equations when the exact solution to \eqref{eq:IBVP} is time-periodic.  One step  (i.e., see \cite{Albrecht1996}) of a RK scheme (\ref{eq:dirk_step_intermediatestages}--\ref{eq:dirk_step_finalstage}) can be written using $\vecipower{u}{n+1} := (u^{n+1}_1,\, \ldots, \, u^{n+1}_s)^T$ as:
\begin{align}\label{eq:dirk_step}
	\textrm{(``internal'' stages)} \quad
	\vecipower{u}{n+1} &= u^{n} \veci{e}
	+ \dt\, A \,
	\left( \diffop\,\vecipower{u}{n+1}+\vecipower{f}{n+1} \right)\/,
	\ \ \text{with b.c.\ for }\vecipower{u}{n+1}\/, \\
    \label{eq:dirk_step2}
	\textrm{(``last'' stage)} \quad
	u^{n+1} &= u^n + \dt \, \vecipower{b}{T} \left( \diffop \vecipower{u}{n+1} + \vecipower{f}{n+1} \right),
	\ \ \text{with no b.c.\ for } u^{n+1},
\end{align}
where $\vecipower{f}{n+1} := (f^{n+1}_1,\, \ldots,\, f^{n+1}_s )^T$, and
$\diffop \vecipower{u}{n+1} := (\diffop u^{n+1}_1,\, \ldots,\, \diffop u^{n+1}_s )^T$.
Denote the exact solution to \eqref{eq:IBVP} by $u^*(x,t)$.  To obtain an
equation for the propagation of the numerical error, let
$\errl^{n}(x) := u^{n}(x) - u^*(x,t_{n})$ be the error,
and  $\err^{n+1}_i(x) := u^{n+1}_i(x)-u^*(x,t_n+c_i\dt)$ be the stage error. Substituting
these expressions for the error into
(\ref{eq:dirk_step}--\ref{eq:dirk_step2}), yields:
\begin{align}\label{eq:ApproxErrIRK}
	\textrm{(``stage'' error)} \quad\quad \vecipower{\err}{n+1}
	&=
	\errl^n \, \vec{e}
	+ \dt\,A\,
	\diffop\,\vecipower{\err}{n+1}
	+ \vecipower{\lte}{n},
	\quad\quad\;\;\text{with b.c.~for}
	\;\;
	\vecipower{\err}{n+1}, \\ \label{eq:ApproxErrIRK1}
	\textrm{(``last'' error)} \quad\quad
	\errl^{n+1} &=
	\errl^n + \dt \, \vecipower{b}{T} \, \diffop\,
	\vecipower{\err}{n+1} + \ltel^{n},
	\quad\quad\;\;\textrm{ with no b.c.~for}\;\; \errl^{n+1}, \\ \nonumber
	\textrm{where} \quad\quad
	\vecipower{\err}{n+1} &= \left(\err^{n+1}_1\/,\,\dots\/,\,\err^{n+1}_s\right)^T,
	\quad\quad\;\;\textrm{and} \quad
	\vecipower{\lte}{n} = \left(\lte_1^n,\dots,\lte_s^n\right)^T.
\end{align}	
Here $\vecipower{\lte}{n}$ and $\ltel^n$ are the local truncation errors (LTEs), and involve only $u^*$ and $\vec{f}$. Formulas for the LTEs can then be obtained using the PDE \eqref{eq:IBVP}, and Taylor expanding $u^*$ (at $t_n$):
\begin{equation}\label{eq:dirk_step_LTE_formula}
	\vecipower{\lte}{n}(x) =
	\displaystyle\sum_{j\geq 1} \frac{\partial_t^ju^*(x,t_n) \, \dt\,^j}{(j-1)!}  \; \sov{j},
	\quad
	\ltel^{n}(x) = \displaystyle\sum_{j\geq 1} \frac{\partial_t^ju^*(x,t_n) \, \dt\,^j}{(j-1)!} \; \big(\vecipower{b}{T}\vecipower{c}{j-1} - \frac{1}{j} \big).
\end{equation}
Here $\sov{j}$ are the stage order residuals defined in \eqref{eq:StageOrderVector}.
One should stress that equations (\ref{eq:ApproxErrIRK}--\ref{eq:dirk_step_LTE_formula}) hold for linear problems only, i.e., \eqref{eq:IBVP}. For a $p$-th order Runge-Kutta scheme with stage order $q$, the first $q\le p$ summands in $\vecipower{\lte}{n}$ vanish; thus $\lte_j^n = O(\dt\,^{q+1})$ ($1\leq j\leq s$).  Meanwhile, the $p$-th order conditions guarantee that $\ltel^{n}=O(\dt\,^{p+1})$ (and hence $\lte_s^n = \ltel^{n} = O(\dt\,^{p+1})$ for stiffly accurate schemes).
For the remainder of~\S\ref{sec:orderloss_gte}, we assume conventional b.c., i.e., $g_i^{n+1} = g(t_n + c_i\,\dt)\/$; equivalently, this yields homogeneous b.c.\ for the error $\vecipower{\err}{n}$ in \eqref{eq:ApproxErrIRK}.
Time-periodic solutions to the IBVP \eqref{eq:IBVP} can be obtained when the forcing and b.c.\ have the form $f=\hat{f}\,e^{\imath\/\omega\/t}\/$ and $g=\hat{g}\,e^{\imath\/\omega\/t}\/$, where $\hat{f}$ and $\hat{g}$ are functions defined on $\Omega\/$ and $\partial\Omega\/$, respectively, and $\omega\/$ is a (real-valued) constant. Then $u^*(x,t) = U^*(x) e^{\imath \omega t}$ is the periodic solution to \eqref{eq:IBVP}, where $U^*$ is the (unique because of (\ref{eq:AssumptionsLS}a)) solution to the BVP $(\imath\/\omega -\diffop)U^* = \hat{f}\/$ with b.c.~$\mathcal{B}\,U^* = \hat{g}\/$. Since RK schemes are linear, time-harmonic forcings and boundary data will also yield periodic numerical solutions.  To obtain periodic solutions for $\vecipower{\err}{n}$ and $\errl^n$ when $u^*(x,t) = U^*(x) e^{\imath \omega t}$, we seek an ansatz of the form (with a slight abuse of notation):
\begin{equation}\label{periodic_ansatz}
	\vecipower{\err}{n}(x) = \veci{\errs}(x) \; z^n, \quad \quad
	\errl^n(x) = \errl(x) \; z^n, \quad \quad
	\vecipower{\lte}{n}(x) = \veci{\ltes}(x) \; z^n, \quad \quad
	\ltel^n(x) = \ltel(x) \; z^n,
\end{equation}
where $z := e^{\imath \omega \dt}$, and
\begin{equation}\label{eq:ltes_vector}
	\veci{\ltes}(x) = U^*(x) \;
	\displaystyle\sum_{j\geq 1}\frac{(\imath \omega)^j \dt\,^j}{(j-1)!}
	\sov{j}, \quad\quad
	\ltel(x) = U^*(x) \;
	\displaystyle\sum_{j\geq 1}\frac{(\imath \omega)^j \dt\,^j}{(j-1)!}
	\left( \vecipower{b}{T}\vecipower{c}{j-1}-\frac{1}{j} \right).
\end{equation}
Substituting \eqref{periodic_ansatz} into \eqref{eq:ApproxErrIRK} and \eqref{eq:ApproxErrIRK1} yields the coupled system for $(\errl, \veci{\err})$:
\begin{equation}\label{eq:MatsystemError}
	\underbrace{\begin{pmatrix}
		1 & 0 \\
		-z^{-1} \vec{e} & I
	\end{pmatrix}}_{\MatB}		
	\begin{pmatrix}
		\errl \\
		\veci{\err}
	\end{pmatrix}	
	-\underbrace{
	\begin{pmatrix}
		0 & (z-1)^{-1} z \dt \, \vecipower{b}{T} \\
		\vec{0} &  \dt \, A
	\end{pmatrix}}_{\BlockM}
	\begin{pmatrix}
		\diffop \errl \\
		\diffop \veci{\err}
	\end{pmatrix}
	=
	\begin{pmatrix}
		(z-1)^{-1} \ltel \\
		z^{-1} \veci{\lte}
	\end{pmatrix}.
\end{equation}
Equation \eqref{eq:MatsystemError} is supplemented with $s$ boundary conditions for $\vec{\err}$, and no boundary conditions for $\errl$. Hence, \eqref{eq:MatsystemError} is actually a differential algebraic equation. The block components in \eqref{eq:MatsystemError} for the differential equation (which involve only the stage errors $\vec{\err}$) may be separated from the algebraic equation (which couple the stage and function errors $\vec{\err}$, $\errl$) by simultaneously transforming $\MatB$ and $\BlockM$ into upper block triangular form. To do this, we first multiply \eqref{eq:MatsystemError} through on the left by a matrix $S_\Gamma$, which block-diagonalizes $\Gamma$, followed by the matrix $\MatD$, where
\begin{equation*}
	\SchurBlock =
	\begin{pmatrix}
		1 & -z \, \vecipower{\tempw}{T} \\
		0 & I
	\end{pmatrix}, \quad
	\MatD =
	\begin{pmatrix}
		1 & 0 \\
		0 & I + \vec{e}\, \vecipower{\tempw}{T}
	\end{pmatrix}
	\begin{pmatrix}
		1 & 0 \\
		z^{-1} \vec{e} & I	
	\end{pmatrix}		
	\begin{pmatrix}
		1 + \vecipower{\tempw}{T}\vec{e}  & 0 \\
		0 & I
	\end{pmatrix}^{-1}, \quad
	\vecipower{\tempw}{T} := \frac{\vecipower{b}{T}A^{-1}}{z-1}.
\end{equation*}
Here we have used that $A$ is invertible to define $\vecipower{\tempw}{T}$. Multiplication yields
\begin{equation}\label{eq:MatsystemError2}
	\underbrace{\begin{pmatrix}
		1 & -\vecipower{\Erralpha}{T} \\
		0 & I
	\end{pmatrix}}_{\MatD \SchurBlock \MatB}		
	\begin{pmatrix}
		\errl \\
		\veci{\err}
	\end{pmatrix}	
	-
	\underbrace{\begin{pmatrix}
		0 & 0 \\
		\vec{0} & M
	\end{pmatrix}}_{\MatD \SchurBlock \BlockM}
	\begin{pmatrix}
		\diffop \errl \\
		\diffop \veci{\err}
	\end{pmatrix}
	=
	\begin{pmatrix}
		\zeromode \\
		\veci{\VecErr}
	\end{pmatrix},	
\end{equation}
where the \emph{\matrixM}
\begin{equation}\label{eq:MatrixM}
		M := \dt\,A + \frac{\dt}{z-1}\,\vece\;\vecipower{b}{T}\/, \quad\quad \textrm{and} \quad\quad
		\vecipower{\Erralpha}{T} := \frac{z \vecipower{b}{T} A^{-1} }{z - 1 + \vecipower{b}{T}A^{-1}\vec{e}},
\end{equation}
appear in the block matrices of \eqref{eq:MatsystemError2}, while
\begin{equation}\label{eq:ErrorFunctions}
	\veci{\VecErr}(x) := \frac{1}{z}\left( \veci{\lte}(x) + \frac{\ltel(x)}{z-1}\veci{e}\right)\/,
	\quad\;
	\zeromode(x) := \frac{1}{z-1 + \vecipower{b}{T} A^{-1} \veci{e}}
	\left( -\vecipower{b}{T} A^{-1} \veci{\lte}(x) + \ltel(x) \right).
\end{equation}
In \eqref{eq:MatsystemError2}, multiplication by $\SchurBlock$ converts $\SchurBlock \, \BlockM$ into a block diagonal matrix; multiplication by $\MatD$ converts $\MatD \SchurBlock \MatB$ into row echelon form while preserving the block structure of $\SchurBlock \BlockM$. Equation \eqref{eq:MatsystemError2} is significant since it allows one to extract the spatial RK error $\errl(x)$. Working out the components of \eqref{eq:MatsystemError2}, the bottom block row yields an $s$-dimensional partial differential equation for the stage errors:
\begin{equation}\label{eq:IRK_mat_eqn}
 \veci{\errs} - M\,\diffop\,\veci{\errs}
 = \frac{1}{z}\left( \veci{\lte} + \frac{\ltel}{z-1}\veci{e}\right),
 \quad \text{with b.c.\ } \veci{\errs} = 0,
\end{equation}
while the top row yields one algebraic expression for the global error $\errl$ in terms of the stage error vector:
\begin{equation}\label{eq:IRK_mat_eqn2}
	\errl(x) = \vecipower{\Erralpha}{T} \veci{\errs} + \zeromode
	= \frac{1}{z-1 + \vecipower{b}{T} A^{-1} \veci{e}}
	\left( z \; \vecipower{b}{T}A^{-1}\veci{\errs} -
	\vecipower{b}{T} A^{-1} \veci{\lte} + \ltel \right).
\end{equation}
To recap, the spatial error vector for the RK stages $\veci{\err}$ (corresponding to the time-periodic response of the error equation) satisfies the BVP system \eqref{eq:IRK_mat_eqn}, in which the \matrixM{}~$M$ pre-multiplies the operator term $\diffop\,\veci{\errs}$. The error of the RK scheme $\errl(x)$ is then computed by evaluating the update rule \eqref{eq:IRK_mat_eqn2}.

For schemes with singular $A$ (see Remark~\ref{rem:q_equal_1}), one could still obtain \eqref{eq:MatsystemError2} by first row-reducing $\BlockM$. Singular $A$, however will yield multiple algebraic equations analogous to \eqref{eq:IRK_mat_eqn} coupling $\veci{\err}$ and $\errl$, and a lower dimensional PDE system analogous to \eqref{eq:IRK_mat_eqn}.

\subsection{Spectral Properties of the \MatrixM}
\label{ssec:EigenValuesofM}
Here we examine the periodic-in-time error equation~\eqref{eq:IRK_mat_eqn} in the eigenbasis defined by the matrix $M$. We carry out an asymptotic analysis and demonstrate that the RK error generically satisfies a singular perturbation problem. We first estimate the eigenvalues and eigenvectors of $M$ for $z = e^{\imath \omega \dt}$ on the unit circle, where $\omega\/$ is real and $\dt$ is small. Of particular relevance is the existence of small eigenvalues for $M$, as they give rise to singularly perturbed BVPs (producing BLs and, potentially, other effects associated with singular BVPs). We call an eigenvalue $\lambda\neq 0$ of $M$ \emph{small} if $\lambda\rightarrow 0$ as $\dt\rightarrow 0$.

Since $z = e^{\imath \omega \dt}$, for $\dt\ll 1$, expanding $\frac{\dt}{z-1}$ in powers of $\dt$ yields
\begin{equation}\label{eq:MatrixM_expansion}
	M = \frac{1}{\imath \omega}\,\vece\,\vecipower{b}{T}
	+ \dt\,(A-\frac{1}{2}\,\vece\,\vecipower{b}{T}) + O(\dt\,^2)\/.
\end{equation}
Because $\vece\,\vecipower{b}{T}$ is a rank-1 matrix with eigenvalue 1 (since $\vecipower{b}{T}\vece = 1$), all eigenvalues of $M$ vanish as $\dt\to 0$, except for one that (due to numerical consistency) is equal to $(\imath \omega)^{-1}$. How these zero eigenvalues are approached as $\dt \to 0$ depends on the structure of $A$ and $\vec{b}$.
\begin{theorem}[Asymptotic eigenvalues of $M$] \label{thm:Meigenvalues}
For $0<\dt\ll 1\/$ and a fixed $\omega\in\mathbb{R}$ such that $z = e^{\imath \omega \dt} \neq 1\/$, the matrix $M\/$, defined in \eqref{eq:MatrixM} satisfies:
\begin{enumerate}[\quad(1)]
\item It has one $O(1)\/$ eigenvalue, which is at most $O(\dt)\/$ away from $\frac{1}{\imath \omega}\/$. The leading order part for the corresponding right and left eigenvectors are $\vece\/$ and $\vec{b}$.
\item It has $s-1\/$ (including multiplicity) small, but nonzero eigenvalues, of magnitude at most $O(\dt)\/$. They have the form $\lambda = \dt \mu_0 + o(\dt)$. Here $\mu_0$ are the eigenvalues of the matrix $QA$ restricted to $\vec{b}^{\perp}$ (denoted by $B$), where $Q = I - \vece\,\vecipower{b}{T}\/$ is the projection onto $\vec{b}^{\perp}\/$ along $\vece$.
\item It has no zero eigenvalues.
\end{enumerate}
\end{theorem}
\begin{proof}
We first prove (3).  Assume $M$ has a zero eigenvalue. Then there exists a nonzero vector $\veci{\ell}$ such that $M\,\veci{\ell} = \dt\,A\,\veci{\ell} + \frac{\dt}{z-1}\, (\vecipower{b}{T}\veci{\ell})\,\vece = \veci{0}$, hence $A\,\veci{\ell}\parallel\vece$. But because $A$ is non-singular, there is a uniquely defined (up to scaling) eigenvector: $\veci{\ell}^* = A^{-1}\vece$. Hence,
\begin{align}\label{Eq:zConstraint}
M\,\veci{\ell}^*= \dt\Big(1+ \frac{\vecipower{b}{T} A^{-1}\vece}{z-1}\Big)\vece = \veci{0}, \quad\quad \Longrightarrow \quad\quad z=1-\vecipower{b}{T}A^{-1}\,\vece.
\end{align}
The constraint \eqref{Eq:zConstraint} can only occur for $\vecipower{b}{T} A^{-1}\vece=2$ and $z=-1$, or $\vecipower{b}{T} A^{-1}\vece=0$ and $z = 1$ (because $\vecipower{b}{T} A^{-1}\vece$ is real and $|z| = 1$). Neither case is possible when $\dt \ll 1$ since $0 < |\omega\dt| < \pi$ implies $z\neq \pm 1$.

Item (1) results from viewing $M$ as a perturbed rank-1 matrix. The matrix $\frac{1}{\imath \omega}\,\vece\,\vecipower{b}{T}$ is diagonalizable and therefore the Bauer-Fike theorem \cite{Bauer1960} implies that $M$ has one eigenvalue $O(\dt)$ away from $(\imath \omega)^{-1}$, and $(s-1)$ eigenvalues of size $O(\dt)$. Since the $(\imath \omega)^{-1}$ eigenvalue is simple (non-repeated), the corresponding eigenvector is, to leading order, $\vec{v}^{(0)} = \vece$ \cite[Chapter~V~2.3]{StewartSun1990}.  The remaining parts for (2) are not as straightforward as (1) since $\lambda = 0$ is a degenerate eigenvalue of $M$ when $\dt = 0$.

To show (2), we write $M = \frac{\dt}{z-1} M^*$, where $M^* := \vece\, \vecipower{b}{T} + \delta A$ and $\delta = z-1$, so that when $\dt \ll 1$, we have $|\delta| \ll 1$. In addition, we write the eigenvalues of $M^*$ as $\lambda^* = \delta \mu$ (with the corresponding eigenvalues of $M$ being $\dt\,\mu$). We now work in a coordinate basis defined by the eigenvectors of the unperturbed matrix $\vece\,\vecipower{b}{T}$. Let $O_b \in \mathbb{R}^{s \times (s-1)}$ be a matrix whose columns form an orthonormal basis for $\vec{b}^{\perp}$. Set $T_0 = [\vece,\; O_b]$. Then $T_0^{-1} = [\vec{b},\; Q^T \/ O_b]^{T}$, because $\vecipower{b}{T}\vece = 1$, $\vecipower{b}{T}O_b = 0$, $(O_b)^T Q \vece = 0$, and $(O_b)^T Q O_b = (O_b)^T O_b = I$. Using $T_0$ as a similarity transformation, the transformed rank-1 matrix becomes $( T_0^{-1} \vece\,\vecipower{b}{T} T_0)_{ij} = \delta_{i1} \delta_{j1}$, where $\delta_{ij}$ is the Kronecker delta. The characteristic equation for $M^*$ then follows from a direct computation of the corresponding determinant:
\begin{equation}\label{Eq:FirstDetLine}
\begin{split}
\det(M^* \!-\! \lambda^* I) &= \det ( T_0^{-1} M^* T_0 - \delta \mu I )
= \det \left( \delta_{i1} \delta_{j1} + \delta \; T_0^{-1} (A - \mu I) T_0 \right) \\
&= \delta^{s-1} \/ \left( \det(B- \mu I)+ \delta \/ \det (A - \mu I) \right).
\end{split}
\end{equation}
Here $B = (O_b)^T \/ QA\/ O_b$ is the bottom right $(s-1)\times (s-1)$ block of $T_0^{-1} \/ A \/ T_0$. The computation shows that the $(s-1)$ eigenvalues $\mu$ are to within $o(1)$ of $\mu_0$, where $\mu_0$ denotes the roots of $\det(B - \mu_0 I) = 0$.
\end{proof}

Without loss of generality, we label the eigenvalues of $M$ in such a way that $\lambda_1 = O(1)$, and $\lambda_2$ through $\lambda_s$ are small but nonzero. The properties below on the matrix $M$ will be used to estimate the size and shape of the RK error in \Srm\ref{ssec:eigenmodeanalysis}.

\begin{proposition}[Eigenvectors of $M$]\label{prop:1stLeftEigenVector}
For a $p$-th order RK scheme, let $\vecipower{\ell}{T}_i$ and $\veci{r}_i$ be the left and right eigenvectors of $M$ associated with $\lambda_i$, normalized so that $\vecipower{\ell}{T}_i \veci{r}_j = \delta_{ij}$ for $1\leq i, j\leq s$.
Then $\vecipower{\Erralpha}{T}\veci{r}_j$ is $O(1)$ or smaller; and $\vecipower{\ell}{T}_1\veci{h} = O(\dt\,^p)$. Here $\veci{\Erralpha}$ and $\veci{h}\/$ are defined in (\ref{eq:MatrixM}--\ref{eq:ErrorFunctions}).
\end{proposition}
\begin{proof}
To show $\vecipower{\Erralpha}{T}\veci{r}_j = O(1)$ or smaller, note that $z-1+\vecipower{b}{T}A^{-1}\vece=O(1)$ due to Assumption~(\ref{eq:AssumptionsLS}e). Second: to leading order in $\dt$, the $\veci{r}_j$ are the right eigenvectors of the rank-1 matrix $\vece\,\vecipower{b}{T}$, so that $\vecipower{b}{T}A^{-1}\veci{r}_j$ (and hence $\vecipower{\Erralpha}{T}\veci{r}_j$) is $O(1)$ or smaller. For $\vecipower{\ell}{T}_1\veci{h}$: Theorem~\ref{thm:Meigenvalues}(2) implies $\veci{\ell}_1$ has the form $\veci{\ell}_1 = \veci{b}+\sum_{j=1}^{p-1}\dt\,^j \veci{\beta}_j+ O(\dt\,^p)$, with a similar expansion for the eigenvalue. Substituting into $\vecipower{\ell}{T}_1 M = \lambda_1\vecipower{\ell}{T}_1$ and collecting powers of $\dt$ reveals that $\vecipower{\beta}{T}_j$, $1 \leq j \leq p-1$, is a linear combination of the vectors $\vecipower{b}{T}A^m$, $0 \leq m \leq j$. Now:
\begin{equation}\label{eq:ldoth}
	\vecipower{\ell}{T}_1\veci{h} = \frac{1}{z}\vecipower{\ell}{T}_1
	\sum_{k=2}^{p-1}\frac{(\imath \omega)^k \dt\,^k}{(k-1)!}\sov{k}
	\; U^*(x) + \frac{\ltel}{z(z-1)}\vecipower{\ell}{T}_1\vece
	+ O(\dt\,^p)\/.
\end{equation}
However, $\vecipower{\ell}{T}_1\sov{k} = \vecipower{b}{T}\sov{k} + \sum_{j=1}^{p-1}\dt\,^{j}\vecipower{\beta}{T}_j\sov{k} + O(\dt\,^p)$, where each $\vecipower{\beta}{T}_j\sov{k}$ term is a linear combination of $\vecipower{b}{T}A^m\sov{k}\/$, $0 \leq m \leq j$. Since Proposition~\ref{prop:StageOrderVectorOrthog} implies that $\vecipower{b}{T}  \sov{k} = 0$ for $k\leq p-1$ and $\vecipower{\beta}{T}_j\sov{k} = 0$ for $j\leq p-1-k$, it follows that
\begin{equation}\label{eq:ldottau}
  \vecipower{\ell}{T}_1\sov{k} = \vecipower{b}{T}\sov{k}+
    \sum_{j=p-k}^{p-1}\dt\,^{j}\vecipower{\beta}{T}_j\sov{k} + O(\dt\,^p)
    = O(\dt\,^{p-k})\/.
\end{equation}
Combining (\ref{eq:ldoth}--\ref{eq:ldottau}) with $\ltel = O(\dt\,^{p+1})$ and $\frac{\ltel}{z-1}=O(\dt\,^p)$ yields $\vecipower{\ell}{T}_1\veci{h} = O(\dt\,^p)$.
\end{proof}

We conclude the section by proving that the eigenvalues of $M$ lie in the RK stability region, which guarantees solutions to \eqref{eq:IRK_mat_eqn} are stable (small right-hand sides yield small solutions).
\begin{theorem}[Eigenvalues of $M$ are not in $W_e(\theta_2)$] \label{thm:MEigLocation}
Fix $|z| = 1, z \neq 1, \dt > 0$, and let $\lambda \neq 0$ be an eigenvalue of $M$; set $\zeta = \dt/\lambda\/$.  Then (at least) one of the two statements applies:
	\begin{enumerate}
		\item $1/\zeta$ is an eigenvalue of $A$; or
		\item $R(\zeta) = z$, where $R$ is the scheme's stability function.
	\end{enumerate}
    In particular:
	\begin{enumerate}
		\item[(3)] $\lambda\/$ is not in the interior of the wedge $W_e(\theta_2)\/$, introduced in (\ref{eq:AssumptionsLS}b).
		\item[(4)] If the scheme is A-stable, then $\text{Re}(\lambda) \geq 0$.
	\end{enumerate}
	If, conversely, $1/\zeta$ is an eigenvalue of $A$, with eigenvector in $\vec{b}^{\perp}$, then $\lambda = \dt/\zeta$ is an eigenvalue of $M$.
\end{theorem}
\begin{proof}	
We first justify items (1--2). Let $\vec{w} \neq 0$ be such that $M\vec{w} = \lambda\vec{w}$. Now if $\zeta^{-1}$ is an eigenvalue of A, then (1) holds and we are done. Assume, instead that $\zeta^{-1}$ is not an eigenvalue of A, and hence $(I-\zeta A)$ is invertible. Then $M\vec{w} = \lambda \vec{w}$ is equal to
\begin{align*}
\frac{\dt}{z-1} \vece\,\vecipower{b}{T}\vec{w} + \dt A\vec{w} = \lambda\vec{w},
\quad\quad \textrm{and thus} \quad\quad
\vec{w} = \frac{\zeta}{z-1}(\vecipower{b}{T}\vec{w}) (I-\zeta A)^{-1}\vece,
\end{align*}
where we have used $\lambda = \dt/\zeta$. The right equation shows $\vecipower{b}{T}\vec{w} \neq 0$ because $\vec{w} \neq 0$. Multiplying the right equation through by $((z-1)/(\vec{b}^T \vec{w}))\vec{b}^T$ and rearranging leads to the following identity:
\begin{equation*}
R(\zeta) = 1+\zeta\vecipower{b}{T}(I-\zeta A)^{-1}\vece = z.
\end{equation*}
Items (3) and (4) follow from (\ref{eq:AssumptionsLS}c) when (1) applies. On the other hand, when (2) applies, $\zeta\/$ must be on the boundary of the stability region, since $|z| = 1$. Then (3) follows from (\ref{eq:AssumptionsLS}b), and (4) from the definition of A-stability. Finally, the converse statement follows from the definition of $M$ in \eqref{eq:MatrixM}.
\end{proof}
\begin{corollary}\label{cor:invertibility}
The operator $I-M\diffop\/$ has an $L^\infty$-bounded inverse. Furthermore, $(I-M\diffop\/)^{-1}$ has an $L^{\infty}$ bound which is uniform for $\dt$ small enough.
\end{corollary}
\begin{proof}
By Assumption~(\ref{eq:AssumptionsLS}d), $M$ is diagonalizable. Let $M = T D T^{-1}$, where the columns of $T$ are eigenvectors of $M$, and $D$ is the diagonal matrix with the corresponding eigenvalues. Then $I-M\diffop = T (I-D\diffop) T^{-1}\/$.
Theorem~\ref{thm:MEigLocation}, and Assumption~(\ref{eq:AssumptionsLS}a) guarantee that $(1-\lambda\diffop)^{-1}$ is bounded in $L^{\infty}$, independent of $\dt$, simultaneously for all eigenvalues of $M$. Assumption~(\ref{eq:AssumptionsLS}d) implies that both $\|T(\dt)\|\/$ and $\|T^{-1}(\dt)\|$ are bounded, and remain bounded (uniformly) as $\dt \rightarrow 0$. Hence, $(I - M \diffop)^{-1}$ is bounded, and is uniform as $\dt \rightarrow 0$.
\end{proof}

Corollary~\ref{cor:invertibility} is used in \Srm\ref{ssec:eigenmodeanalysis} and \Srm\ref{ssec:WeakStageOrder_OrderResults} to estimate the magnitude of the errors (i.e., the amplitude of the numerical BLs) incurred by the scheme.

\subsection{Qualitative Behavior of the Global Error}
\label{ssec:eigenmodeanalysis}
We now use the spectral decomposition of the \matrixMtilde~$M$, derived in \Srm\ref{ssec:EigenValuesofM}, to analyze the behavior of equation \eqref{eq:IRK_mat_eqn} for $|z| = 1$, thus characterizing the spatial approximation error for numerical solutions that are periodic in time.

Let $\ef_i = \vecipower{\ell}{T}_i\veci{\errs}$ be the component of the error $\veci{\errs}$ in the eigenmode corresponding to the eigenvalue $\lambda_i$. Then, left-multiplying equation \eqref{eq:IRK_mat_eqn} by the left eigenvectors of $M$ (note that $\vecipower{\ell}{T}_i\diffop = \diffop\,\vecipower{\ell}{T}_i$, because $\diffop$ is linear), we obtain the following set of decoupled BVPs:
\begin{equation}\label{eq:eqn_eigenmodes}
 \ef_i - \lambda_i\diffop \ef_i = \vecipower{\ell}{T}_i\,\veci{h}
 \quad\mbox{with b.c.}\;\,\ef_i = 0\/,\quad\mbox{for}\quad 1\leq i\leq s\/.
\end{equation}
When $\lambda_i\/$ is small, \eqref{eq:eqn_eigenmodes} is a singular perturbation problem that can be analyzed with standard methods \cite{BenderOrszag1978}. We can thus conclude:
\begin{enumerate}[(I)]
\item \label{itm:IRK_ev_zero} \fbox{$\zeromode$}. The function $\zeromode = O(\dt\,^{q+1})$ or smaller, is comprised of $\ltel$ and $\veci{\lte}$. It has no singular behavior: spatial derivatives of $\zeromode$ are (generally) of the same order as $\zeromode$.  For stiffly accurate schemes, $\vecipower{b}{T}A^{-1} = (0,\dots,0,1)$ and $\lte_s = \ltel$ imply that $\zeromode \equiv 0$.
\item \label{itm:IRK_ev_1} \fbox{$\ef_1$; $\lambda_1 = O(1)\neq 0$.} As shown in \Srm\ref{ssec:EigenValuesofM}, the matrix $M$ has one $O(1)$ eigenvalue which is close to $(\imath \omega)^{-1}$. By Corollary~\ref{cor:invertibility}, the magnitude of the eigenmode $\ef_1$ is determined by $\vecipower{\ell}{T}_1\veci{h}$ in the BVP, which is $O(\dt\,^p)$ by Proposition~\ref{prop:1stLeftEigenVector}. Thus $\ef_1 = O(\dt\,^p)$. Further, $\ef_1$ has no singular behavior: the spatial derivatives of $\ef_1$ are of the same order as $\ef_1$ (provided $U^*(x)$ is smooth enough).
\item \label{itm:IRK_ev_small} \fbox{$\ef_j$; $\lambda_j$ is small but nonzero $(2 \leq j \leq s)$.} Then \eqref{eq:eqn_eigenmodes} is a singularly perturbed BVP, with $\lambda$ the small parameter. The solution $\ef$ generally has singular behavior, often in the form of boundary layers (BLs) (see Lemma~\ref{lem:HFO} regarding other possible effects). From Corollary~\ref{cor:invertibility}, the BL amplitude in $\ef$ is determined by the right hand side of the BVP,  $\vecipower{\ell}{T}\veci{h}$ --- thus, in general, $\ef_j = O(\dt\,^{q+1})$ ($2\leq j\leq s$) or smaller.  Spatial derivatives of $\ef_j$ will lose orders of accuracy, where the exact loss of accuracy depends on $\diffop$. For example, the heat equation will introduce BLs in $\ef_i$ that scale as $x/\sqrt{\dt}$, hence each derivative introduces a 1/2 order loss. The occurrence of singular behavior and BLs in the solutions of $\ef$ are unavoidable for generic time-dependent b.c.\ and forcing.
\end{enumerate}
Expanding $\veci{\errs}(x)$ in the eigenbasis of $M$:
\begin{equation*}
\veci{\errs}(x) =
\sum_{i = 1}^{s} \veci{r}_i \vecipower{\ell}{T}_i \veci{\errs}(x) =
\sum_{i = 1}^{s} \veci{r}_i \ef_i(x), \quad
\text{using that } \ef_i(x) = \vecipower{\ell}{T}_i \veci{\errs}(x)\/,
\end{equation*}
the RK error \eqref{eq:IRK_mat_eqn2} can then be expressed as:
\begin{equation}\label{eq:EigenModesLinComb}
	\errl(x) = \zeromode(x) \/ + \sum_{i=1}^s \vecipower{\Erralpha}{T}\veci{r}_i \, \ef_i(x) = \zeromode(x) +
	\sum_{i = 1}^{s}
	\frac{z\,\vecipower{b}{T}A^{-1}\veci{r}_i}{z-1+\vecipower{b}{T}A^{-1}\vece}
	\,\ef_i(x).
\end{equation}
Equation \eqref{eq:EigenModesLinComb} shows that the global error, $\errl(x)$, is composed of errors that are of the scheme's order ($\ef_1$), or of the scheme's stage order ($\zeromode$; $\ef_2$ through $\ef_{s}$); and that the error may have singular behavior.  Note that if Assumption~(\ref{eq:AssumptionsLS}e) is violated in that $\vecipower{b}{T}A^{-1}\vece = 0$, then the coefficients $\vecipower{\Erralpha}{T}\veci{r}_i$ of $\ef_i$ in \eqref{eq:EigenModesLinComb} scale like $O(\dt\,^{-1})$, resulting in an additional loss of convergence order. We conclude this section with several remarks.
\begin{remark}[Boundary mismatch for non-stiffly accurate schemes]
\label{rem:BoundaryMismatch}
Equation \eqref{eq:EigenModesLinComb} shows that the error $\errl$ evaluated at the domain boundary is: $\errl = \zeromode$. Stiffly accurate schemes guarantee that conventional b.c.\ (i.e., $g_i^{n+1} = g(t_n + c_i\,\dt)\/$) yield $\zeromode  = \errl = 0$ and hence exactly enforce the b.c.~$u^{n+1} = g^{n+1}$.  For non-stiffly accurate schemes, $\errl$ is in general non-vanishing yielding $u^{n+1} = g^{n+1} + O(\dt\,^{q+1})$.
\TheoremEnd
\end{remark}

\begin{remark}[Slowly decaying modes]
The calculation in this section is restricted to periodic in-time modes, which (see Appendix~\ref{app:periodicsol}) is sufficient to capture the order reduction phenomena, \emph{provided that the normal modes for both the equation and the scheme decay in time, and do so sufficiently fast as their space frequency grows.} Here we describe two situations where this condition is violated:

Schemes with growth factor such that $|R(\zeta)|\to 1$ as $\zeta\to -\infty$. Then, the numerical solution may contain transient artifacts. For example, in the heat equation those artifacts resemble BLs, but they thin out in width slowly over time (and thus can compromise the observed order for the solution and its derivatives). The artifacts can be triggered by BLs produced in the initial step, via the mechanism outlined in \Srm\ref{ssec:orderloss_lte}. The introduced high frequency modes then die arbitrarily slowly---slower the higher the frequency, which is why the artifacts tend to become narrower as time grows. An important RK scheme exhibiting this behavior is the implicit mid-point rule, defined by the Butcher tableau $A=\vecc=[1/2]$ with $\veci{b}=[1]$. Because it has only one stage, the matrix $M$ has no small eigenvalues, and the scheme has no time-periodic numerical modes with BLs. The implicit mid-point rule is the simplest case of a Gauss method, which achieve order $2s$ with $s$ stages. These methods have $R(\zeta)\rightarrow(-1)^s$ as $\zeta\rightarrow-\infty$, thus they are all examples for the issue described here. In addition, for $s\geq 2$, they also exhibit OR in the time-periodic sense, due to existence of small eigenvalues.

A second example of ``slowly decaying modes'' occurs when the operator $\diffop$ in \eqref{eq:IBVP} is purely dispersive. In this case the normal modes for the equation itself are time-periodic, with no decay. An accurate numerical scheme will approximate this behavior, with normal modes that decay very slowly---at least as long as their frequencies are not too high. Just as for the schemes where $|R(\zeta)|\to 1$ as $\zeta\to -\infty$, this can lead to long-lived transients in the numerical solution (also triggered by BL effects) which compromise the observed order for the solution. An example of this situation in provided in \Srm\ref{ssec:NumExamples_Schroedinger}.
\TheoremEnd
\end{remark}

\begin{remark}[Jordan blocks] \label{rem:JordanBlock}
The eigen-equation \eqref{eq:eqn_eigenmodes}, error expansion \eqref{eq:EigenModesLinComb}, and general discussion in this section, are formulated for matrices $M$ that are diagonalizable. However, they can be modified to include the general case in which $M$ has Jordan blocks. To see this, assume that $\vecipower{\ell}{T}_{i,j+1} M = \lambda_i\vecipower{\ell}{T}_{i,j+1} + \vecipower{\ell}{T}_{i,j}$ for $1\leq j\leq J_i-1$, where $\veci{\ell}_{i,1}$ is the eigenvector associated with $\lambda_i$ and $J_i$ denotes the size of the Jordan block corresponding to $\lambda_i$. Then equation \eqref{eq:eqn_eigenmodes} is modified to
\begin{equation}
\label{eq:eqn_jordan_block}
\ef_{i,j+1} - \lambda_i\,\diffop\,\ef_{i,j+1} =
\vecipower{\ell}{T}_{i,j+1}\,\veci{h} +
\diffop\,\ef_{i,j}\/, \;\;\text{with b.c.}\;\; \ef_{i,j+1} = 0\/,
\end{equation}
where $\ef_{i,j} = \vecipower{\ell}{T}_{i,j}\veci{\errs}$. Note that the occurrence of $\diffop$ in the right hand side of \eqref{eq:eqn_jordan_block} provides a BL feedback mechanism through the derivatives of the BL in the prior generalized eigenfunction. This can potentially trigger worse OR effects than in the diagonal case.
\TheoremEnd
\end{remark}

\subsection{Asymptotic Analysis of the Boundary Layers}
\label{ssec:orderresults}
In this subsection, we conduct an asymptotic analysis ($\dt \ll 1$) of the singular functions $\ef_i(x)$ and the global RK numerical error \eqref{eq:EigenModesLinComb}, i.e., $\errl(x)$. The modes $\ef_i(x)\/$ solve the BVP
\begin{equation}\label{eq:singulareigenmodes}
  \ef_i - \lambda_i \diffop \, \ef_i = \ampk{i} \/ U^*(x) \/,
  \quad\text{with b.c.}\quad \ef_i = 0\/,
\end{equation}
where we have introduced $\ampk{i}$ (independent of $x$) by writing $\vecipower{\ell}{T}_i \veci{h}\/ = \ampk{i}\, U^*(x)$ using $\veci{h}$ in \eqref{eq:ErrorFunctions}. Further, by Proposition~\ref{prop:1stLeftEigenVector}: $\ampk{1} = O(\dt^{p})$, while $\ampk{i} = O(\dt^{q+1})$ or smaller for $2\leq i \leq s$.

As mentioned in \Srm\ref{ssec:eigenmodeanalysis}, the solution to \eqref{eq:singulareigenmodes}, for $i = 1$, is non-singular since $\lambda_1 = O(1)$ (and as shown below, will not contribute to order reduction). For $2\leq i \leq s$, $\lambda_j = O(\dt)$, and as we show here, the solution to \eqref{eq:singulareigenmodes} can be described using standard matched asymptotic expansions \cite{BenderOrszag1978, MHHolmes1995, JKHunter2004, kevorkian1996multiple}. Below, we work out the analysis for $\diffop$ a second order operator in one space dimension. This restriction allows us to showcase how the spatial structures that arise due to OR can be constructed in a concrete fashion. Note that in higher dimensions, such concrete constructions are not as easily done (for example, when corner layers arise); however, the general nature of OR, namely its manifestations via singularly perturbed problems and its resulting asymptotic structures, persists in any dimension.

When $U^*(x)$ is smooth, $\ef_i\/$ has two BLs (one at each boundary) of thickness $O(\sqrt{\dt})$.  Away from the BLs, the solution is described by an ``outer'' expansion that will not contribute to OR. Inside the BLs, the solution is described by the ``inner'' expansion, and together the inner and outer expansions generate a ``composite'' expansion valid on the whole domain.

\subsubsection{Outer expansion, $2\leq i\leq s$.}
Valid away from the boundaries (i.e., $\sqrt{\dt} \ll x$ and $x \ll 1-\sqrt{\dt}$) is a ``regular'' expansion based on Taylor expanding the solution $\ef_i$ in powers of the small parameter $\lambda_i$ up to order $m$. Namely, $\ef_i \sim \rsol_i$ where $\rsol_i$ is the truncated (Neumann) series for $(I - \lambda_i \diffop)^{-1}\ampk{i}U^*$:
\begin{align}\label{eq:reg_sol}
		\rsol_i(x)  = \ampk{i}\left(
			U^* + \lambda_i \diffop U^* + \lambda_i^2 \diffop^2 U^* + \ldots + \lambda_i^{\mpow} \diffop^{\mpow} U^*
		\right), \quad\quad \textrm{for } 2 \leq i \leq s.
\end{align}
Equation \eqref{eq:reg_sol} is an $m-$th order expansion in powers of $\lambda_i = \mu_i \dt + o(\dt)$; how large one can take $m\/$ depends on how many derivatives $U^*\/$ has. Note that the focus here is on $m$ fixed and $\dt\to 0$.  No statement is made about $\dt$ fixed and $m\to\infty$ in \eqref{eq:reg_sol}.

In the following, we discuss, first in a special case and then in the general setting, the situation where $U^*(x)$ is smooth and OR occurs due to BLs in $\ef_i\/$, as well as one situation where $U^*(x)$ is not smooth and generates an internal \emph{interface} layer in $\ef_i\/$ that leads to OR inside the domain.  We will make use of the following:
\begin{enumerate}
	\item[a.] Rescaled spatial variables: $X := \frac{x}{\sqrt{\dt}}$, $Y := \frac{1-x}{\sqrt{\dt}}$ and $Z := \frac{x-1/2}{\sqrt{\dt}}$.
	\item[b.] Exponential function: $S(x) := e^{-x}$.
	\item[c.] Eigenvalues $\lambda_i=\mu_i\/\dt + o(\dt)\/$, $2 \le i \leq s\/$: From Theorem~\ref{thm:Meigenvalues}, $\mu_i \neq 0\/$. Theorem~\ref{thm:MEigLocation}(3) implies $\mu_i \notin W(\theta_2)$, so that $\mu_i$ is not a negative real number. Hence,\footnote{Using the principle branch of the square root.} $\rm{Re}(\sqrt{\mu_i}) > 0$ so $S(x/\sqrt{\lambda_i}) \approx e^{-x/\sqrt{\dt \mu_i}}$ is exponentially (in $x$) and rapidly decaying (in $\dt$).
\end{enumerate}

\subsubsection{Composite solution when $\diffop = \partial_x^2$, $U^*(x)$ is smooth.}
The composite solution has the form $\ef_i(x) \sim \rsol_i(x) + \EF_{L, i}(X) + \EF_{R,i}(Y)$, with $\EF_{L,i}(x)$ localized near $x = 0$. It arises from constructing an inner solution consisting of a BL function $\EF_{L,i}(X)$ that connects the b.c.\ $\ef_i(0) = 0$ to the outer solution $\ef_i(x) \sim \rsol_i(x)$ when $\sqrt{\dt} \ll x \ll 1 - \sqrt{\dt}$.  After rescaling space, the left BL (the one near $x=1$ is analogous) function $\EF_{L,i}(X)$ solves the ODE $\EF_{L,i} - \mu_i  \EF_{L,i}'' = 0$, with b.c.\ $\EF_{L,i}(0) = -\rsol_i(0)$ and $\EF_{L,i}(+\infty) = 0$. The general solution of this ODE is a superposition of the exponentials $S(\pm X/\sqrt{\mu_i})$, which after matching b.c.\ and using property c. (i.e., only the $+X$ exponential contributes) yields $\EF_{L,i}(X) = -\rsol_i(0)\/S(X/\sqrt{\mu_i})$. By a similar argument, the right BL has the form $\EF_{R,i}(Y) =  -\rsol_i(1)\/S(Y/\sqrt{\mu_i})$, and together the $m$-th order composite expansion valid in the whole interval is:
\begin{equation}\label{eqn:compositeExp1}
\ef_i \sim \rsol_i(x) -\rsol_i(0)\,e^{-\frac{X}{\sqrt{\mu_i}}}
-\rsol_i(1)\,e^{-\frac{Y}{\sqrt{\mu_i}}}\/, \quad\quad \textrm{for } 2 \leq i \leq s.
\end{equation}

\subsubsection{Composite solution when $\diffop = \alpha_2(x) \partial_x^2 + \alpha_1(x) \partial_x + \alpha_0(x)$, $U^*(x)$ is smooth.}
Here $\alpha_2(x) > 0$ is positive. The asymptotics of the variable coefficient $\diffop$ are only a minor modification of the case $\diffop = \partial_x^2$.  After rescaling space near the left BL:
\begin{equation*}
\lambda_i \diffop = \alpha_2(X \dt) \frac{\lambda_i}{\dt} \partial_{X}^2 + \alpha_1(X \dt) \frac{\lambda_i}{\sqrt{\dt}} \partial_X + \alpha_0(X \dt) = \mu_i \alpha_2(0) \, \partial_{X}^2 + o(\dt),
\end{equation*}
so that at leading order in $\dt$, $\EF_{L,i}(X)$ solves $\EF_{L,i} - \mu_i \alpha_2(0) \EF_{L,i}'' = 0$, with b.c.\ $\EF_{L,i}(0) = -\rsol_i(0)$ and $\EF_{L,i}(+\infty) = 0$. The solution \eqref{eqn:compositeExp1} is modified to:
\begin{equation}\label{eqn:compositeExp2}
\ef_i \sim \rsol_i(x) -\rsol_i(0)\,e^{-\frac{X}{\sqrt{\mu_i \, \alpha_2(0)}}}
-\rsol_i(1)\,e^{-\frac{Y}{\sqrt{\mu_i \,\alpha_2(1)}}}\/, \quad\quad \textrm{for } 2 \leq i \leq s.
\end{equation}

\subsubsection{Composite solution when $\diffop = \alpha_2(x) \partial_x^2 + \alpha_1(x) \partial_x + \alpha_0(x)$, $U^*(x)$ is piecewise smooth.}
\label{subsubsec:non_smooth_Ustar}
Here $\alpha_2(x) > 0$ is positive.  Assume $U^*(x)$ is $C^\infty$ on $[0, \frac{1}{2})$ and also on $(\frac{1}{2}, 1]$, and both $U^*(x)$ and all its derivatives have finite one-sided limits at $\frac{1}{2}$ (where $x=\frac{1}{2}$ is chosen without loss of generality). In addition, there is some $1 \leq \powfail < m$ where $\diffop^{\powfail} U^*$ does not exist at $x = \frac{1}{2}$.  The expression \eqref{eqn:compositeExp2} fails near $x = \frac{1}{2}$ since $\rsol_i(\frac{1}{2})$ does not exist, however \eqref{eqn:compositeExp2} still remains valid on $0 \leq x \ll \frac{1}{2} - \sqrt{\dt}$ and separately on $\frac{1}{2} + \sqrt{\dt} \ll x \leq 1$. In the vicinity of $x = \frac{1}{2}$, the function $\ef_i$ has an internal layer $\EF_{I,i}(Z)$ that connects $\ef_i(x)$ to the two sides of the outer solution $\ef_i(x) \sim \rsol_i(x)$ as $|x - \frac{1}{2}| \gg \dt$. After rescaling space, $\EF_{I,i}(Z)$ solves $\EF_{I,i} - \mu_i \, \alpha_2(\frac{1}{2}) \EF_{I,i}'' = 0$ with b.c.\ $\EF_{I,i}(\pm \infty) = 0$ and interface conditions that enforce continuity of $\ef_i$ and $\ef_i'$ across $x = \frac{1}{2}$.
Let $\Phi_{\pm} := \lim_{\tau \rightarrow 0^+} \rsol_i(\frac{1}{2} \pm \tau)$ and
$\Phi_{\pm}' := \lim_{\tau \rightarrow 0^+} \left(\frac{d\rsol_i }{dx}(\frac{1}{2} \pm \tau) \right)$ and introduce the jumps across $x = \frac{1}{2}$ as
$\left[ \Phi \right] := \Phi_+ - \Phi_-$ and
$\left[ \Phi' \right] := \Phi_{+}' - \Phi_{-}'$. We have
\begin{equation}\label{eqn:intsol}
	\EF_{I,i}(Z) =
	\left(-\frac{\mathrm{sgn}(Z)}{2} \left[ \Phi \right] +
	\frac{\sqrt{\dt \, \mu_i \, \alpha_2(\frac{1}{2})}}{2}\left[ \Phi' \right] \right) e^{-\frac{|Z|}{\sqrt{\mu_i \, \alpha_2(\frac{1}{2})}}},
\end{equation}
where $\mathrm{sgn}(Z)$ is the sign of $Z$. The composite solution \eqref{eqn:compositeExp1}, is modified to ($\EF_{L,i}$, $\EF_{R,i}$ are the same as \eqref{eqn:compositeExp1})
\begin{equation}\label{eqn:compositeExp3}
\ef_i \sim \left\{
	\begin{array}{lc}
     	\rsol_i(x) + \EF_{L,i}(X) + \EF_{R, i}(Y) + \EF_{I,i}(Z),
		& \textrm{for } x \neq \frac{1}{2}, \\
		\frac{1}{2}\left(\Phi_{+} + \Phi_{-} \right)
		+ \frac{\sqrt{\dt \, \mu_i \, \alpha_2(\frac{1}{2})}}{2}\left[ \Phi' \right]
		& \textrm{for } x = \frac{1}{2}
   \end{array}
		\right.
		\quad \textrm{for} \quad
	2 \leq i \leq s.
\end{equation}

\subsubsection{Shape of the RK error inside a boundary layer.}
We turn our attention to the global error $\errl(x)$. When the functions $\ef_i$ are well approximated by Taylor expansions in powers of $\dt$ (about $\dt = 0$) for $\dt \ll 1$, the RK scheme is designed so that both the modes $\ef_i$, and coefficients $\vecipower{\Erralpha}{T}\veci{r}_i $, may be expanded via Taylor series in \eqref{eq:EigenModesLinComb}, and cancel out to order $O(\dt^p)$. The regular solution $\rsol_i$ is exactly the part of $\ef_i$ that can be expanded via Taylor series (the boundary or interface layers, i.e., $\EF_{L,i}(X)$, can not). Note that, while we present the analysis here for the special case of a singularity as in \Srm\ref{subsubsec:non_smooth_Ustar}, the formulation \eqref{eqn:compositeExp2}, and thus also \eqref{eqn:compositeExp3}, is the most general form if $U^*(x)$ smooth. Substituting the form \eqref{eqn:compositeExp3} into \eqref{eq:EigenModesLinComb}, we group the terms as follows:
\begin{equation}\label{eq:groupterms}
	\errl(x) =
	\underbrace{\left[\vecipower{\Erralpha}{T}\veci{r}_1 \, \ef_1(x) +
	\zeromode(x) +
	\sum_{i=2}^{s}\vecipower{\Erralpha}{T}\veci{r}_i \, \rsol_i(x)
	\right]}_{\textrm{Bracket } 1 = O(\dt^p)} \!+\!
	\underbrace{\left[\sum_{i=2}^{s} \vecipower{\Erralpha}{T}\veci{r}_i \,
	\big( \EF_{L,i}(X) \!+\! \EF_{R, i}(Y) \!+\! \EF_{I, i}(Z)		
	\big) \right]}_{\textrm{Bracket } 2}
\end{equation}
Note that the terms in Bracket~1 of \eqref{eq:groupterms} are individually $O(\dt^{q+1})$ (or smaller). However, all those terms can be Taylor-expanded in powers of $\dt$, and by consistency of the RK scheme, they sum together to $O(\dt^p)$ (or smaller); see Appendix~\ref{app:proof} for a formal proof.
Bracket~2 has terms $\EF_{\beta, i}$, $\beta = \{L,R,I\}$ that are potentially $O(\dt^{q+1})$, however do not have Taylor expansions in $\dt$ and generally do not cancel to high order.
Note that the magnitude of $\EF_{L,i}$, $\EF_{R,i}$ always occurs at the boundary, which is in general $\rsol_i =O(\dt^{q+1})$, unless the leading order terms in the regular expansion $\rsol_i$, i.e., $U^*$, $\diffop U^*$, vanish on the boundary. Similarly, if $U^*(x)$ has a singularity at $x = \frac{1}{2}$, the magnitude of $\EF_{I,i}$ is determined by the jumps in the regular solution $\left[ \Phi \right]$ and $\sqrt{\dt}\left[ \Phi' \right]$ at $x = \frac{1}{2}$ and is determined by the largest value of $\powfail$ for which $\diffop^{\powfail}U^*(\frac{1}{2})$ exists. The fact that a loss of spatial regularity in $U^*$ can result in order reduction has been discussed in \cite{OstermannRoche1992}. However, the results in \Srm\ref{subsubsec:non_smooth_Ustar} and \eqref{eq:groupterms} characterize the precise asymptotic shape of the error in the vicinity of a singularity in $U^*$.

We now examine \eqref{eq:groupterms} in the vicinity of the left boundary $x =0$ (the right boundary, or near a point $x$ where $U^*(x)$ is not smooth is similar).  Taylor-expanding $\vecipower{\Erralpha}{T}\veci{r}_i$ in $\dt$ and using the fact that Bracket~1 is $O(\dt^p)$, the error near the left boundary is:
\begin{equation}\label{eqn:ErrorInBL}
\errl(x) = \dt^{q+1}\, \sum_{i=2}^s
P_i(\dt) \,
S\left(\frac{x}{\sqrt{\dt \/ \mu_i \, \alpha_2(0)}}\right)
+ O(\dt^p)\/, \quad\quad \textrm{ for } 0 \leq x \ll 1,
\end{equation}
where $P_i(\dt)$ ($2 \leq i \leq s$) are polynomials of degree $p - q - 2$.
Equation \eqref{eqn:ErrorInBL} reveals the structure in the BLs, and explains why \emph{(generically) RK schemes incur order reduction in BLs.} Specifically: the functions $S(\tfrac{x}{\sqrt{\dt\/\mu_i\,\alpha_2(0)}})$ are singular in $\dt$, and are linearly independent when the $\mu_i\/$ are distinct. Hence, Taylor-based cancellations in the summation \eqref{eqn:ErrorInBL} (on which the scheme relies to achieve its order) do not occur. The convergence order in $\errl(x)$ is controlled by the coefficients $P_i(\dt)$ of the $O(1)$ functions $S(\tfrac{x}{\sqrt{\dt\/\mu_i \, \alpha_2(0)}})$.
An alternative viewpoint when the $\mu_i$ in \eqref{eqn:ErrorInBL} are distinct, is that the functions $\EF_{L,i}$ have different widths of $O(\sqrt{\dt})$, see~\Srm\ref{ssec:orderresults}. Hence they generally do not cancel through a linear combination and result in a ``composite'' BL in $\errl(x)$. Note that: if $\sqrt{\mu_i}$ is not a real number, the BL includes high frequency oscillations triggered by Im$(\sqrt{\mu_i})$. How visible these oscillations are depends on the ratio Im$(\sqrt{\mu_i})/\mbox{Re}(\sqrt{\mu_i})$. The larger this ratio, the larger the role of the oscillations.

Finally, away from the BLs ($\dt \ll x \ll 1-\dt$) or any point where $U^*$ is singular, the functions $\hsol_{\beta, i}$ for $\beta = \{L,R,I\}$ are exponentially small (so that Bracket 2 in \eqref{eq:groupterms} is exponentially small); hence $\ef_i \sim \rsol_i$ is just the regular solution and consequently $\errl(x)$ does not suffer from order reduction.

Equation \eqref{eqn:ErrorInBL} also highlights a crucial structural property of the approximation error of RK schemes for IBVPs. Aside from a few special cases (e.g., backward Euler), the singular functions
$S(\tfrac{x}{\sqrt{\dt\/\mu_i \, \alpha_2(0) }})$ can generally not be avoided. Instead, methods that overcome order reduction (cf.~\Srm\ref{sec:WeakStageOrder} and~\Srm\ref{sec:MBC}) render the singular functions in $\errl(x)$ to be of the scheme's formal order $p$, or higher. While this remedies OR in the solution $u$, the persistence of BLs (or internal layers) implies that (sufficiently high) spatial derivatives of the solution generally still incur OR. The fact that derivatives generally are less accurate follows because the BL or internal layer functions $\EF(x)$ (in unscaled variables) always satisfy the homogeneous equation $\diffop \EF(x) \propto \dt^{-1} \EF(x)$, which shows that $\diffop \EF(x)$ is one order less that $\EF(x)$ (see also, \cite{Alonso2002}).

\subsubsection{Beyond second order operators $\diffop$.}
The asymptotic analysis leading to \eqref{eq:groupterms} reveals that the BLs' error contributions amplify by $\dt^{-1/\ell}$ per spatial derivative, if $\diffop$ is an $\ell$-th order differential operator. These considerations are of particular importance for any practical problem in which gradients of the solution at/near the boundary are needed.

\begin{remark}[High order equations and composite BLs]\label{rem:HOEcBL}
For general $\diffop$, a version of \eqref{eqn:ErrorInBL} holds and the structure of the RK error can be obtained via asymptotic analysis; however the BLs in \eqref{eqn:ErrorInBL} are determined by the highest order derivative in $\diffop$. For example, let $\diffop = \alpha_{\ell}(x)\,\partial_x^{\ell} + \alpha_{\ell-1}(x)\,\partial_x^{\ell-1} + \ldots + \alpha_0(x)$, where $\alpha_{\ell}(x)$ is a nonvanishing function. Then the modes $\EF_{R,i}(x)$ and $\EF_{L,i}(x)$ contain a superposition of exponentials $\exp(x/(\dt^{1/\ell}\rho_{i\/j}))$, where $\rho_{i\/j}$ ($1 \leq j \leq \ell$) are the $\ell$ roots of $\rho_{i\/j}^\ell = \alpha_\ell(0)\,\mu_i$. Values of $\rho_{i\/j}$ with negative (resp.\ positive) real part correspond to an exponentially (in $x$) decaying (resp.\ growing) function, and contribute to a BL near $x = 0$ (resp.~$x = 1$). Values of $\rho_{i\/j}$ on the imaginary axis correspond to purely oscillatory functions.
\TheoremEnd
\end{remark}

Remark~\ref{rem:HOEcBL} highlights that in principle the roots $\rho_{i\/j}$ of the singular equation \eqref{eq:eqn_eigenmodes} could be purely imaginary, leading to modes $\ef_{i}\sim \exp(x/(\dt^{1/\ell}\rho_{i\/j}))$ that have \emph{high frequency oscillations} (HFO) extending over the whole domain. (This is the same type of behavior that arises in WKB theory~\cite{BenderOrszag1978}). For constant coefficient PDEs, however, dissipation (i.e., the spectrum of $\diffop$ is contained within the wedge in (\ref{eq:AssumptionsLS}a)) eliminates the possibility of HFO:
\begin{lemma}[High frequency oscillations] \label{lem:HFO}
Under the assumptions in \eqref{eq:AssumptionsLS}, HFO cannot occur for constant coefficient differential operators $\diffop\/$.
\end{lemma}
\begin{proof}
Suppose that the constant coefficient $\diffop = \alpha\,\partial_x^{\ell} +$lower order terms, where $\ell\in\mathbb{N}$ and $\alpha \neq 0$, were to produce HFO in the numerical error. That means, in the limit as $\dt \rightarrow 0$, there is at least one mode $\ef_i \sim \exp(x/(\dt^{1/\ell} \imath r))$ that solves the ODE $(1 - \dt \mu_i \diffop) \psi_i = 0$, where $r$ is purely real. Substituting $\diffop$ and the exponential ansatz for $\ef_i$ into the ODE, yields the relationship $\alpha = (\imath r)^\ell/\mu_i$. This allows us to write the eigenvalues of $\diffop$ (found by substituting $e^{\imath k x}$ into $\diffop u = \lambda u$) as
\begin{equation} \label{eq:HFO_eigenvalues_of_L}
\lambda = (r k)^\ell/\mu_i +\text{(lower order terms in $k$)},
\end{equation}
where $k$ is real. In general only certain values of $k$ are allowed, but those (infinitely many) include arbitrarily large $k$. We now show that \eqref{eq:HFO_eigenvalues_of_L} violates the hypothesis \eqref{eq:AssumptionsLS} when $k \rightarrow \infty$. Specifically:
\begin{enumerate}[(i)]
\item By Theorem~\ref{thm:MEigLocation}(3): $|\textrm{arg}(1/\mu_i)| \leq \pi - \theta_2$ ($\theta_2$ defines the wedge in (\ref{eq:AssumptionsLS}b)).
\item Assumption~(\ref{eq:AssumptionsLS}a) asserts that the eigenvalues $\lambda$ of $\diffop$ satisfy $|\textrm{arg}(\lambda)| > \pi - \theta_1$.
\end{enumerate}
In the limit as $k \rightarrow \infty$, the eigenvalues \eqref{eq:HFO_eigenvalues_of_L} grow with a slope that approaches $1/\mu_i$ (i.e., $\mathrm{Im}(\lambda) / \mathrm{Re}(\lambda) \rightarrow \mathrm{Im}(\mu_i^{-1}) / \mathrm{Re}(\mu_i^{-1})$ as $k \rightarrow \infty$). Since $1/\mu_i$ has a slope with angle $\leq \pi - \theta_2$, the eigenvalues must cross every line with a larger slope angle, including the line with angle $\pi - \theta_1$. However, this violates item (ii).
\end{proof}

The assumptions in \eqref{eq:AssumptionsLS} were introduced to (in particular) guarantee that the scheme equations can be solved at each step, and avoid numerical instabilities. But they are in no way necessary (particularly, they exclude non-dissipative systems to simplify the analysis), and less constraining assumptions are possible. Thus, in the general situation we cannot rule out HFO---even though we have not observed them in actual examples. That being said, even if HFO situations are not possible, BLs that include oscillations do occur. Further, because the BL thickness scales like $\dt^{1/\ell}$, unless $\dt$ is very small, these BLs with oscillations can be quite thick, for example see Figure~\ref{fig:errshape_schro_dirk3} in \Srm\ref{ssec:NumExamples_Schroedinger}.
\begin{remark}[Multiple dimensions] \label{rem:MultiD}
Clearly, OR occurs in higher dimensions as well. Along smooth parts of the domain boundary, BLs should arise, similarly to the 1D BLs studied here. In addition, at corners, a breakdown in solution smoothness can also effect OR; see \cite{OstermannRoche1993} for error estimates.
\TheoremEnd
\end{remark}

\begin{figure}
	\hfillL
	\includegraphics[width=0.47\textwidth]{%
		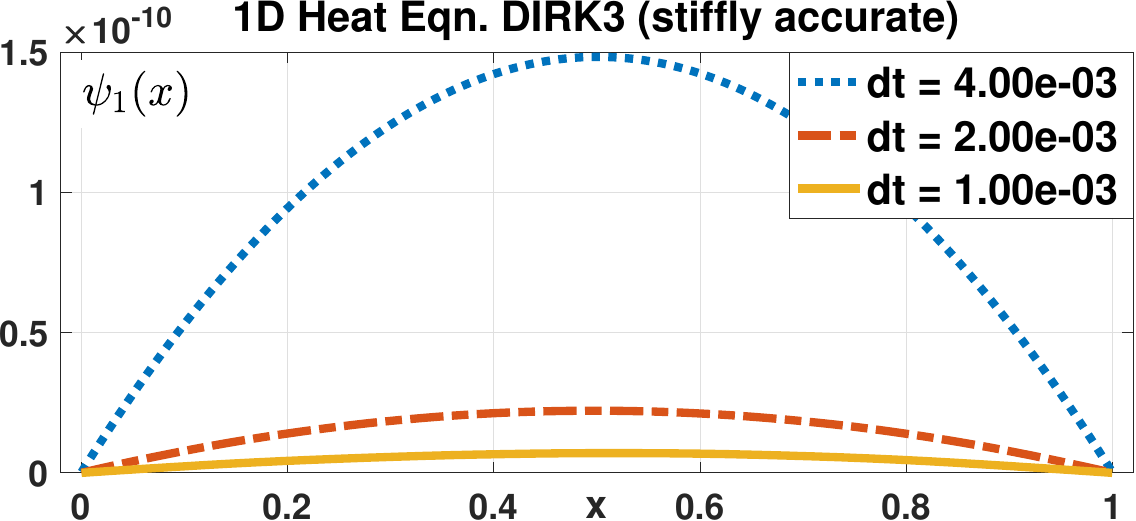}
	\hfill
	\includegraphics[width=0.47\textwidth]{%
		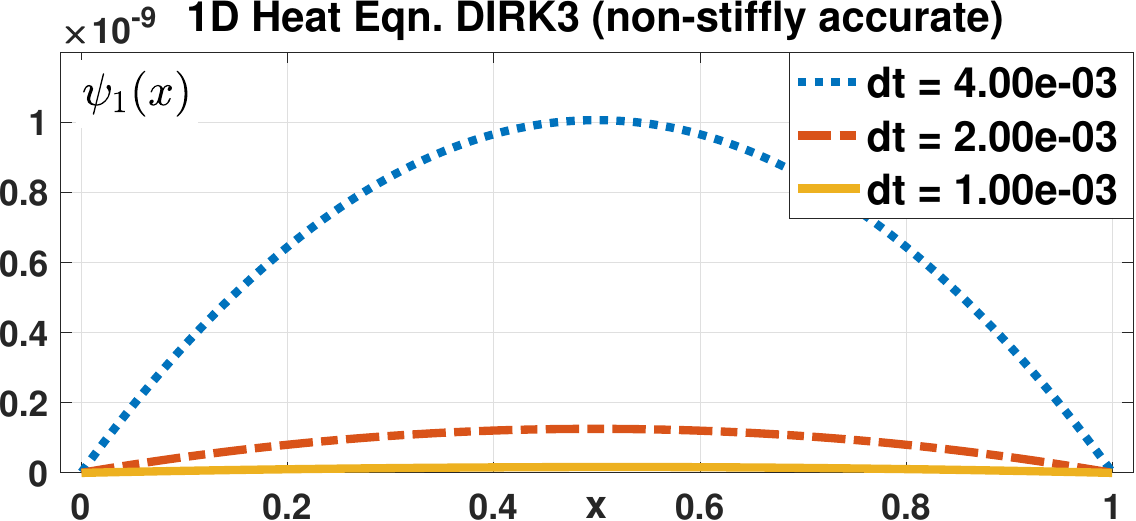}
	\hfillR	
	\\
	\hfillL
	\includegraphics[width=0.47\textwidth]{%
		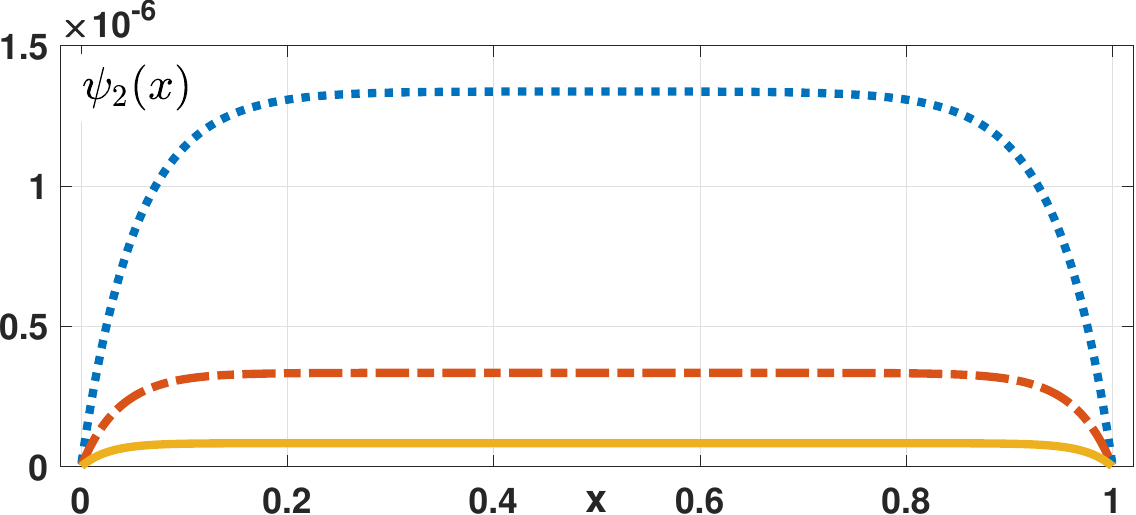}
	\hfill
	\includegraphics[width=0.47\textwidth]{%
		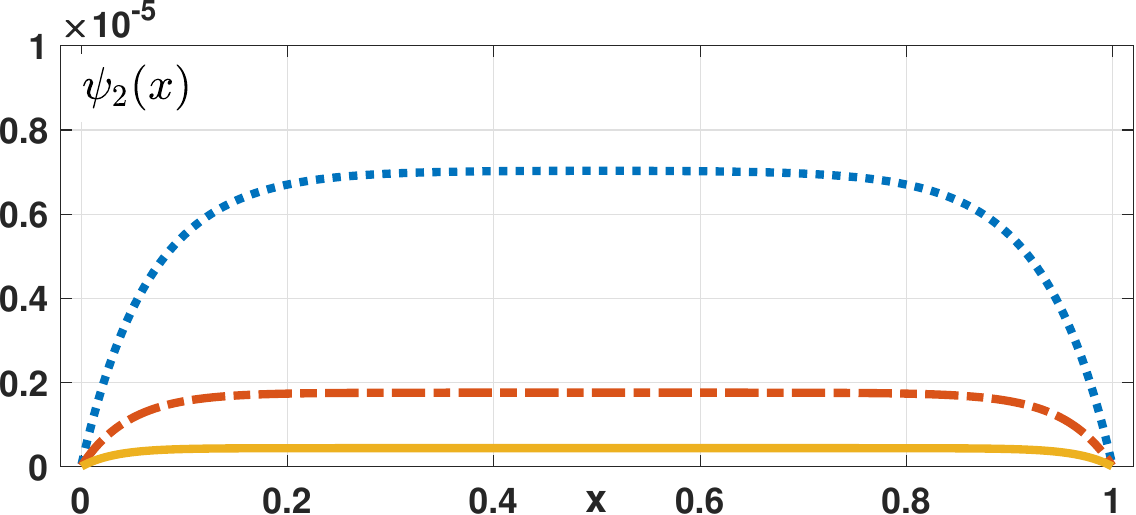}
	\hfillR	\\
	\hfillL	
	\includegraphics[width=0.47\textwidth]{%
		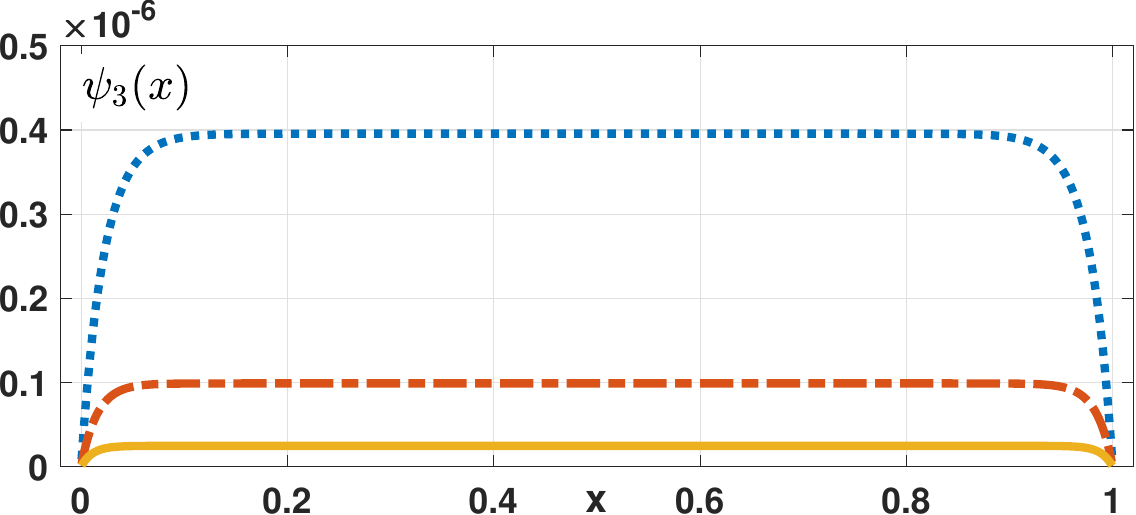}
	\hfill
	\includegraphics[width=0.47\textwidth]{%
		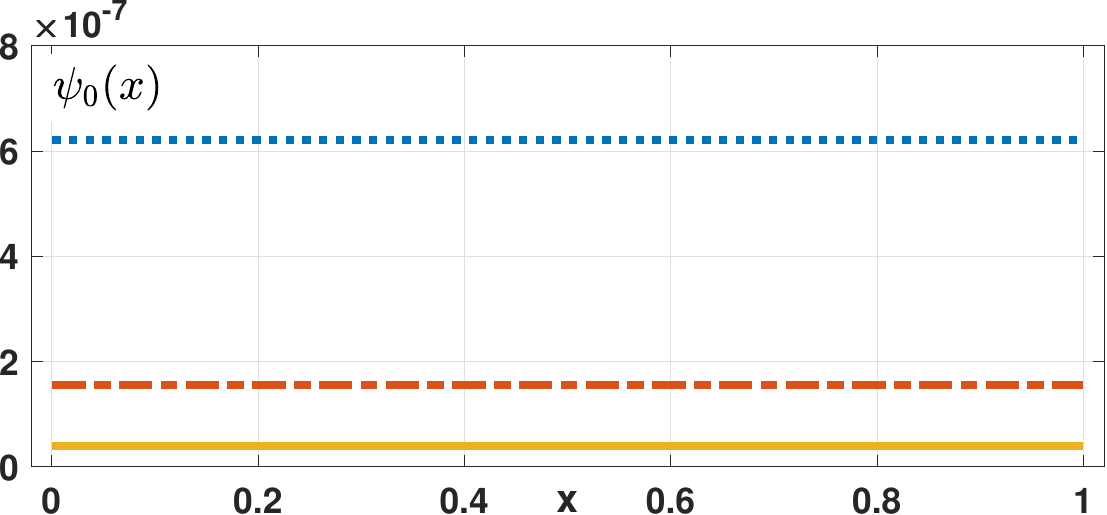}
	\hfillR	\\
	\hfillL	
	\includegraphics[width=0.47\textwidth]{%
		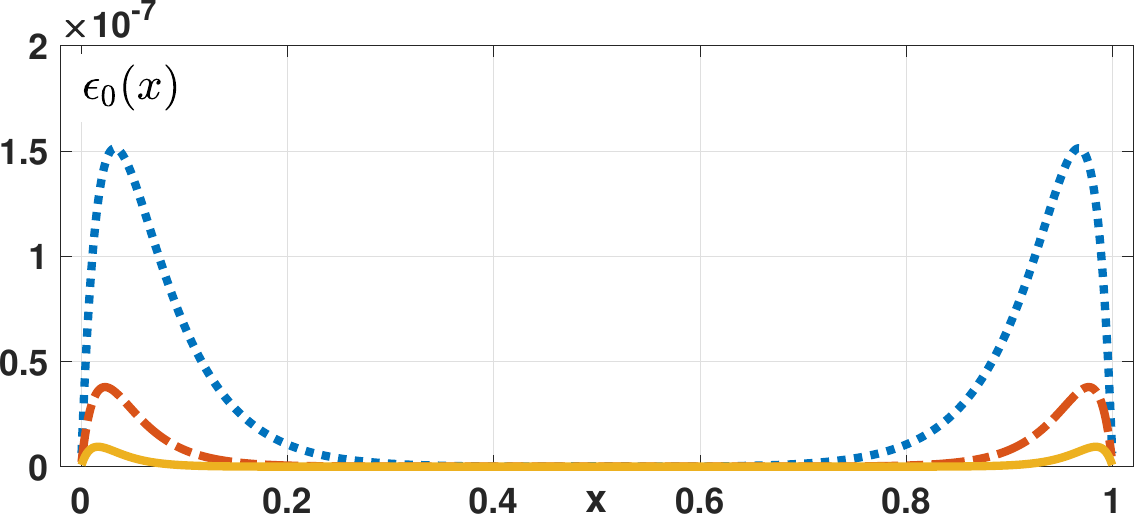}
	\hfill
	\includegraphics[width=0.47\textwidth]{%
		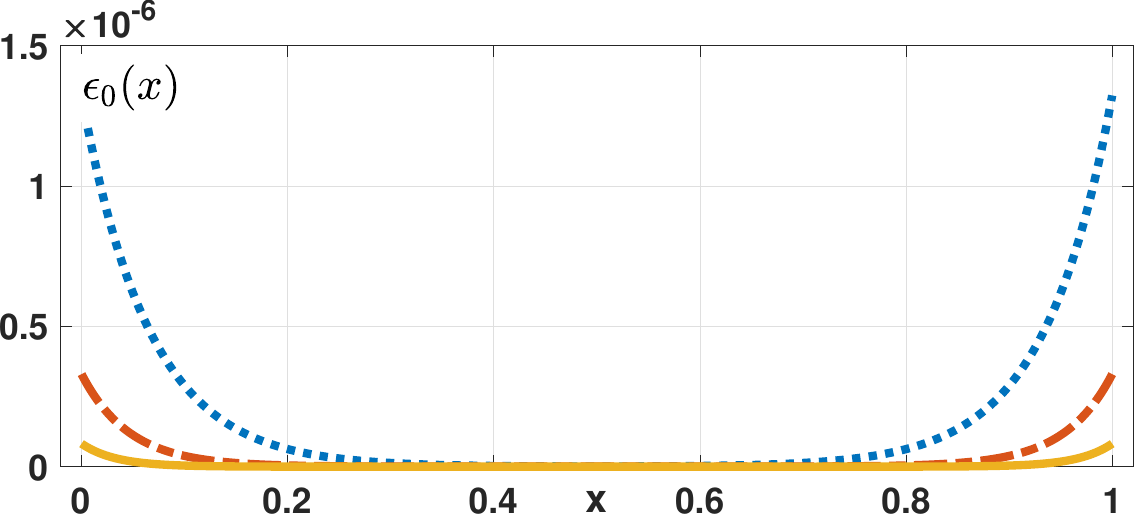}
	\hfillR	
    \caption{Eigenmodes $\ef_i(x)$ (real part) and error $\errl(x)$ for a 3-stage, 3rd order, stiffly accurate DIRK (left), and a 2-stage, 3rd order, non-stiffly accurate DIRK (right). Three $\dt$ choices used.}
    \label{fig:ShapeEigenModes}
\end{figure}

\subsubsection{Numerical example.}
We illustrate the modal analysis with the same 1D heat equation example used in \Srm\ref{ssec:example_order_loss}: a time-varying, constant-in-space, exact solution, approximated with two DIRK schemes: a 3-stage stiffly accurate DIRK3 and a 2-stage non-stiffly accurate DIRK3. The errors are computed with the eigenmodes, obtained by solving \eqref{eq:eqn_eigenmodes} with standard second order finite differences on a fixed uniform grid with 2000 points.

Figure~\ref{fig:ShapeEigenModes} shows the eigenmodes $\ef_i(x)$, and the spatial part of the global-in-time error $\errl(x)$, for the stiffly accurate scheme (left panel), and the non-stiffly accurate scheme (right panel). Three different choices of time step $\dt$ are used.

For the 3-stage stiffly accurate DIRK3, we have
\begin{equation*}
	\ef_1 = O(\dt^3), \quad \ef_2=O(\dt^2), \quad \ef_3=O(\dt^2),
	\quad \mbox{and} \quad \zeromode\equiv0\;\mbox{(not shown).}
\end{equation*}
The modes $\ef_2$ and $\ef_3$ solve singularly perturbed BVPs, and produce BLs in the global-in-time error ($\errl = \errs_3$, since the scheme is stiffly accurate). All modes vanish at the boundary, hence the numerical solution has no error in the boundary values.

For the 2-stage non-stiffly accurate DIRK3, we have
\begin{equation*}
	\ef_1 = O(\dt^3), \quad \ef_2 = O(\dt^2), \quad\mbox{and}\quad
	\zeromode = O(\dt^2)\/.
\end{equation*}
The mode $\ef_2$ solves a singularly perturbed BVP, and produces a BL in the global error. The constant-in-space mode $\zeromode$ (associated to the zero eigenvalue) produces a mismatch in the boundary values for $\errl$. This is reflecting the fact that non-stiffly accurate schemes with conventional boundary conditions do not guarantee that the numerical solution satisfies the exact boundary conditions, see Remark~\ref{rem:BoundaryMismatch}.

For both schemes, (i)~the mode $\ef_1$ associated with the $O(1)$ eigenvalue has no BLs, and (ii)~the global-in-time error exhibits the full 3rd order accuracy away from the BLs.

\section{Weak Stage Order}
\label{sec:WeakStageOrder}
High stage order, i.e., $q > 1$, is not possible with DIRK schemes (see Remark~\ref{rem:q_equal_1}). In this section we introduce new conditions on $(A, \veci{b}, \veci{c})$ that relax the stage order condition to a more general one that is both: (i) devoid of order reduction (OR); and (ii) compatible with a DIRK structure. We therefore refer to the condition as \emph{weak stage order} (WSO).

Weak stage order addresses OR by means of only the RK time-stepping coefficients, and it is compatible with a DIRK structure (cf.~\cite{Scholz1989} for a similar in spirit approach for ROW methods). The concept of WSO generalizes the work by \cite{Rang2014} and provides the sharpest condition on the Butcher coefficients to alleviate OR in linear ODEs. For the time-stepping of PDEs, WSO can provide a practical approach for avoiding OR when using conventional b.c.. It should be stressed that neither high stage order, nor high weak stage order fully remove BLs from the numerical solution. Rather, they reduce the size of the BLs (to the order of the scheme). Thus, in general OR still occurs in the solution's derivatives.

\subsection{Definition of Weak Stage Order}
\label{ssec:WeakStageOrder_Definition}
The idea behind WSO is to find conditions on $(A, \veci{b}, \veci{c})$ independent of $\dt$, that decrease the amplitude (from $O(\dt^{q+1})$) of the singular terms $\vecipower{\Erralpha}{T} \veci{\errs}$ appearing in the error $\errl(x)$ from \eqref{eq:IRK_mat_eqn2}; equivalently viewed as decreasing the contributions $\vecipower{\Erralpha}{T}\vec{r}_i \, \ef_i$ of the singular functions $\ef_i$ in \eqref{eq:EigenModesLinComb}.  We show below that the invariant subspaces of the Butcher tableau matrix $A$ play a key role and define the WSO condition:
\begin{definition}[Weak stage order]\label{def:WeakStageOrder}
A Runge-Kutta scheme $(A, \veci{b}, \veci{c})$ has weak stage order $\soq \geq 1$, if there exists a vector space $\mathcal{V} \subseteq \mathbb{R}^s$ that contains the stage order residuals $\sov{j} \in \mathcal{V}$, for $1 \leq j \leq \soq$, and also satisfies the following two properties:
\begin{enumerate}[(i)]
\item \label{def:WSO_Orthog} (Orthogonality property) $\mathcal{V} \subset \vec{b}^\perp$, i.e., $\vecipower{b}{T} \veci{v} = 0$ for all $\veci{v} \in \mathcal{V}$.
\item \label{def:WSO_Invariant} (Invariant subspace property) $\mathcal{V}$ is $A$-invariant, i.e., $A\veci{v} \in \mathcal{V}$ for all $\veci{v} \in \mathcal{V}$.
\end{enumerate}
\end{definition}
\begin{remark}[Weak stage order as order conditions]
\label{def:WeakStageOrderPractical}
By the Cayley-Hamilton theorem, Definition~\ref{def:WeakStageOrder} is equivalent to (see \cite{KetchesonSeiboldShirokoffZhou2019} for a short proof): The vector $\vec{b}$ is orthogonal to the column space $C(G)$ of the (controllability) matrix
	\begin{equation}
		G :=
		\begin{pmatrix}
			\sov{1}, A\sov{1}, \ldots, A^{s-1}\sov{1}, \sov{2}, A\sov{2}, \ldots,
		A^{s-1}\sov{\soq}
		\end{pmatrix}  \in \mathbb{R}^{s \times s\soq}.
	\end{equation}	
The standard order conditions already imply, via Proposition~\ref{prop:StageOrderVectorOrthog}, that $\vec{b}$ is orthogonal to a subset of the columns of $G$. Hence, WSO can be viewed as imposing extra order conditions $\vec{b} \perp C(G)$.
\TheoremEnd
\end{remark}

Definition~\ref{def:WeakStageOrder} generalizes the notion of stage order, and is automatically satisfied by a scheme with classical stage order $q$. Every RK scheme has both classical stage order $q \geq 1$ and WSO $\soq \geq 1$, since $A \veci{e} = \veci{c}$ guarantees that $\sov{1} = \veci{0}$.  The (abstract) Definition~\ref{def:WeakStageOrder} is helpful in simplifying proofs involving WSO, while the alternative viewpoint in Remark~\ref{def:WeakStageOrderPractical} is useful in practice to construct schemes satisfying high WSO.  Note that WSO $\soq$ implies WSO $\soq-1$, which follows directly from the construction of $G$ in Remark~\ref{def:WeakStageOrderPractical}.  A simplifying criterion for WSO arises when the stage order residuals $\sov{j}$ are eigenvectors of $A$.  We refer to this situation as the weak stage order eigenvector criterion:
\begin{definition}[WSO eigenvector criterion]
\label{def:eig_criterion}
A RK scheme satisfies the eigenvector criterion of order $\soq_{e}$ if for each $1 \leq j \leq \soq_{e}$ there exists $\zeta_j$ such that $A \sov{j} = \zeta_j \sov{j}$, and $\vecipower{b}{T}\sov{j} = 0$.
\end{definition}

Weak stage order is a linear concept, and the analysis in this section shows that it remedies order reduction for linear IBVPs \eqref{eq:IBVP}. In contrast, for problems in which the root cause of order reduction itself is nonlinear, or time dependent, WSO may not achieve the same benefit that classical stage order does. See also \cite{KetchesonSeiboldShirokoffZhou2019}, which further devises additional schemes that satisfy the WSO eigenvector criterion, as well as limitations of the WSO eigenvector criterion.

\subsection{Impact of Weak Stage Order on Error Convergence and Boundary Layers}
\label{ssec:WeakStageOrder_OrderResults}
We show that weak stage order, paired with assumptions \eqref{eq:AssumptionsLS}, can avoid order reduction in RK schemes for the periodically forced solutions examined in \Srm\ref{sec:orderloss_gte}, i.e., that $\errl = O(\dt\,^{\min\{ \soq+1, p\}} )$. The following proposition demonstrates how solutions to  \eqref{eq:IRK_mat_eqn2} with a right hand side proportional to $\sov{j}$ for $j \leq \soq$ do not contribute to error $\errl(x)$.
\begin{proposition}\label{prop:weak_stage_orthogonality}
Consider a Runge-Kutta scheme $(A, \veci{b}, \veci{c})$ with WSO $\soq$, and let $M$ be given by \eqref{eq:MatrixM}. Then for any smooth function $f(x)$ and stage order residuals $\sov{j}$, $1\leq j \leq \soq$, the following quantities vanish for any $x \in \Omega$:
$\vecipower{b}{T} \veci{v}(x) = 0$ and $\vecipower{b}{T} A^{-1} \veci{v}(x) = 0$, where $\veci{v}(x)$ solves
\begin{equation}\label{eqn:prop_orthgonoality}		
  (I - M \diffop )\,\veci{v} = f(x)\,\sov{j}
  \quad \text{with b.c. }\; \veci{v} = 0.
\end{equation}
\end{proposition}
\begin{proof}
Let $\mathcal{V}$ denote the $A$-invariant subspace in Definition~\ref{def:WeakStageOrder}. It suffices to show that $\veci{v}(x) \in \mathcal{V}$ for all $x \in \Omega$. Then $\vecipower{b}{T} \veci{v}(x) =0$ follows by property~(\ref{def:WSO_Orthog}) in Definition~\ref{def:WeakStageOrder}, and property~(\ref{def:WSO_Invariant}) implies that $\mathcal{V}$ is also $A^{-1}$-invariant, so that $A^{-1}\veci{v}(x) \in \mathcal{V}$, and thus $\vecipower{b}{T}A^{-1}\veci{v}(x) = 0$.

We show that $\veci{v}(x) \in \mathcal{V}$ by working in a coordinate basis defined by $\mathcal{V}$ and its orthogonal space $\mathcal{V}^{\perp}$. Let $\dimV = \text{dim}(\mathcal{V})$. Note that no assumption is made on $\dimV$ relative to $\soq$. If some of the vectors $\sov{j}$ are linearly dependent or vanish, $\dimV < \soq$ is possible; and $\dimV \geq \soq$ is possible if the vectors $\sov{j}$ are only contained in a larger $A$-invariant space, but do not span an $A$-invariant space themselves.

Let $\{ \veci{v}_1, \cdots, \veci{v}_{\dimV} \}$ and $\{ \veci{v}_{\dimV + 1}, \ldots, \veci{v}_{s} \}$ form two orthonormal bases for $\mathcal{V}$ and $\mathcal{V}^{\perp}$, respectively. Moreover, define the matrices $\projP = \begin{pmatrix} \veci{v}_1, \ldots, \veci{v}_{\dimV} \end{pmatrix} \in \mathbb{R}^{s \times \dimV}$ and $\projPperp = \begin{pmatrix} \veci{v}_{\sigma+1}, \ldots, \veci{v}_{s} \end{pmatrix} \in \mathbb{R}^{s \times (s - \dimV)}$, and denote the full orthogonal matrix $P = \begin{pmatrix} \projP, \projPperp \end{pmatrix} \in \mathbb{R}^{s\times s}$. We have $\vecipower{b}{T} \projP = 0$, thus (and similarly for $\vecipower{b}{T}A^{-1} \veci{v}(x)$):
\begin{equation*}
  \vecipower{b}{T} \veci{v}(x) =
  \vecipower{b}{T} P P^T \veci{v}(x) =
  \vecipower{b}{T} \projP \projP^T \veci{v}(x)
   + \vecipower{b}{T} \projPperp \projPperp^T \veci{v}(x)
   = \vecipower{b}{T} \projPperp \projPperp^T \veci{v}(x).
\end{equation*}

We now show that $\projPperp^T \veci{v}(x) = 0$, which will complete the proof. First, observe that $\mathcal{V}$ is also an $M$-invariant subspace: for any $\veci{v} \in \mathcal{V}$ we have
\begin{equation*}
  M\,\veci{v}
  = \left(\frac{\Delta t}{z-1}\vece\,\vecipower{b}{T} +
    \Delta t\,A\right)\veci{v}
  = \frac{\Delta t}{z-1}\vece\left(\vecipower{b}{T}\veci{v}\right)+
    \Delta t A\,\veci{v}
  = \Delta t \left( A\,\veci{v}\right) \in \mathcal{V},
\end{equation*}
which follows from the fact that $\mathcal{V}$ is $A$-invariant and orthogonal to $\veci{b}$.

Because $\mathcal{V}$ is $M$-invariant, the matrix $P^T M P $ (which is $M$ written in the coordinate basis $\{\veci{v}_j\}_{j=1}^s$) is block upper-triangular \cite[Chapter 8.6]{Nicholson2003}. Multiplying \eqref{eqn:prop_orthgonoality} by $P^T$, and using the block structure of $P^T M P$, we obtain
\begin{equation*}
  \left[ \begin{pmatrix} I & 0 \\ 0 & I \end{pmatrix}
   -
   \begin{pmatrix}
    \projP^T M \projP   & \projP^T M \projPperp \\
               0        & \projPperp^T M \projPperp
   \end{pmatrix}
   \diffop \right]
   \begin{pmatrix}
     \projP^T \veci{v}(x) \\
     \projPperp^T \veci{v}(x)
   \end{pmatrix}
   = f(x)
   \begin{pmatrix}
    \projP^T \sov{j} \\
    \projPperp^T \sov{j}
   \end{pmatrix}.
\end{equation*}
Hence the vector field $\projPperp^T \veci{v}(x)$ decouples from the $\projP^T \veci{v}(x)$ components. Moreover, the corresponding right hand side vanishes, $\projPperp^T \sov{j} = 0$, because $\sov{j} \in \mathcal{V}$. Hence,
\begin{equation}\label{eqn:reduced_equation}
   \left[ I - \projPperp^T M \projPperp \diffop \right]
   (\projPperp^T \veci{v}(x)) = 0, \quad \text{with b.c. }
   \projPperp^T \veci{v}(x) = 0,
\end{equation}
where the coordinate transformation does not modify the homogeneous b.c.. If $\projPperp^T \veci{v}(x) \neq 0$ were to solve equation \eqref{eqn:reduced_equation}, then equation \eqref{eqn:prop_orthgonoality} would not have a unique solution, in contradiction to Corollary~\ref{cor:invertibility}. Therefore, $\projPperp^T \veci{v}(x) = 0$ is the unique solution to \eqref{eqn:reduced_equation}.
\end{proof}

The importance of Proposition~\ref{prop:weak_stage_orthogonality} is that it does not depend on either $z$ or $\dt$. We now state the main theorem demonstrating that weak stage order avoids OR in IBVP.
\begin{theorem}\label{thm:weakordercondition}
Consider an $s$-stage, $p$-th order implicit Runge-Kutta scheme with weak stage order $\soq \geq 1$, satisfying the assumptions in \eqref{eq:AssumptionsLS}. Then the convergence order for periodic solutions with conventional b.c.\ is $\min\{p,\soq+1\}$, i.e., $\errl =  O(\dt\,^{\min\{p,\soq+1\}})$.
\end{theorem}
\begin{proof}
Using the definition of the LTEs \eqref{eq:ltes_vector}, write $\vec{\lte} = \veci{\varphi}(x) + O(\dt^{\soq + 1})$ where $\veci{\varphi}(x)$ is a linear combination of stage order residuals: $\veci{\varphi}(x) := U^*(x) \sum_{j = 2}^{\soq} \frac{(\imath \omega)^j\dt\,^{j}}{(j-1)!}\,\sov{j}$.  Next, we expand the error $\errl(x)$ in \eqref{eq:IRK_mat_eqn2} in terms of $\dt \ll 1$, and use the fact that $\ltel = O(\dt\,^{p+1})$, and $z = 1 + O(\dt)$, to obtain:
\begin{equation}\label{eq:err_expansion}
  \errl(x) =
  \left( \frac{1}{\vecipower{b}{T} A^{-1} \veci{e}} + O(\dt) \right)		
  \left( \vecipower{b}{T} A^{-1} \veci{\errs}(x) -
  \vecipower{b}{T} A^{-1} \veci{\varphi}(x) +
     O(\dt\,^{\min\{p+1,\soq+1\}}) \right).
\end{equation}
In \eqref{eq:err_expansion}, $\vecipower{b}{T} A^{-1} \veci{e} \neq 0$ by assumption (\ref{eq:AssumptionsLS}e). Furthermore, the term $\vecipower{b}{T} A^{-1} \veci{\varphi}(x) = 0$, because $\varphi(x) \in \mathcal{V}$ is a linear combination of $\sov{j}$ (see Proposition~\ref{prop:weak_stage_orthogonality}).  Therefore we need to estimate $\vecipower{b}{T} A^{-1} \veci{\errs}(x)$. Proposition~\ref{prop:weak_stage_orthogonality} implies that $\vecipower{b}{T} A^{-1} \veci{\errs}_{\varphi} = 0$, where $\veci{\errs}_{\varphi}$ solves
\begin{equation*}
  \veci{\errs}_{\varphi} - M \diffop\,\veci{\errs}_{\varphi} =
  z^{-1} \, \veci{\varphi}(x),	\quad \text{with b.c. } \veci{\errs}_{\varphi} = 0.
\end{equation*}
Since $\veci{\errs} - \veci{\errs}_{\varphi}$ solves a BVP similar to \eqref{eq:IRK_mat_eqn}  with right hand side $z^{-1}(\veci{\lte}+\frac{\ltel\vece}{z-1}-\veci{\varphi})=O(\dt\,^{\soq+1}) + O(\dt\,^{p})$,
Corollary~\ref{cor:invertibility} implies that
$\veci{\errs} - \veci{\errs}_{\varphi} = O(\dt\,^{\min\{p, \soq+1\}})$.
We may then subtract $\veci{\errs}_{\varphi}$ from $\veci{\errs}$ and compute
\begin{equation*}
  \vecipower{b}{T} A^{-1} \veci{\errs}
   = \vecipower{b}{T} A^{-1} (\veci{\errs}-\veci{\errs}_{\varphi} )
   = O(\dt\,^{\soq+1}) + O(\dt\,^p) = O(\dt\,^{\min\{ p, \soq+1\}})\;,
\end{equation*}
which finalizes the proof.
\end{proof}

Theorem~\ref{thm:weakordercondition} demonstrates that order reduction in function value (for periodic solutions) can be avoided with a weak stage order $\soq = p - 1$. The following remark, and numerical calculations in the following sections, indicate that high WSO removes order reduction for non-periodic solutions as well.
\begin{remark} \label{rem:OstermannRoche}
Ostermann and Roche \cite{OstermannRoche1992} proved (under assumptions similar to \eqref{eq:AssumptionsLS}) that if a RK scheme applied to \eqref{eq:IBVP} with homogeneous boundary conditions satisfies:
\begin{align}\label{eq:cond_Wk}
	W_k(z) \equiv 0, \quad \textrm{for } 1 \leq k \leq \soq \quad \quad \textrm{where} \quad W_k(z) := \frac{k \vecipower{b}{T} (I - z A)^{-1} \sov{k} }{R(z) - 1},
\end{align}
then the scheme converges in $L^{r}$ ($1 < r < \infty$), with order $\min\{ p, \soq+2+\nu \}$, where $\nu$ depends\footnote{For $\diffop = \partial_{xx}$, and $L^2$ convergence, $\nu = \frac{1}{4} - \varepsilon$ for any $\varepsilon > 0$, so effectively one can take $\nu = \frac{1}{4}$.} on $\diffop$ and $r$.  The focus of \cite{OstermannRoche1992} was to establish $L^{r}$ convergence results, and not to investigate what conditions would guarantee \eqref{eq:cond_Wk}.  It is easy to verify that WSO $\tilde{q}$ immediately implies \eqref{eq:cond_Wk} since $\mathcal{V}$ is an invariant subspace of $(I-zA)$, and $\vecipower{b}{T}(I - z A)^{-1} \sov{k} = 0$ for all $1 \leq k \leq \soq$.

Note that (i)~homogeneous b.c.\ on $u^*$ increase the convergence rate to $\min\{p,\soq + 2\}$; while (ii)~the constant $\nu$ stems from measuring the BL size of the $\ef_i$'s in the $L^r$ norm, i.e., $\ef_i$ has BL width $O(\sqrt{\dt})$ for $\diffop = \partial_{xx}$.  Items (i--ii) imply that the  convergence rate of $\min\{p, \tilde{q}+2 + \nu\}$ proved in \cite{OstermannRoche1992} is consistent with Theorem~\ref{thm:weakordercondition} and~\S\ref{sec:orderloss_gte}.
\TheoremEnd
\end{remark}

\subsection{A DIRK Scheme with High Weak Stage Order}
\label{ssec:WeakStageOrder_DIRK}
An important advantage of WSO is that it allows DIRK schemes to avoid order reduction (cf.~Remark~\ref{rem:q_equal_1}). Here we present a stiffly accurate, L-stable, 4-stage, 3rd order DIRK scheme with WSO 2. This scheme is constructed using the eigenvector criterion in Definition~\ref{def:eig_criterion}, i.e., the stage order residual $\sov{2}$ is a right eigenvector of $A$. The coefficients $A = (a_{ij})\in\mathbb{R}^{4\times 4}$, and $\veci{b}$, $\veci{c}\in\mathbb{R}^4$ are given by
\begin{equation}
\label{WSO:scheme}
\begin{split}
	A &=
	{\footnotesize
		\setlength\arraycolsep{3pt}
		\begin{bmatrix}
			0.019000728905359 & & & \\
			0.404346056017447 & \phantom{-}0.384357175123333 & & \\
			0.064879084117003 & -0.163896402946036 & 0.515452312221597 & \\
			0.023435493738931 & -0.412078778885435 & 0.966611612813460 & 0.422031672333044\\
		\end{bmatrix},
	}	
	\\
	b_i &= a_{4i}\quad\mbox{and}\quad
	c_i = \sum_{j=1}^i a_{ij}\quad\mbox{for}\quad 1\leq i \leq 4.
\end{split}
\end{equation}
In line with \cite{KetchesonSeiboldShirokoffZhou2019}, this scheme has been found by searching the parameter space of stiffly accurate 4-stage DIRK schemes (with nonzero diagonal entries), while imposing the order conditions \eqref{eq:ordercondition}, the WSO eigenvector criterion (Def.~\ref{def:eig_criterion}), and A-stability (verified by evaluating the stability function $R(\zeta)$ along the imaginary axis) as constraints. MATLAB's \texttt{fmincon} (with default settings) is employed, minimizing the $L^2$ norm of the residual of the 4th order conditions, starting from thousands of randomly chosen initial points, and selecting the scheme with the smallest objective function. It has generally been observed that this optimization problem is non-convex and not well-conditioned; hence, the scheme \eqref{WSO:scheme} should not be expected to be optimal. However, it does satisfy all constraints up to machine precision and yields good convergence results for various test problems, as shown in \Srm\ref{sec:NumericalExamples}.

Weak stage order reduces the magnitude of the coefficients $\vecipower{\Erralpha}{T}\vec{r}_i$ in front of the singular functions $\ef_i$ ($2 \leq i \leq 3$) in \eqref{eq:EigenModesLinComb}. This decreases the amplitude of the boundary layers that contribute to the error expansion for $\errl$.  For example, in the scheme \eqref{WSO:scheme}:  $A\sov{2}=a_{11}\sov{2}$ implies $\sov{2}$ is a right eigenvector of $M$ (for any $\dt$). Without loss of generality, setting $\vec{r}_2 = \sov{2}$, renders the coefficient $\vecipower{\Erralpha}{T}\vec{r}_2 \propto \vecipower{b}{T}\,A^{-1}\, \sov{2} = 0$ so that $\ef_2$ does not contribute to the error $\errl$. Furthermore, one can work out that $\vecipower{\Erralpha}{T}\vec{r}_3 \, \ef_3 = \vecipower{\Erralpha}{T}\vec{r}_4 \, \ef_4 = O(\dt^3)$ so that the singular modes $\ef_3$ and $\ef_4$ contribute one order less to the global error ($\zeromode = 0$ due to stiff accuracy). The BL amplitude in the error $\errl$ is then reduced (but not eliminated) to $O(\dt\,^3)$.  One will still observe a further order reduction in the solution derivatives.

\section{Modified Boundary Conditions}
\label{sec:MBC}
This section presents an alternative approach for avoiding order reduction by modifying the prescribed RK b.c.---hereon referred to as \emph{modified boundary conditions} (MBC). The concept of MBC itself is not new (see~\S\ref{ssec:PriorResearch}), with the most general formulation given in \cite{Alonso2002, AlonsoCano2004}.  The purpose of this section is to show how MBC can be systematically derived by removing the boundary layers in the RK spatial approximation error. The advantage of MBC (over weak stage order) is that they do not restrict the RK scheme that is used; the disadvantage is that they are more complicated to implement. In \Srm\ref{ssec:MBC_derivation} we derive MBC via a power series expansion; and in \Srm\ref{ssec:MBC_BoundaryMismatch} we show that they suitably reduce the magnitude of BLs, and also reduce any boundary mismatch for non-stiffly accurate schemes.

\subsection{Derivation of MBC via Power Series Expansion}
\label{ssec:MBC_derivation}
In this subsection we choose the b.c.\ for \eqref{eq:dirk_step} so that the solution $\veci{u}^{n+1}$ may be expanded in formal powers of $\dt$ (uniformly across the entire domain) up to the order of the RK scheme. This will, effectively, suppress the BLs in $\veci{u}^{n+1}$ up to the scheme's order (but not to all orders) and alleviate order reduction.

In the absence of BLs, the stage vector $\vecipower{u}{n+1}$ can be written via a formal power series expansion as: $\vecipower{u}{n+1}_p \sim \sum_{j\geq 0}\dt\,^{j}\,\veci{U}_j$. Substituting this expansion into \eqref{eq:dirk_step} and collecting equal powers of $\dt$ leads to expressions for $\veci{U}_j$. When $\diffop$ is linear, the power series expansion for $\vecipower{u}{n+1}_p$ reduces to the Neumann series expansion \cite[Chapter 6]{Schecter2000}, and results in a recursive formula for $\veci{U}_j$:
\begin{equation*}
 \veci{U}_j = A\diffop\,\veci{U}_{j-1}\;\;\mbox{for}\;\; j\geq 2\quad
   \mbox{with}\quad
 \veci{U}_1 = A\diffop\,\veci{U}_0+A\vecipower{f}{n+1}\;\;
   \mbox{and}\;\;\veci{U}_0 = u^n\vece.
\end{equation*}
Thus $\veci{U}_j=(A\diffop)^j u^n\vece + (A\diffop)^{j-1}(A\vecipower{f}{n+1})$ for $j\geq 1$. Hence, $\vecipower{u}{n+1}_p$ takes the form
\begin{equation}\label{eq:FormalNeumannExpansion}
	\vecipower{u}{n+1}_p \sim
	u^n\vece + \sum_{j\geq 1}\,\dt\,^j
	\left(A^j\diffop^j\,u^n\vece + A^j\diffop^{j-1}\vecipower{f}{n+1}\,\right).
\end{equation}
We refer to the expansion \eqref{eq:FormalNeumannExpansion} (also known in the matched asymptotic expansions theory~\cite{BenderOrszag1978, CarrierPearson1988, MHHolmes1995, kevorkian1996multiple}) as the \emph{regular solution} to \eqref{eq:dirk_step}. If $\vecipower{f}{n+1}$ is (infinitely) smooth, the regular solution is a particular solution to \eqref{eq:dirk_step} and has no boundary layers---but does not satisfy homogeneous boundary conditions $\vecipower{u}{n+1}_p \neq 0$. Hence, to avoid BL in $\veci{u}^{n+1}$ (to some order), we need the b.c.\ of \eqref{eq:dirk_step} to match the values of $\vecipower{u}{n+1}_p$.

Truncating the series \eqref{eq:FormalNeumannExpansion} up to the scheme's order $p$, and evaluating at the boundary, yields a set of b.c.\ that match \eqref{eq:FormalNeumannExpansion} up to order $p$.  The PDE \eqref{eq:IBVP} can then be use to replace terms involve high powers of $\diffop$, in terms of the data on the boundary $g(t)$. A technical detail is that, in \eqref{eq:FormalNeumannExpansion}, the operator ${\diffop}^j$ is not applied to the exact solution, but to the numerical solution which does not satisfy the PDE exactly. Hence we (i) express the numerical solution $u^n = u^*(t_n) + \errl^n$ in terms of the exact solution $u^*$ and discretization error $\errl^n$ at time $t_n$; and (ii) use the PDE $u^*_t = \diffop u^* + f$ to replace $\diffop^j u^*(t_n)$ by $\partial_t^j u^*(t_n)$ and the forcing $f$ at time $t_n$. Taylor-expanding $\vecipower{f}{n+1}$ at $t_n$, the truncated expansion yields
\begin{equation}\label{eq:MBC_series}
\begin{split}
\vecipower{u}{n+1}_p = u^n\vece
+ \sum_{j=1}^{p} \dt^j\!
&\Bigg[ \partial_t^j u^*(t_n)A^{j-1}\vecc +
\sum_{k=2}^{j-1}\left(\diffop^{j-k-1}\partial_t^{k}f^n\right)\\[-.2em]
&\times\left(\tfrac{1}{k!}A^{j-k}\vecipower{c}{k}-A^{j-1}\vecc\right) \Bigg]
+ \sum_{j=1}^{p} \dt^j A^{j-1} \vecc\,\diffop^j\,\errl^n\/.
\end{split}
\end{equation}

In equation \eqref{eq:MBC_series}, $\errl^n$ is the error incurred by the formal expansion, so that by construction it is assumed to be $O(\dt\,^p)$. The MBC are then obtained by neglecting the error term $\errl^n$, and evaluating the truncated series \eqref{eq:MBC_series} at the boundary:
\begin{equation}\label{eq:MBC_generalformula}
	\mbc{} := g^n\,\vece + \sum_{j=1}^{p}
	\dt^j\! \left[ \partial_t^j g^n A^{j-1}\vecc +
	\sum_{k=2}^{j-1}\left(\diffop^{j-k-1}\partial_t^{k}f^n\right)
	\!\left(\frac{A^{j-k}\vecipower{c}{k}}{k!}-A^{j-1}\vecc\right)\!\right]\!.
\end{equation}
Here $u^n$ at the boundary was set to $g^n$. By construction, $\mbc{}$ matches \eqref{eq:FormalNeumannExpansion} up to the scheme's order $p$, and incorporates b.c.\ information. The MBC are unique up to order $p$, i.e., any other b.c.\ that suppress the singular behavior up to the same order, can differ from the MBC only by $O(\dt\,^{p+1})$ terms. As an example, the 3rd order MBC (MBC3), i.e., $p=3$, take the following form:
\begin{equation*}
 \mbc{} = g^n\vece + \dt\partial_t g^n\vecc + \dt^2\partial_t^2 g^n A\vecc
 + \dt^3\left( \partial_t^3 g^n A^2\vecc + \left(\partial_t^2 f^n\right)
 \left(\textstyle\frac{A\vecipower{c}{2}}{2}-A^2\vecc\right)
 \right).
\end{equation*}

This derivation shows that any b.c.\ one prescribes that agrees with $\mbc{}$ up to the order of the method, will remove order reduction---for instance those obtained by \cite{Alonso2002}.

\begin{remark}
In two important special cases, the MBC \eqref{eq:MBC_generalformula} simplify. First, when the boundary data $g$ and the forcing $f$ at the boundary are time-independent, the summation in \eqref{eq:MBC_generalformula} vanishes. This reflects the fact that order reduction does not arise for autonomous problems. Second, the MBC $\mbc{}$ can also be written as the conventional b.c., modified by a sum involving only the stage order residuals $\sov{j}$. Thus, when the RK scheme's stage order satisfies $q\geq p$, all terms involving the stage order residuals vanish, implying that the MBC $\mbc{}$ agree with the conventional b.c.\ up to the scheme's order $p$.
\TheoremEnd
\end{remark}

\subsection{Boundary Value Mismatch}
\label{ssec:MBC_BoundaryMismatch}
Although the MBC \eqref{eq:MBC_generalformula} can be used to avoid order reduction, they may still result in a small mismatch of the numerical solution at the boundary with the exact prescribed boundary data, i.e. $u^{n+1} \neq g^{n+1}$, even for stiffly accurate schemes (see Remark~\ref{rem:BoundaryMismatch}). Enforcing $u^{n+1} = g^{n+1}$, however, may be of practical interest. In this subsection, we first show that the MBC yields a boundary mismatch error, $\errl^{n+1} = u^{n+1}-g^{n+1}$, that is always of the scheme's order (which is good). Moreover, we provide a recipe to further modify $\mbc{}$ to ensure that $u^{n+1}=g^{n+1}$ while still avoiding order reduction (which is even better). Note that the MBC in \cite{Alonso2002} reduces the boundary mismatch to the scheme's order, however, they did not investigate enforcing the exact boundary conditions.

\subsubsection{Quantification of the boundary error in the MBCs.}
To obtain an expression for the boundary error, we use \eqref{eq:dirk_step} to rewrite the update \eqref{eq:dirk_step2} in terms of $\veci{u}^{n+1}$ in lieu of $\diffop \, \veci{u}^{n+1}$:
\begin{equation}\label{eq:IRK_explicitupdate}
  u^{n+1} = u^n + \vecipower{b}{T} A^{-1}\,
  (\vecipower{u}{n+1} - u^n\vece).
\end{equation}
Evaluating \eqref{eq:IRK_explicitupdate} at the boundary and subtracting the true boundary value yields an expression for the error at the boundary
\begin{equation}\label{eq:BoundaryError}
 \errl^{n+1} = (1-\vecipower{b}{T}A^{-1}\vece)\,\errl^n +
      \vecipower{b}{T}A^{-1}(\vecipower{u}{n+1}-g^n\vece) + g^n-g^{n+1}.
\end{equation}
We may quantify the boundary mismatch $\errl^{n+1}$ introduced by the MBC, by (i)~substituting the MBC $\vecipower{u}{n+1} = \mbc{}$ from equation \eqref{eq:MBC_generalformula} into \eqref{eq:BoundaryError}, and (ii)~Taylor-expanding $g^{n+1}$ at $t_{n}$ to obtain
\begin{equation*}
 \begin{split}
  \errl^{n+1} &= (1-\vecipower{b}{T}A^{-1}\vece)\,\errl^n + \sum_{j=1}^{p}\,
       \dt^j\,\left[ \partial_t^j g^n\,\left(\vecipower{b}{T}A^{j-2}\vecc -
       \frac{1}{j!}\right)\right.\\
            &+\left.\sum_{k=2}^{j-1}\,\diffop^{j-k-1}\,\partial_t^{k}f^n\,
       \left(\frac{1}{k!}\,\vecipower{b}{T}A^{j-k-1}\vecipower{c}{k}-
       \vecipower{b}{T}A^{j-2}\vecc \right)\right]+ O(\dt\,^{p+1})\/.
 \end{split}
\end{equation*}
In the above expression for $\errl^{n+1}$, the first $p$ terms in the summation vanish due to the order conditions \eqref{eq:ordercondition}. Hence $\errl^{n+1}=O(\errl^{n}) + O(\dt\,^{p+1})$, which implies that the global error at the boundary is at most $O(\dt\,^p)$. In order words: the MBC generally introduce an error in $u^{n+1}$ at the boundary, but order reduction is avoided.

\subsubsection{Eliminating boundary mismatch.}
Equation \eqref{eq:BoundaryError} can be used to further modify $\mbc{}$ to ensure that the numerical solution satisfies the true b.c.\ at every time step. If $\errl^{n} = 0$ for all $n$, then the numerical solution $\vecipower{u}{n+1}$ at the boundary satisfies
\begin{equation}\label{eq:ExactBC_constraint}
 \vecipower{b}{T}A^{-1}\vecipower{u}{n+1} =
 (\vecipower{b}{T}A^{-1}\vece -1) g^n + g^{n+1}.
\end{equation}
Conversely, if equation \eqref{eq:ExactBC_constraint} holds and the initial data satisfy the true b.c., i.e., $\errl^0 = 0$, then $\errl^{n} = 0$ for all $n>0$. Equation \eqref{eq:ExactBC_constraint} defines one linear constraint on the values of $\vecipower{u}{n+1}$ at the boundary. Hence to ensure that $\errl^{n+1} = 0$, one only needs to modify the component of $\mbc{}$ in the direction of $\vecipower{b}{T}A^{-1}$ to satisfy the constraint \eqref{eq:ExactBC_constraint}, while keeping components orthogonal to $\vecipower{b}{T}A^{-1}$ unchanged. This leads to a new set of MBC
\begin{equation}\label{eq:EBCformula}
 \mbc{}^{\,*} = \mbc{} -
  \left[\vecipower{b}{T}A^{-1}\left(\mbc{}-g^n\vece\right) +
        g^{n} - g^{n+1}\right]
  \frac{(\vecipower{b}{T}A^{-1})^T}{\|\vecipower{b}{T}A^{-1}\|^2}.
\end{equation}
By construction, the b.c.\ \eqref{eq:EBCformula} $\vecipower{u}{n+1} = \mbc{}^{\,*}$ satisfy the linear constraint \eqref{eq:ExactBC_constraint}, and hence ensure that $\errl^{n} = 0$ for every $n$ (provided that $\errl^0 = 0$). In addition, the modification $\vecipower{b}{T}A^{-1}\left(\mbc{}-g^n\vece\right) + g^{n} - g^{n+1}$ in formula \eqref{eq:EBCformula} is only an $O(\dt\,^{p+1})$ correction to $\mbc{}$.  Hence, the boundary conditions $\mbc{}^{\,*}$ still suppress the singular behavior in the numerical solution, to order $p$, and thereby avoids order reduction.  For stiffly accurate RK schemes (i.e., $\vecipower{b}{T}A^{-1} = (0,\dots,0,1)$),
$\mbc{,i}^{\,*} = \mbc{,i}$ for stages $1 \leq i \leq s-1$, while $\mbc{,s}^{\,*} = g^{n+1}$.

\subsection{Limitations of MBC}\label{ssec:MBC_Limitations}
The MBC formulas derived above hold for linear problems, where the power series solution for $\vecipower{u}{n+1}_p$ matches a Neumann series expansion. A key part of the derivation is to use the PDE to express the MBC in terms of $\partial_t^j g(t_n)$ and $\diffop^{i}\partial_t^j f(t_n)$, which are computable from the data $g$ and $f$. Consequently, MBC are challenging to apply when the data $g$ and $f$ are given in a way that their derivatives are difficult to obtain/compute.

Another fundamental difficulty arises when $\diffop$ is nonlinear. In this case, the power series expansion of the solution to \eqref{eq:dirk_step} involves, in general, terms that are not directly computable from the known data $g$ and $f$. Such a limitation may seriously hinder the practical use of MBC for nonlinear problems. For example, consider the viscous Burgers' equation (see \Srm\ref{ssec:NumExamples_VicousBurgers}) where $\diffop u$ is replaced by the nonlinear operator $\mathcal{N} u = \nu u_{xx} - uu_x$. When evaluating the truncated power series expansion at the boundary, the terms up to 2nd order in $\dt$ \emph{can} be expressed in terms of $g(t_n)$, $\partial_t g(t_n)$ and $\partial_{tt}g(t_n)$. However, the 3rd order term contains the boundary evaluation of $(\partial_t u^*(t_n))\partial_x(\partial_t u^*(t_n))$, which requires knowledge of spatial derivatives of the exact solution.

\section{Numerical Examples}
\label{sec:NumericalExamples}
In this section we illustrate the order reduction phenomenon, and the two remedies developed above (weak stage order (\Srm\ref{sec:WeakStageOrder}) and modified b.c.\ (\Srm\ref{sec:MBC})), in several numerical examples: heat equation (\Srm\ref{ssec:NumExamples_HeatEqn}), Schr\"odinger (\Srm\ref{ssec:NumExamples_Schroedinger}), advection-diffusion (\Srm\ref{ssec:NumExamples_AdvectionDiffusion}), viscous Burgers' (\Srm\ref{ssec:NumExamples_VicousBurgers}), linear advection and Airy's equation (\Srm\ref{ssec:NumExamples_Airy}). The method of manufactured solutions \cite{SalariKnupp2000} is used to construct a solution in each case. The spatial approximation is conducted via fourth-order centered differences on a fixed grid with 10000 cells. This renders spatial approximation errors negligible, thus isolating temporal discretization errors (measured in the maximum norm, unless noted otherwise), in line with the analysis in \Srm\ref{sec:orderloss_gte}.

We focus on stiffly accurate DIRK schemes here due to their practical interest. In each example, we first demonstrate OR incurred with the standard 3-stage, 3rd-order DIRK scheme with weak stage order (WSO) 1 \eqref{eq:DIRK3}. Then we show how MBC, applied to the same RK scheme, recover the full order of convergence. Finally, we show that the new 4-stage, 3rd-order DIRK scheme with WSO 2 (\Srm\ref{ssec:WeakStageOrder_DIRK}) recovers the full convergence order (for function values) as well. The examples also highlight some limitations of the new approaches, such as arising in high-order MBC for nonlinear problems or the recovery of the full order in derivatives.

\subsection{Heat Equation}\label{ssec:NumExamples_HeatEqn}
This test case has been considered already in \Srm\ref{ssec:example_order_loss} to introduce the order reduction phenomenon. Here we show that: (i)~MBC recover full convergence order for function values as well as some derivatives; (ii)~the full order for derivatives is only recovered when the MBC are carried out to the appropriate order; and (iii)~DIRK schemes with high WSO recover the full order for function values, but generally not for derivatives. Consider the 1D heat equation
\begin{equation}\label{eq:heat}
	u_t = u_{xx} + f\quad\mbox{for}\quad (x,t)\in(0,1)\times(0,1],
\end{equation}
with solution $u^{*}(x,t) = \cos(15t)\sin(5x+5)$, and forcing $f$, Dirichlet b.c.~$u = g$ on $\{0,1\}\times(0,1]$, and initial condition $u = u_0$ chosen accordingly.

\begin{figure}
	\hfillL
	\includegraphics[width=0.32\textwidth]{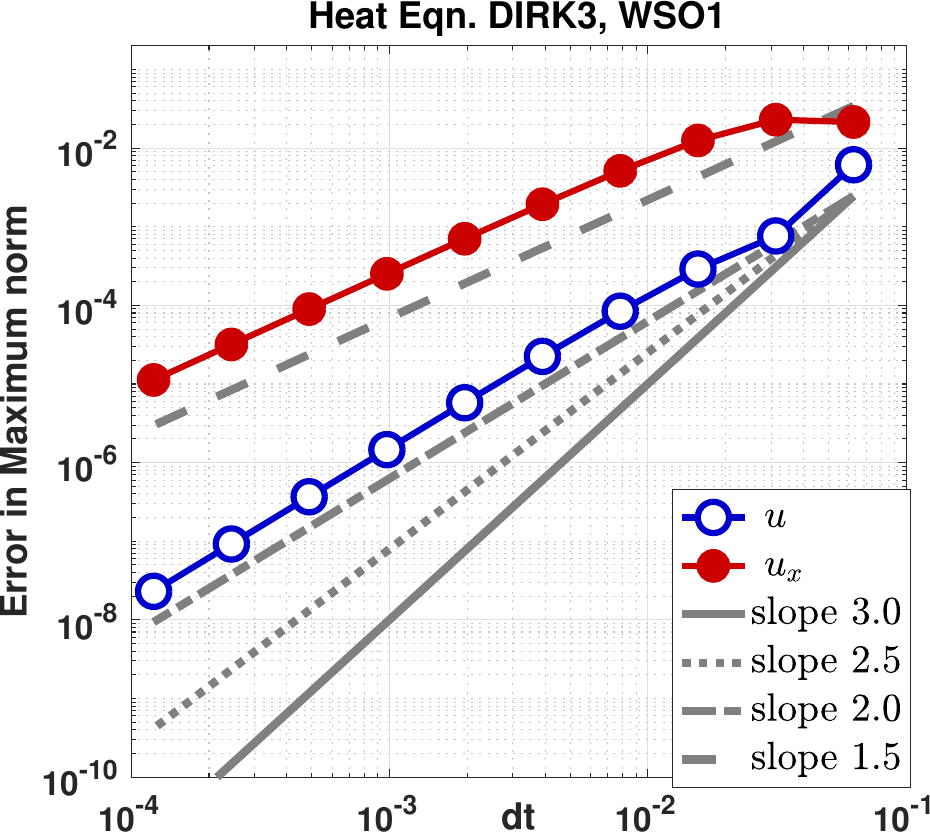}
	\hfill
	\includegraphics[width=0.32\textwidth]{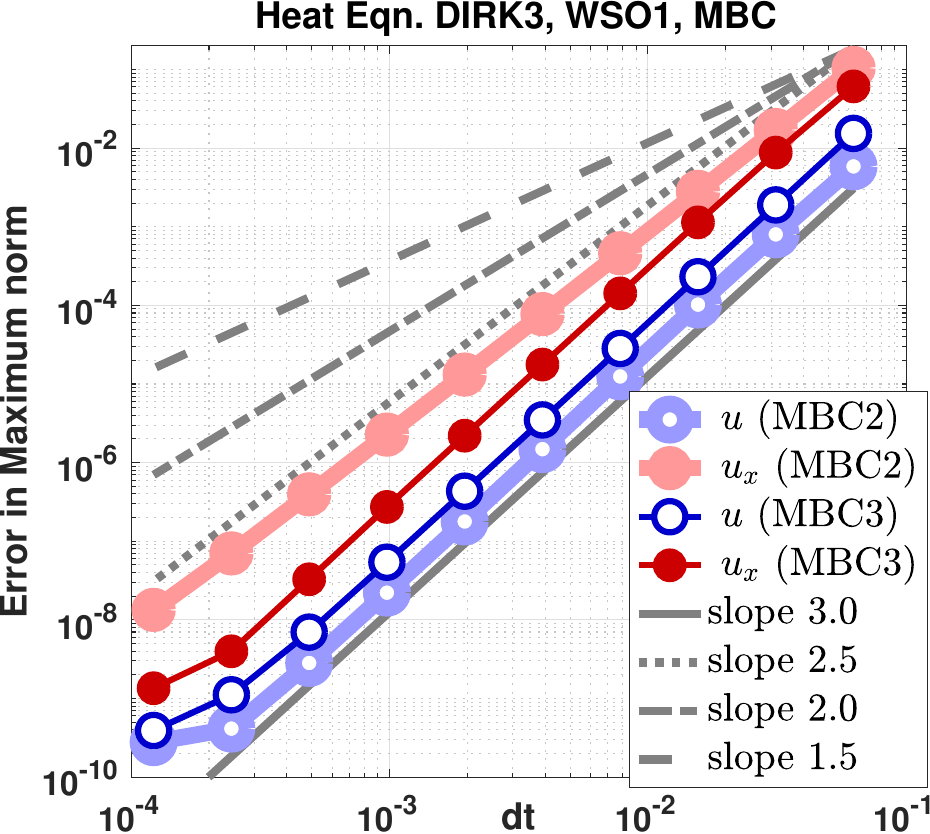}
	\hfill
	\includegraphics[width=0.32\textwidth]{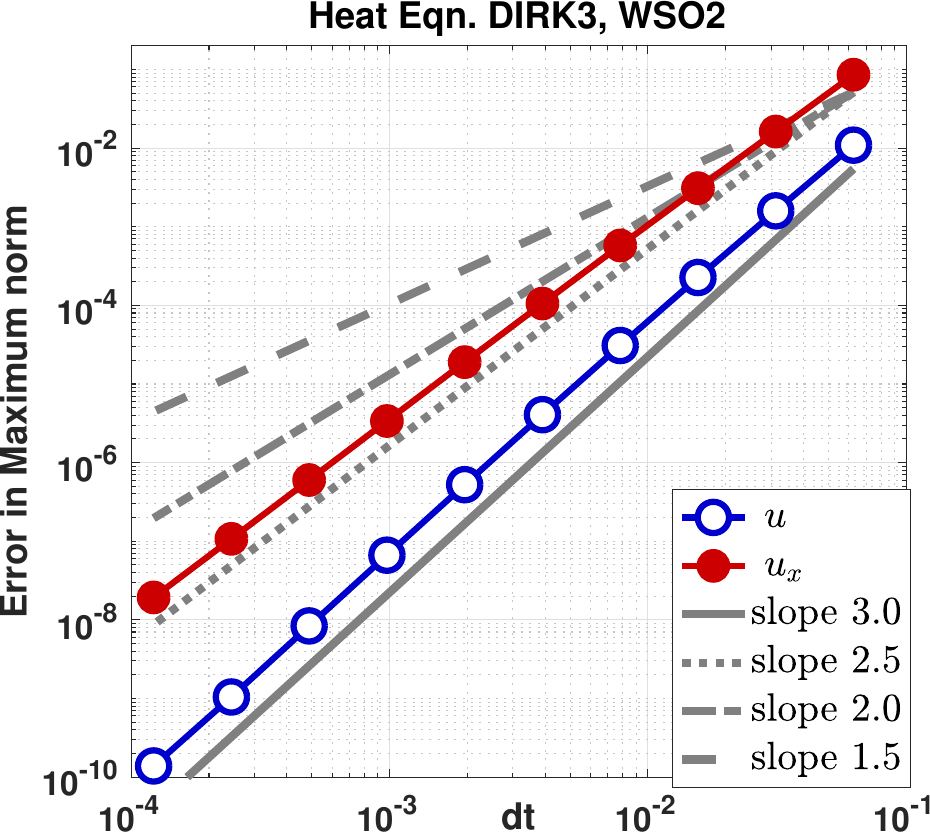}
	\hfillR
	\caption{Error convergence ($u$ blue; $u_x$ red) for 3rd order schemes for the heat equation \eqref{eq:heat}. WSO~1 scheme with conventional b.c.~(left); same scheme with MBC (middle); WSO~2 scheme (right).}
	\label{fig:errcvg_heat_dirk3}
\end{figure}

Figure~\ref{fig:errcvg_heat_dirk3} shows the convergence orders of $u$ and $u_x$ for the WSO~1 scheme with conventional b.c.~(left), for the same scheme with MBC (middle), and for the WSO~2 scheme with conventional b.c.~(right). In the middle panel, MBC are carried out up to the 2nd (MBC2) and 3rd (MBC3) order terms, respectively. Order reduction renders $u = O(\dt^2)$, $u_x = O(\dt^{1.5})$, and $u_{xx} = O(\dt^1)$ for the WSO~1 scheme. For the same scheme, the full MBC3 recover $u = O(\dt^3)$, $u_x = O(\dt^3)$, and $u_{xx} = O(\dt^3)$, while the MBC2 recover $u = O(\dt^3)$, but yield reduced orders in $u_x = O(\dt^{2.5})$, and $u_{xx} = O(\dt^2)$. The same orders are obtained with the WSO~2 scheme (with conventional b.c.). Note that the errors in $u_{xx}$ are not displayed in the figure, but the convergence orders are equally clear as for $u$ and $u_x$.

\subsection{Schr\"odinger Equation}\label{ssec:NumExamples_Schroedinger}
As an example of a dispersive problem (which fails assumption~(\ref{eq:AssumptionsLS}a)), we consider the Schr\"o\-dinger equation
\begin{equation}\label{eq:schroedinger}
u_t = \frac{\imath \omega}{k^2} u_{xx} \quad\mbox{for}\quad (x,t)\in(0,1)\times(0,1.2],
\end{equation}
with $k=5$ and $\omega = 2\pi$, solution $u^{*}(x,t) = \exp(\imath (kx-\omega t))$, Dirichlet b.c.~$u = g$ on $\{0,1\}\times(0,1.2]$, and initial condition $u = u_0$ chosen accordingly.

Figure~\ref{fig:errshape_schro_dirk3} shows that in addition to a time-periodic error with BLs, the RK scheme produces transient dispersive waves in the error far from the domain boundaries. And even more: these dispersive waves may produce order-reduction-like effects in the interior of the domain. The total RK error can be understood as a superposition of the time-periodic error (having $O(\dt^2)$ BLs and $O(\dt^3)$ error away from the BLs outlined in \Srm\ref{sec:orderloss_gte}), and a transient dispersive wave that solves the RK scheme applied to the homogeneous equation \eqref{eq:schroedinger} (i.e., $f=0$ and $g=0$).

\begin{figure}
	\hfillL
	\includegraphics[width=0.32\textwidth]
	{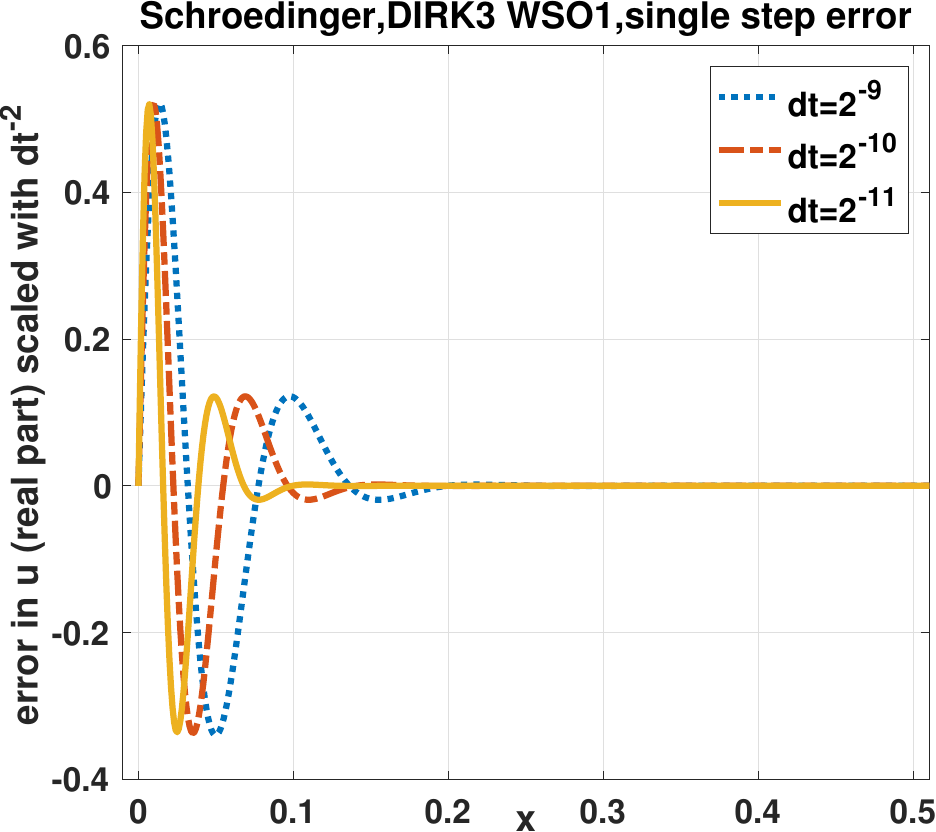}
	\hfill
	\includegraphics[width=0.32\textwidth]
	{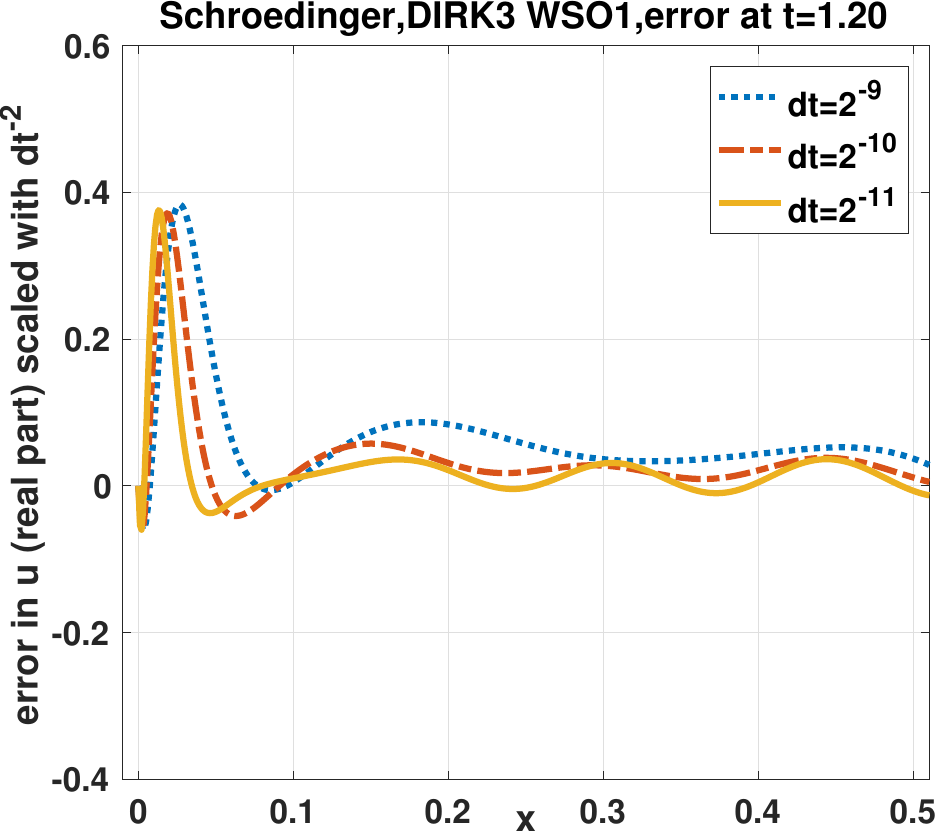}
	\hfill
	\includegraphics[width=0.32\textwidth]
	{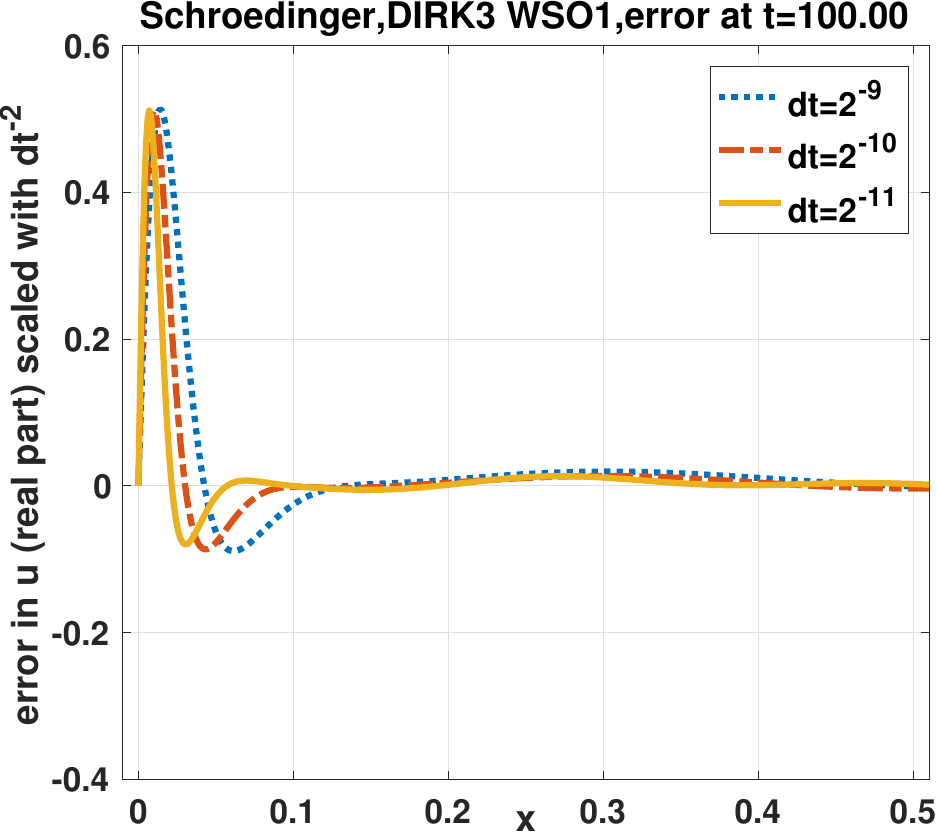}
	\hfillR
	\caption{Errors (real part, scaled with $1/\dt^2$) as functions of $x$ for the Schr\"odinger equation \eqref{eq:schroedinger}, solved with a WSO~1 scheme with conventional b.c., after a single step (left), at a transient time (middle), and at a large time (right). Shown is the left half of the domain (the right half looks similar). The transient error component away from the BLs is clearly visible.}
	\label{fig:errshape_schro_dirk3}
\end{figure}

\begin{figure}
	\hfillL
	\includegraphics[width=0.32\textwidth]{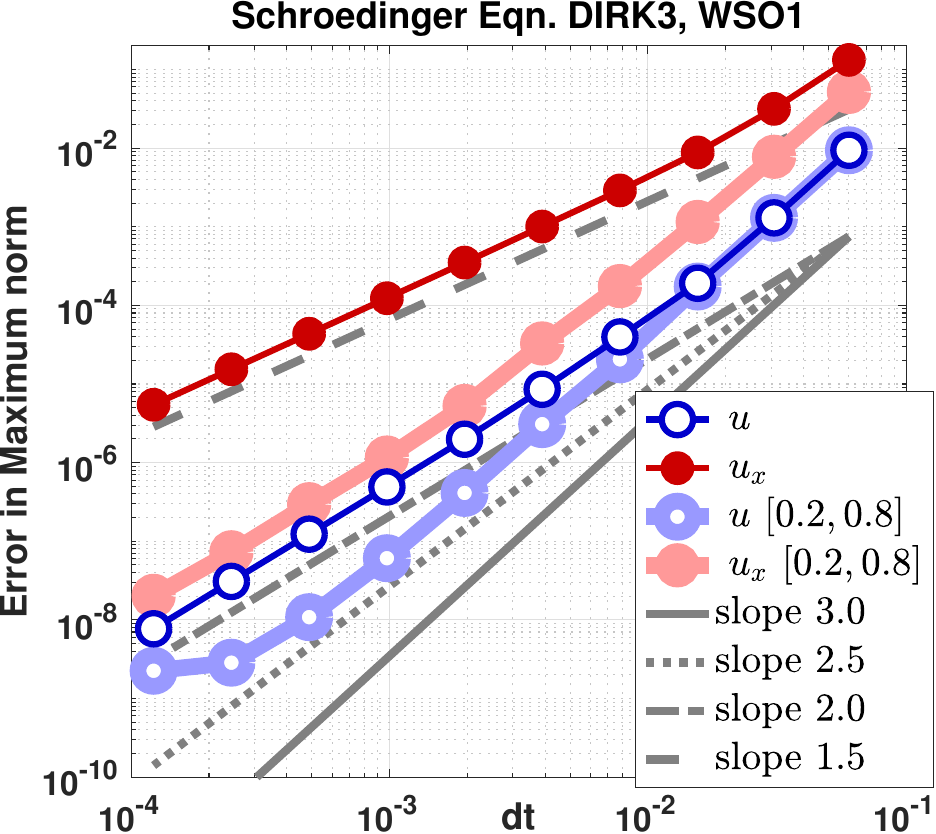}
	\hfill
	\includegraphics[width=0.32\textwidth]{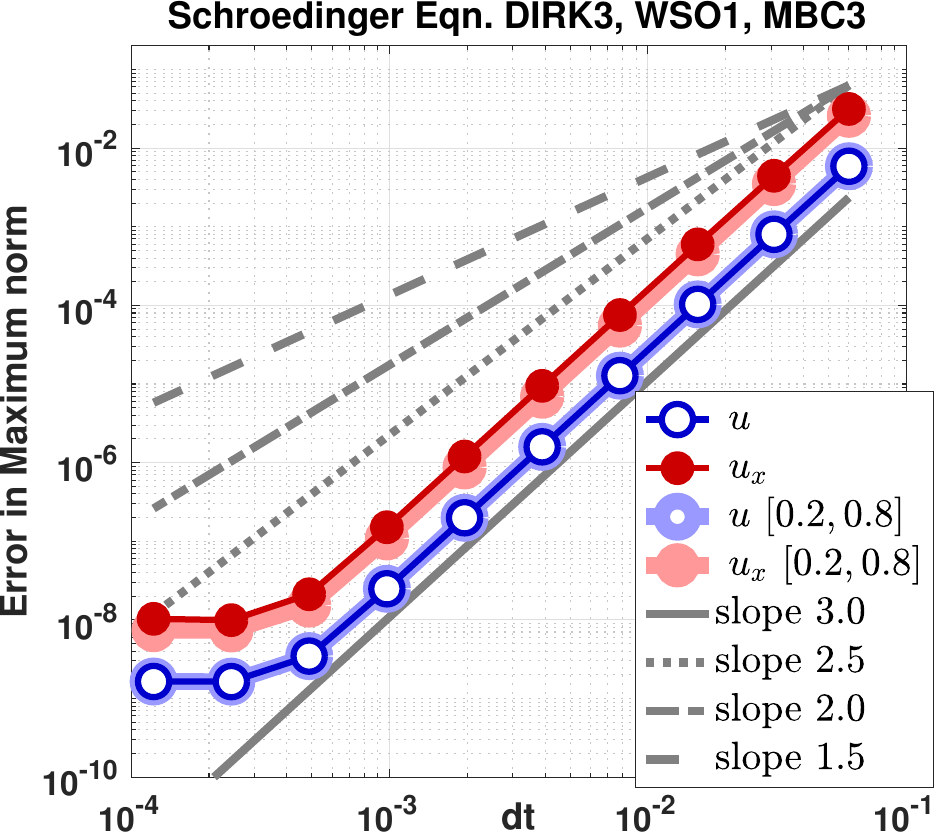}
	\hfill
	\includegraphics[width=0.32\textwidth]{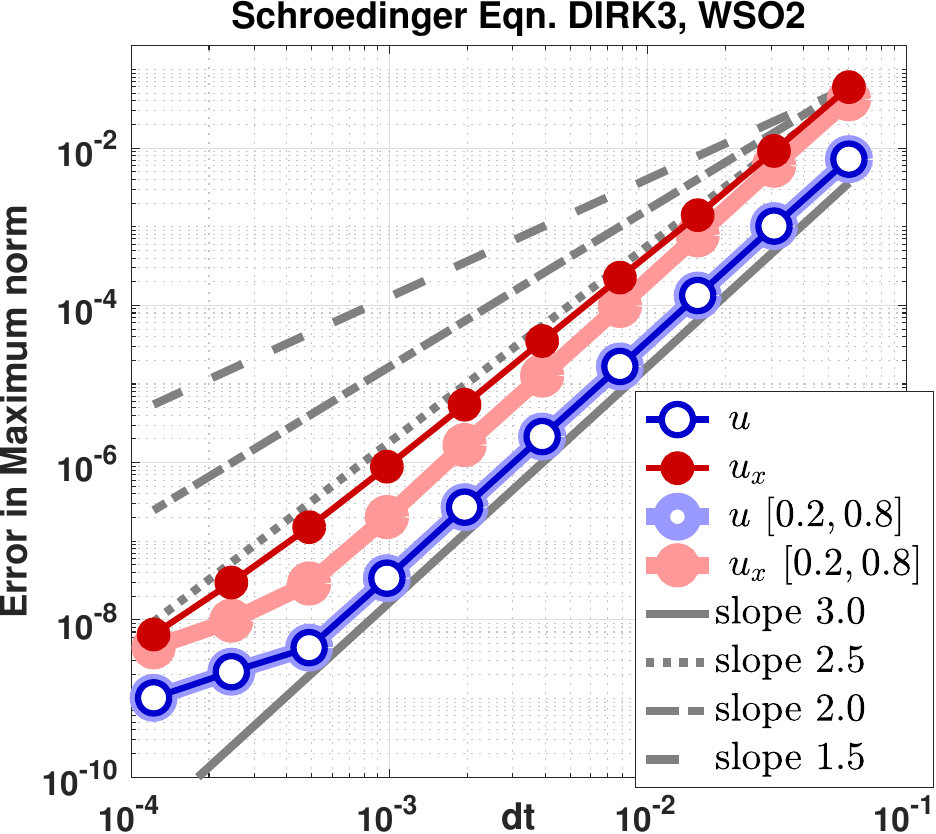}
	\hfillR
	\caption{Error convergence ($u$ blue; $u_x$ red) for 3rd order
		schemes for the Schr\"odinger equation \eqref{eq:schroedinger}.
		WSO~1 scheme with conventional b.c.~(left); same scheme with MBC (middle); WSO~2 scheme (right).
		The error convergence measured away from the BLs is shown for $u$ (light blue) and $u_x$ (light red).}
	\label{fig:errcvg_schro_dirk3}
\end{figure}

Figure~\ref{fig:errshape_schro_dirk3} shows the shape (in $x$, only half of the domain is shown) of the error (here re-scaled with $1/\dt^2$) at three different times: after one step (left panel), at a transient time (middle), and after a long time (right). Except for the $O(\dt)$ time, the BL is dominated by the time-periodic component, while the domain's interior is dominated by the transient component. One can clearly see that the transient component decays slowly in time (because the RK scheme is asymptotically stable for any imaginary eigenvalue). However, for transient times (middle panel), it yields a noticeable contribution to the error away from the BLs. The plots indicate that the transient component (a) scales (roughly) like $O(\dt^{2.5})$ in amplitude, and (b) has an $O(\dt^{0.5})$ wave length. This observed scaling occurs because the transient component has an i.c.\ with BLs of width $O(\dt^{0.5})$, and thus its dominant Fourier modes occur at wave numbers $O(\dt^{-0.5})$ and with magnitude $O(\dt^{0.5})$.

Figure~\ref{fig:errcvg_schro_dirk3} shows the error convergence results. When errors are considered over the full spatial domain (i.e., including the BLs), precisely the same results as for the heat equation (see Figure~\ref{fig:errcvg_heat_dirk3}) are obtained. In addition, we consider errors evaluated away from the BLs (light colors). As expected from the results above, these exhibit a more interesting behavior. Without any remedies to order reduction, we observe (roughly) an error scaling of $u = O(\dt^{2.5})$, $u_x = O(\dt^2)$, and $u_{xx} = O(\dt^{1.5})$, thus indicating order reduction effects away from the BLs. In addition, MBC and high WSO remove the order reduction, not only in the BLs, but also inside the domain. It should be re-iterated that the transient effects vanish after sufficiently long times (which, for most RK schemes, are $O(\dt)$, with a very large constant).

The observations collected here highlight that order reduction effects need not necessarily be limited to thin zones (i.e., BLs) near the domain boundaries, but can propagate into the interior of the domain, if for instance the PDE is a dispersive wave equation.

\subsection{Advection-Diffusion Equation}
\label{ssec:NumExamples_AdvectionDiffusion}
As revealed by the analysis in \Srm\ref{sec:orderloss_gte}, order reduction for IBVPs is intricately linked to boundary layers (BLs) produced by the time-stepping. A natural question is therefore: what happens in problems that possess a physical BL? To answer this question, we consider the linear advection-diffusion equation
\begin{equation}\label{eq:advdiff}
	u_t = \nu\,u_{xx} - u_x + f\quad\mbox{for}\quad(x,t)\in(0\/,\,1)\times(0,\,1.2]\/,
\end{equation}
with manufactured solution $u^*(x\/,\,t) = \sin(2\,\pi\,(x-t))\/$, and the nondimensional viscosity $\nu = 10^{-3}$. When $\nu$ is small, the equation becomes advection-dominated, and prescribing Dirichlet b.c.\ at the outflow boundary $x=1$ results in a BL of width $O(\nu)$ in the error (note that even though our $u^*$ does not have a BL, the error does).

\begin{figure}
\includegraphics[width=0.32\textwidth]{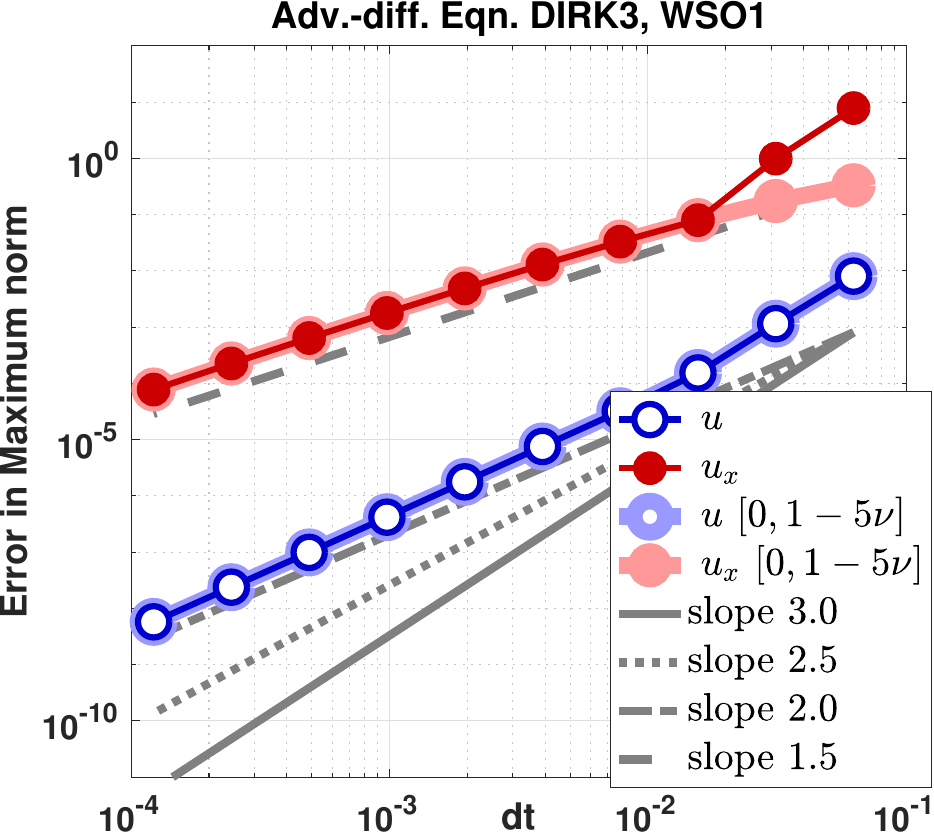}
\includegraphics[width=0.32\textwidth]{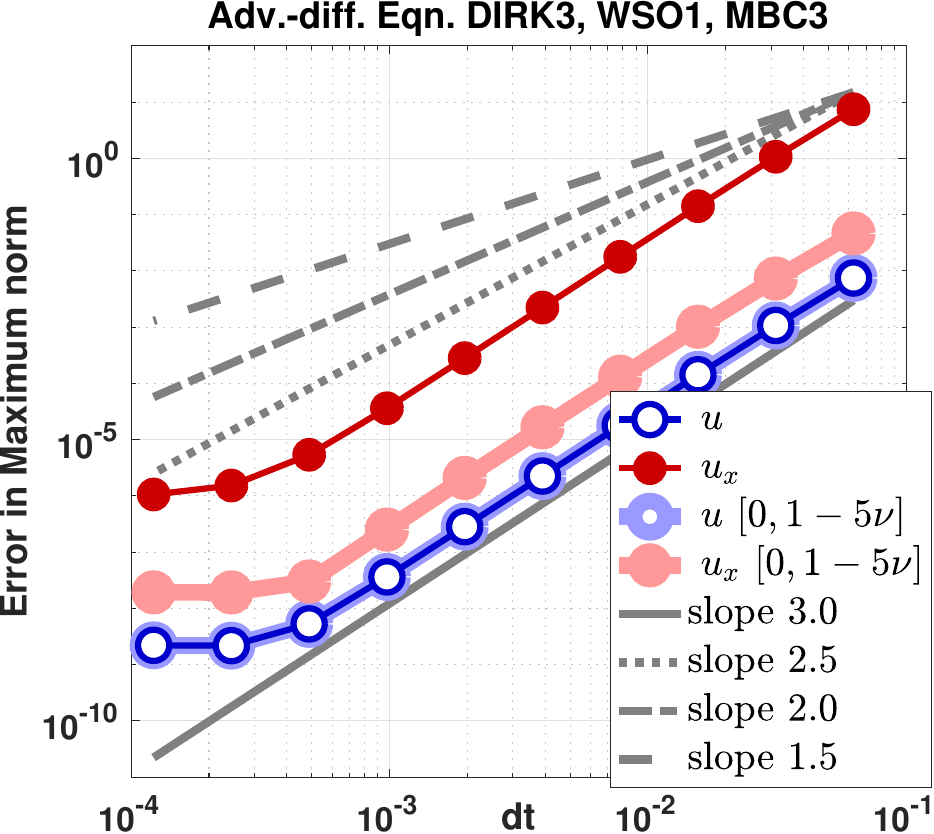}
\includegraphics[width=0.32\textwidth]{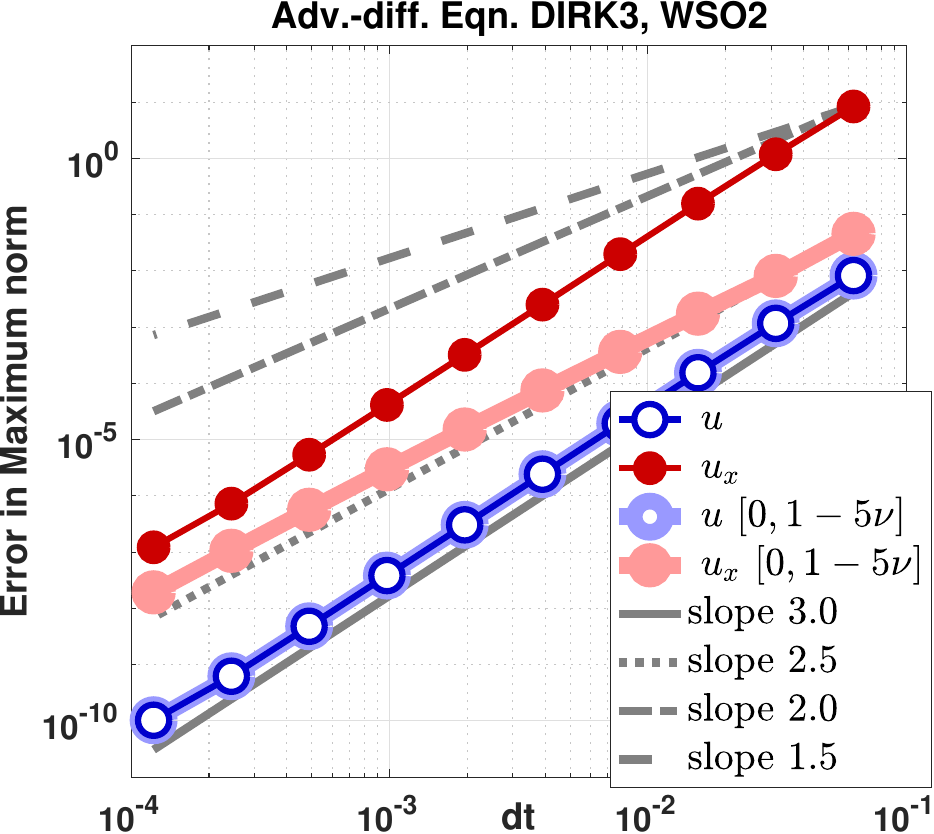}
\caption{Error convergence ($u$ blue; $u_x$ red) for 3rd order schemes for the advection-diffusion equation \eqref{eq:advdiff}. WSO~1 scheme with conventional b.c.~(left); same scheme with MBC (middle); WSO~2 scheme (right). The error measured away from the outflow boundary ($[0,1-5\nu]$) is shown for $u$ (light blue) and $u_x$ (light red).}
\label{fig:errconvg_advdiff_dirk3}
\end{figure}

Figure~\ref{fig:errconvg_advdiff_dirk3} show the convergence results. In addition to measuring errors (in the maximum norm) over the whole domain, we also study the error over the domain $x\in [0,1-5\nu]$, i.e., away from the BL (light colors). The results show L-shaped transitions in error behavior, depending on which types of BLs dominate. The $\nu$-BL is of magnitude $O(\dt^3)$ and width $O(\nu)$, thus its effect on the $u_x$ error is $O(\dt^3/\nu)$. In turn, the order reduction BLs (for the WSO~1 scheme) are of magnitude $O(\dt^2)$ and of width $O(\dt^{0.5})$, thus affecting the $u_x$ error with $O(\dt^{1.5})$. Balancing these expressions explains why the kink in the $u_x$ error (left panel) occurs at $\dt = O(\nu^{2/3})$. Likewise, the same error experiences a kink at $\dt = O(\nu^2)$ for the WSO~2 scheme (right panel).

Thus, for large $\dt$ values the $\nu$-BL dominates the error, and the scheme appears to not exhibit order reduction. However, for $\dt$ sufficiently small, the error behaves the same way it does for the heat equation (\Srm\ref{ssec:NumExamples_HeatEqn}), and order reduction becomes apparent. In line with our theory, the L-shapes in the errors are removed when MBC3 (middle panel) are applied. Those recover the full 3rd order convergence throughout the full range of $\dt$ values. The WSO~2 scheme (right panel) recovers a clean third order for $u$. However, it still loses an order in $u_x$, even though this becomes visible only for very small $\dt$ values.

\subsection{Viscous Burgers' Equation}
\label{ssec:NumExamples_VicousBurgers}
With this example we demonstrate that order reduction, as well as some remedies, also apply in nonlinear problems. We consider the viscous Burgers' equation
\begin{equation}\label{eq:visburgers}
	u_t + u\,u_x = \nu\,u_{xx} + f\quad\mbox{for}\quad(x,t)\in(0\/,\,1)\times(0,1]\/,
\end{equation}
with Dirichlet b.c.\ and $\nu=0.1$. The manufactured solution is $u^*(x,t) = \cos(2+10t)\sin(0.2+20x)$. The nonlinearity yields a nonlinear implicit equation at each stage, with $\mathcal{N} u = \nu u_{xx} - u u_x$, which is solved via standard Newton iteration.

A crucial limitation is that the third order term in the MBC, obtained from the expansion in \Srm\ref{sec:MBC}, contains terms that are not accessible without knowledge of the exact solution (see \Srm\ref{ssec:MBC_Limitations}). Hence, MBC3 cannot be formulated via the procedure introduced in \Srm\ref{sec:MBC}. However, MBC2 \emph{can} be formulated in terms of the data, and they coincide with the corresponding expression obtained for linear problems.

\begin{figure}
	\hfillL
	\includegraphics[width=0.32\textwidth]{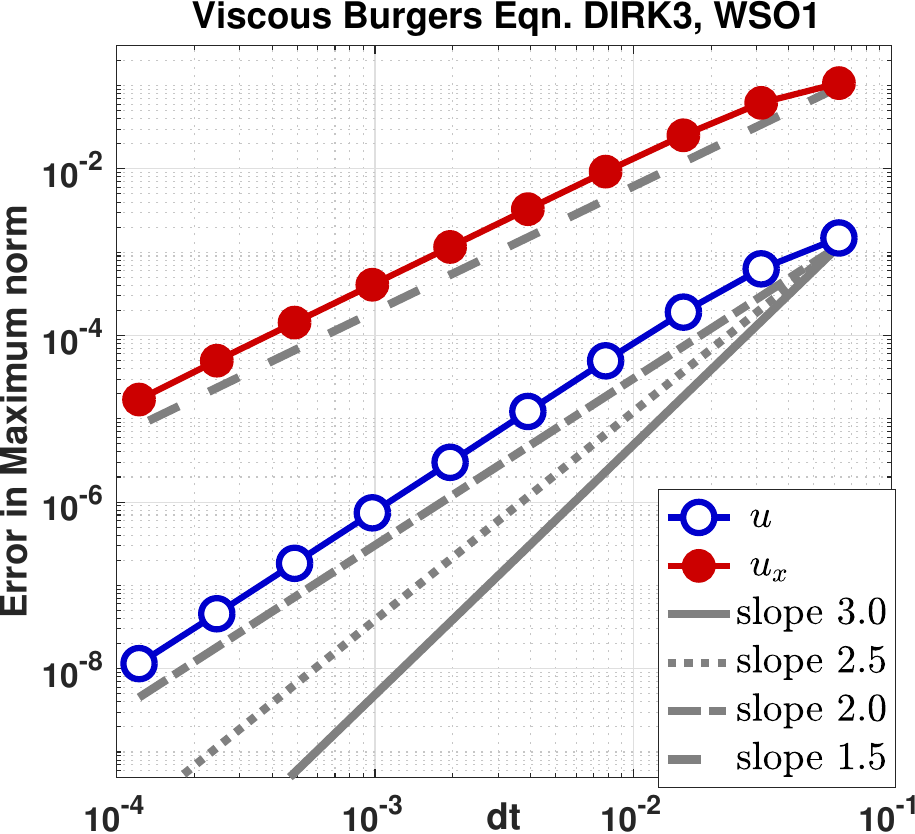}
	\hfill
	\includegraphics[width=0.32\textwidth]{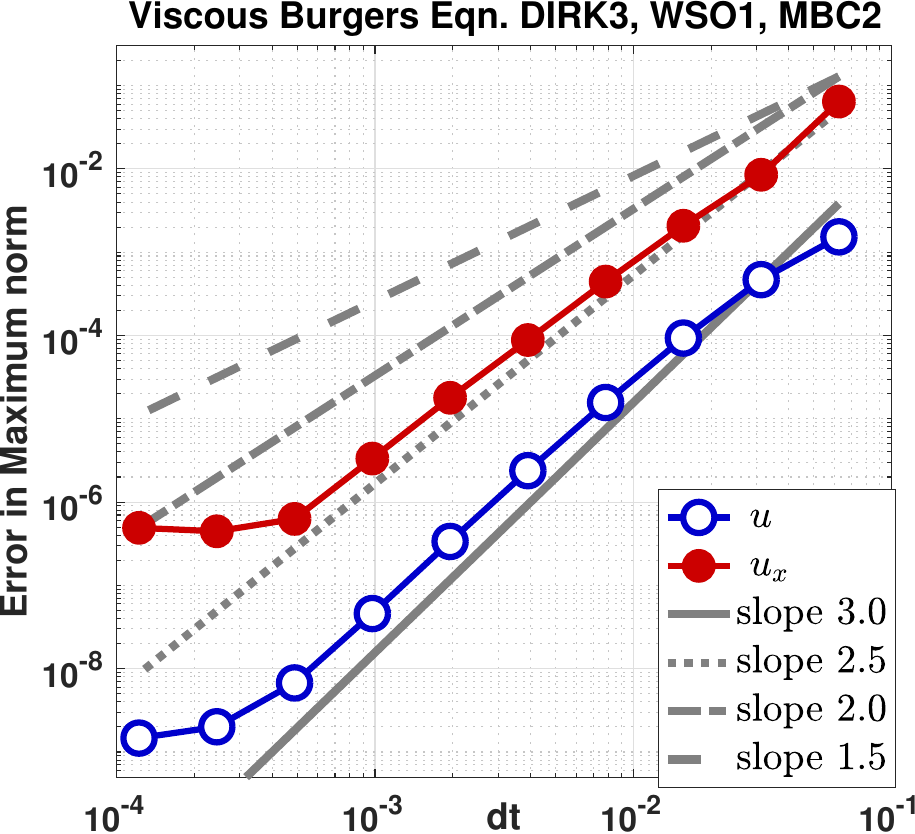}
	\hfill
	\includegraphics[width=0.32\textwidth]{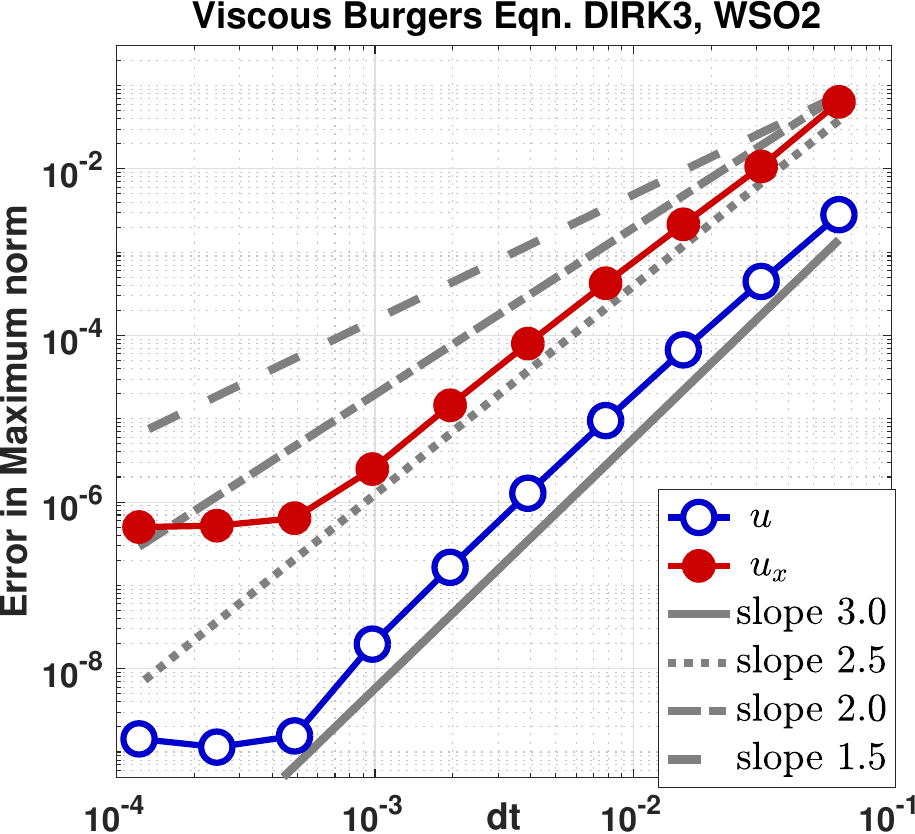}
	\hfillR
	\caption{Error convergence ($u$ blue; $u_x$ red) for 3rd order
		schemes for the viscous Burgers' equation \eqref{eq:visburgers}.
		WSO~1 scheme with conventional b.c.~(left); same scheme with MBC2 (middle); WSO~2 scheme (right).
		}
	\label{fig:errconvg_visburgers_dirk3}
\end{figure}

Figure~\ref{fig:errconvg_visburgers_dirk3} shows that the same type of order reduction arises as for linear problems (left panel). The MBC2 recover the full 3rd order for $u$, but lose orders for derivatives (middle). The same results are obtained with the WSO~2 scheme (right).

\subsection{Linear Advection Equation and Airy's Equation}
\label{ssec:NumExamples_Airy}
All examples above have $\diffop$ a second-order differential operator. To demonstrate that the order reduction results, as well as the remedies, apply more generally, we consider the following two problems:
\begin{enumerate}[ (1)]
	\item The linear advection equation: $u_t = u_x $ for $(x,t)\in(0,1)\times(0,1.2]$
	with Dirichlet b.c.\ at $x=0$ and manufactured solution
	$u^*(x,t) = \sin(2\pi(x-t))$.
	\item Airy's equation: $u_t = u_{xxx} + f$ for $(x,t)\in(0,1)\times(0,1]$
	with b.c.\ $u(0) = g(t)$, $u_x(0) = h_0(t)$, $u_x(1) = h_1(t)$, and manufactured
	solution $u^*(x,t) =\cos(15t)$.
\end{enumerate}
The respective results are shown in Figures~\ref{fig:errconvg_adv_dirk3} and~\ref{fig:errconvg_airy_dirk3}. In line with the theory in \Srm\ref{sec:orderloss_gte}, conventional b.c.\ render function values one order more accurate than the scheme's WSO, and $1/m$ orders per derivative are lost, where $m$ is the order of $\diffop$. MBC3 recover the full 3rd order for $u$, as well as derivatives up order $m$. Hence, one obtains $u = O(\dt^3)$, $u_x = O(\dt^3)$, and $u_{xx} = O(\dt^2)$ for $m=1$, and $u = O(\dt^3)$, $u_x = O(\dt^3)$, and $u_{xx} = O(\dt^3)$ for $m=3$. In contrast, the WSO~2 scheme recovers the full order in $u$ only. Hence $u = O(\dt^3)$, $u_x = O(\dt^2)$, and $u_{xx} = O(\dt^1)$ for $m=1$, and $u = O(\dt^3)$, $u_x = O(\dt^{2.67})$, and $u_{xx} = O(\dt^{2.33})$ for $m=3$.

\begin{figure}
	\hfillL
	\includegraphics[width=0.32\textwidth]{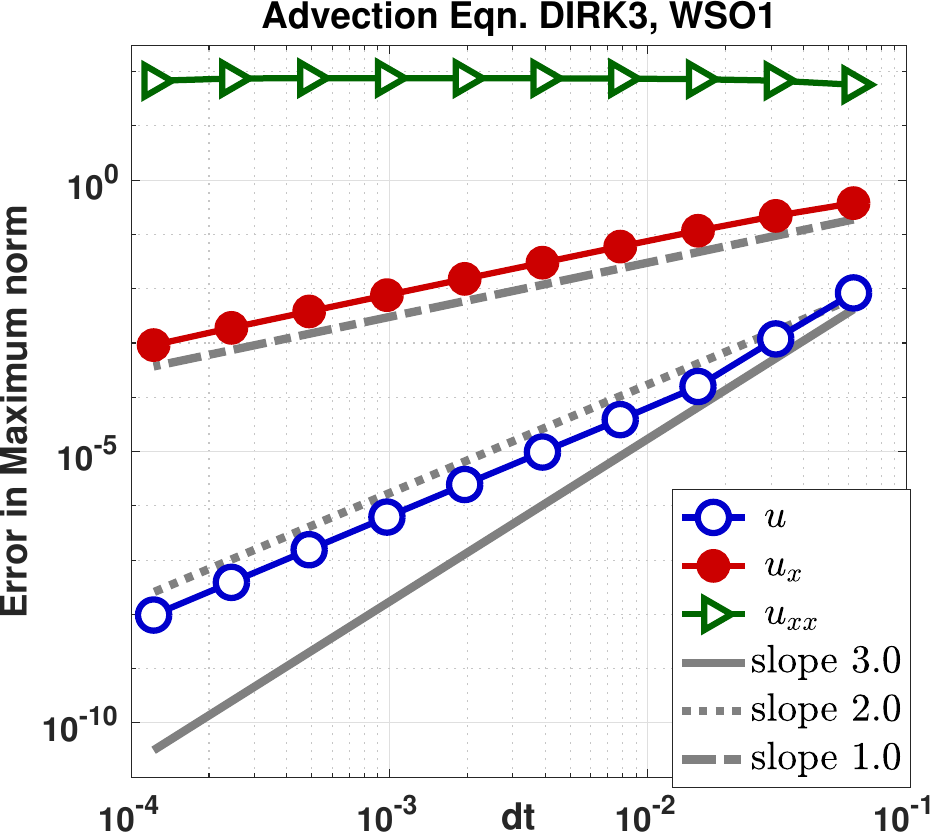}
	\hfill
	\includegraphics[width=0.32\textwidth]{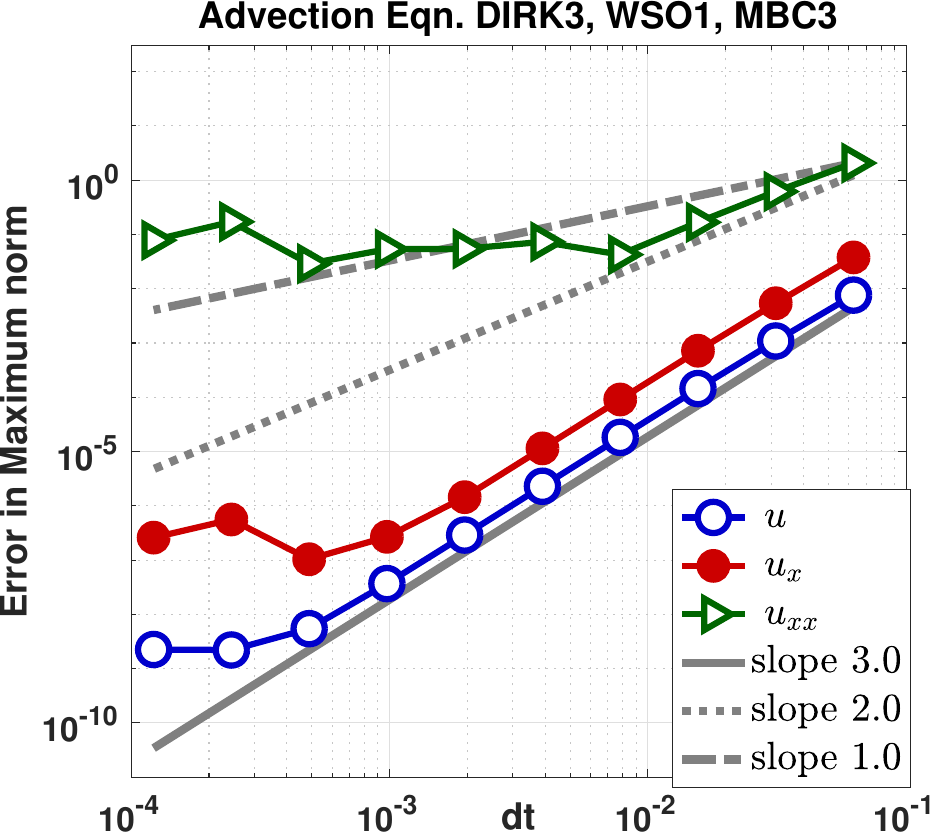}
	\hfill
	\includegraphics[width=0.32\textwidth]{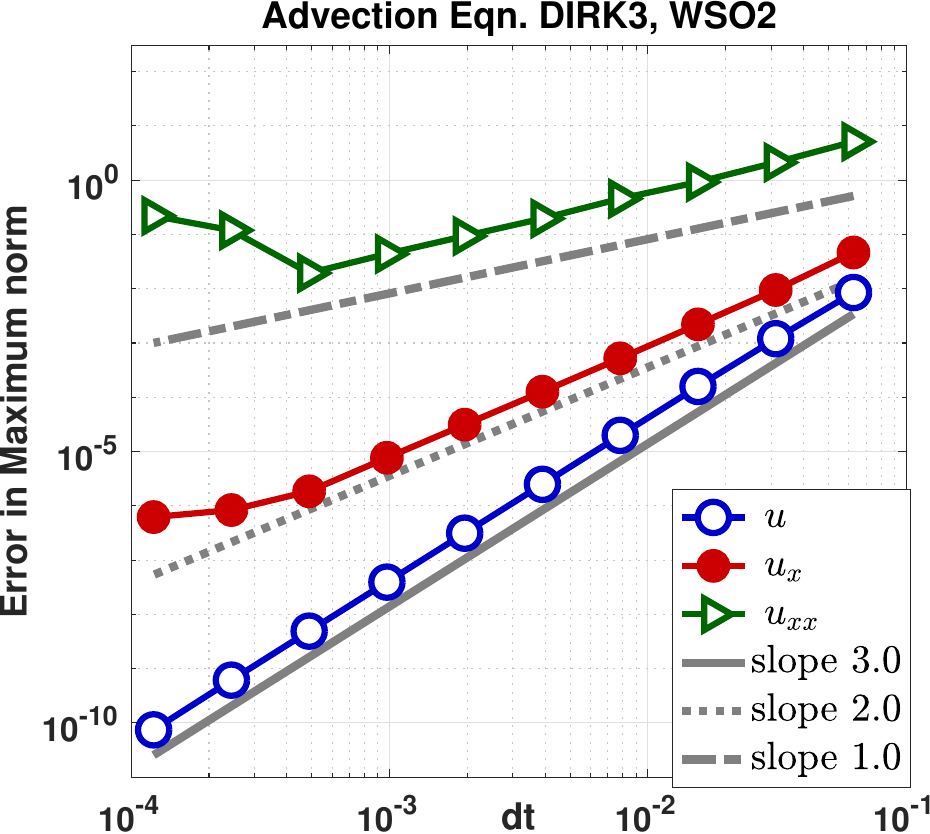}
	\hfillR
	\caption{Error convergence ($u$ blue; $u_x$ red; $u_{xx}$ green)
		for 3rd order schemes for the linear advection equation. WSO~1 scheme with
		conventional b.c.~(left); same scheme with MBC (middle); WSO~2 scheme (right).}	
	\label{fig:errconvg_adv_dirk3}
\end{figure}

\begin{figure}
	\hfillL
	\includegraphics[width=0.32\textwidth]{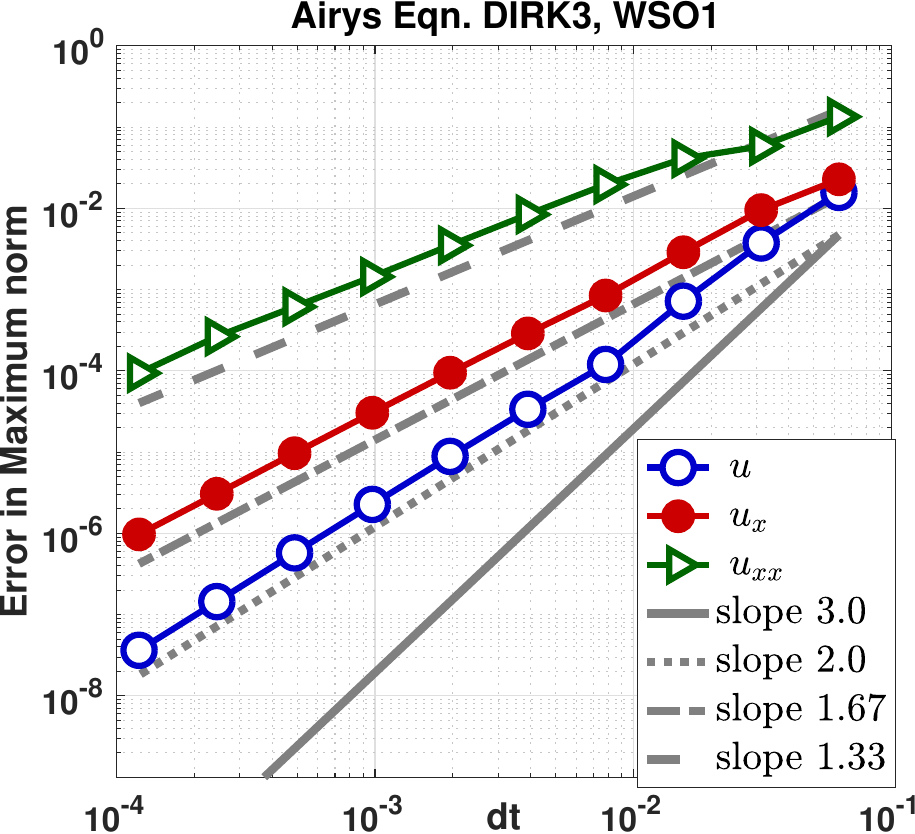}
	\hfill
	\includegraphics[width=0.32\textwidth]{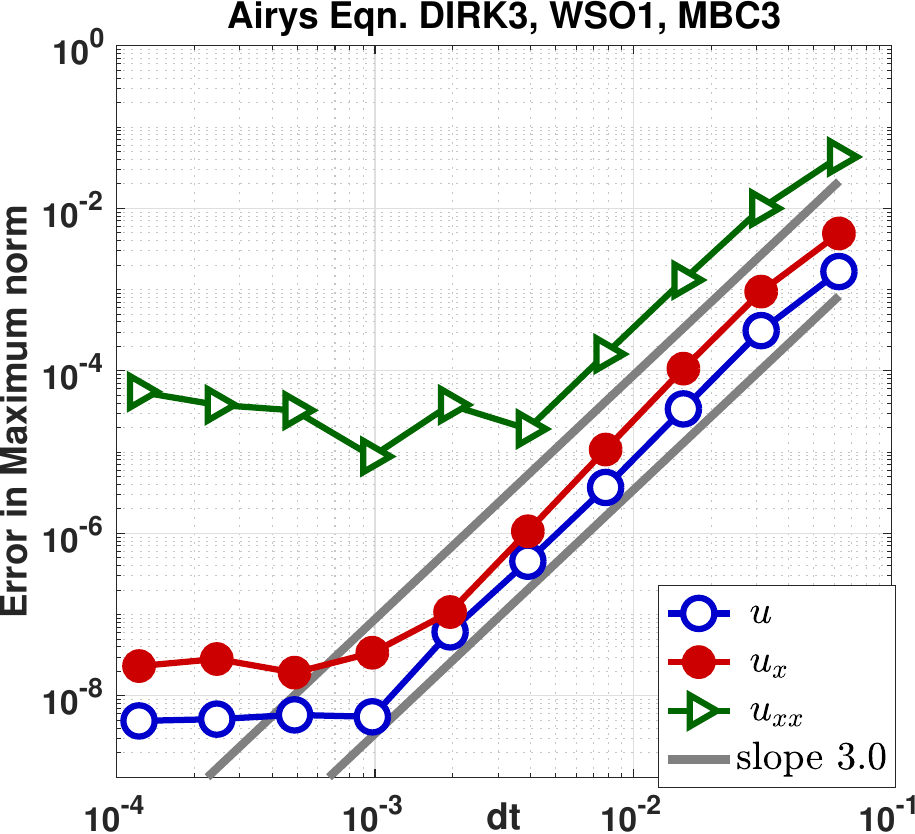}
	\hfill
	\includegraphics[width=0.32\textwidth]{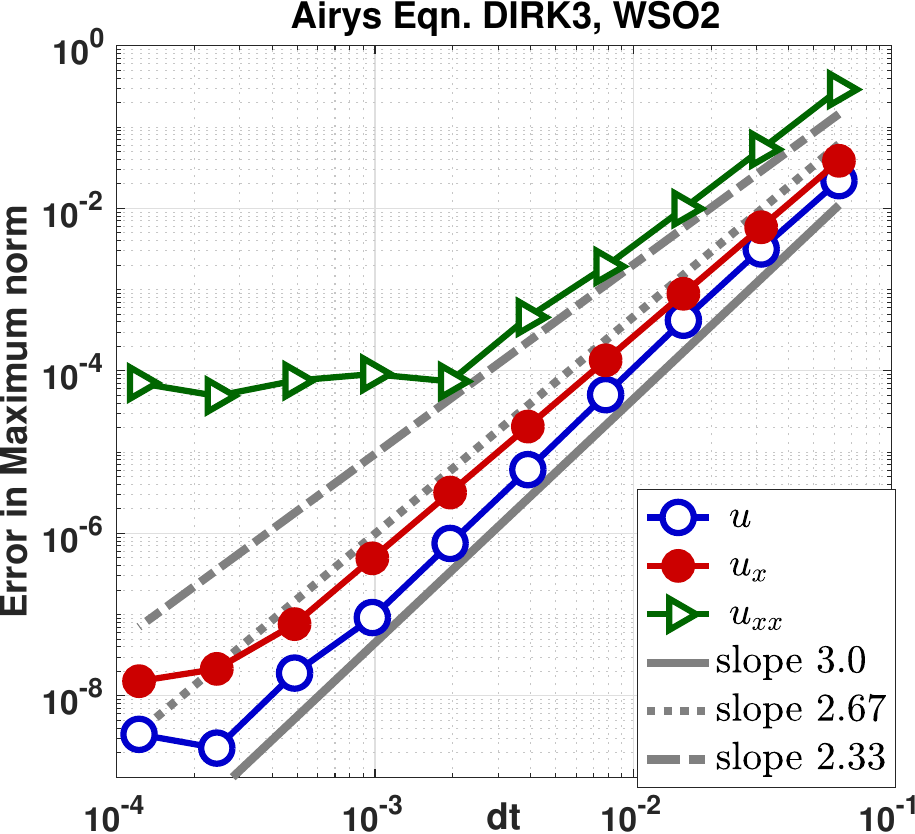}
	\hfillR
	\caption{Error convergence ($u$ blue; $u_x$ red; $u_{xx}$ green)
		for 3rd order schemes for the Airy's equation. WSO~1 scheme with
		conventional b.c.~(left); same scheme with MBC (middle); WSO~2 scheme (right).}
	\label{fig:errconvg_airy_dirk3}
\end{figure}

\section{Conclusions and Outlook}
\label{sec:outlook}
We have demonstrated that order reduction is a generic phenomenon in implicit RK time-stepping for IBVPs (with time-dependent data) that manifests in the formation of spatial boundary layers in the numerical solution. These originate because the stage update equations are singularly perturbed boundary value problems, where most types of b.c.\ generate a mismatch between the boundary and the interior. The global-in-time behavior of these boundary layers has been studied via modal and asymptotic analysis; and in that light, two different approaches to overcome order reduction have been examined: (a)~new conditions on the RK coefficients (weak stage order); and (b)~modified b.c.\ that render the boundary data mismatch as small as the order of the RK scheme. Modified b.c.\ have the advantage that they work for a wide range of RK schemes; however, they may be complicated to compute, and difficult to implement for non-linear problems. In turn, schemes with weak stage order require no modification of the spatial approximation; however, these new conditions rule out many existing RK methods.

Because order reduction is caused by boundary layers, one could be inclined to think that it is only a minor concern, because its effects can be ``felt'' only near the boundary. That is not the case, because:
\begin{enumerate}[(i)]
\item In many applications, the solution (and particularly its derivatives) at the boundary are important, such as stresses at boundaries (lift and drag in CFD).
\item In many multi-physics problems, information near the boundaries feeds back to the interior of the domain (non-local terms, fast waves, etc.).
\item The analysis in \Srm\ref{sec:orderloss_gte} only studies the $t\to\infty$ behavior of the error. Many problems (e.g., the Schr{\"o}dinger equation, see \Srm\ref{ssec:NumExamples_Schroedinger}) exhibit transient features that reduce the observed order away from the boundary at $O(1)$ times.
\item The boundary layers' thickness scales like some power of the time step: $\dt^{1/m}$, where $m$ is the order of the space operator. Hence, unless the time step is very small, their thickness may be considerable.
\end{enumerate}
In this context (particularly point (i)), it must be stressed that the presence of boundary layers that multistage methods almost always develop generally implies order reduction in (sufficiently high) spatial derivatives, even if the numerical solution itself is devoid of order reduction (due to weak stage order or modified b.c.).

Although the analysis in this paper focuses on problems with Dirichlet b.c., the order reduction phenomenon also arises with other types of b.c.\ (such as Neumann), and the analysis in \Srm\ref{ssec:orderresults} carries over with minor adaptations. With Neumann b.c., the obtained convergence orders are slightly different. For instance, for the heat equation, one obtains $O(\dt^{\tilde{q}+1.5})$, where $\tilde{q}$ is the weak stage order, for the error in function value (i.e., half an order better than with Dirichlet conditions), and half an order loss per derivative \cite{LubichOstermann1995quasilinear}.

The analysis in \Srm\ref{sec:orderloss_gte} also applies to RK schemes with a singular coefficient matrix $A$, such as Crank-Nicolson (CN) and EDIRK schemes \cite{KvaernoNorsettOwren1996}. The CN scheme is an example of a second-order scheme devoid of boundary layers, because the matrix $M$ has no small eigenvalues (one is $O(1)$ and the other is zero). However, CN is not $L$-stable, and it incurs the same problem as the implicit midpoint rule (see \Srm\ref{ssec:eigenmodeanalysis}), namely the growth factor approaches $-1$ as $\zeta\to-\infty$.

It should be stressed that order reduction arises in explicit RK schemes as well \cite{CarpenterBottliebAbarbanelDon1995, SanzVerwerHundsdorfer1986}. However, the semi-discrete analysis in this paper does not directly apply.

Finally, the weak stage order and the modal analysis presented in this paper apply beyond IBVPs and RK schemes. In \Srm\ref{ssec:stiff_ode} we briefly outline the role of weak stage order in avoiding order reduction in stiff ODEs, and in \Srm\ref{ssec:LMM_NoOR} we employ the modal analysis to show that linear multistep methods (LMMs) do \emph{not} exhibit order reduction.

\subsection{Order Reduction and Weak Stage Order for Stiff Linear ODEs}
\label{ssec:stiff_ode}
The concept of weak stage order, introduced in \Srm\ref{sec:WeakStageOrder}, has been studied in terms of its impact on boundary layers in IBVPs. Because the concept is also of interest to stiff ODEs (without a spatial interpretation), we examine the model ODE proposed by Prothero and Robinson \cite{ProtheroRobinson1974}:
\begin{equation}\label{Stiff_ODE}
	y' = \lambda (y-\phi(t)) + \phi'(t)
\end{equation}
with i.c.\ $y(0)=\phi_0$ and $\mbox{Re}\,\lambda< 0$. The exact solution $y(t) = \phi(t)$ is assumed analytic.

When a RK scheme (with coefficients $A$, $\vecipower{b}{T}$, and $\vecc$) is applied to \eqref{Stiff_ODE}, the error $\errl^{n+1}$, at time $t_{n+1}$, can be computed (see \cite[Chapter IV.15]{WannerHairer1991}) to be
\begin{equation}\label{TruncationErr}
	\errl^{n+1} = R(\zeta)\,\errl^{n} + \zeta\vecipower{b}{T}
	(I-\zeta\,A)^{-1}\vecipower{\lte}{n+1} + \ltel^{n+1},
\end{equation}
where $\zeta = \lambda\dt$. The vector $\vecipower{\lte}{n+1}$ and scalar $\ltel^{n+1}$ denote the truncation errors incurred at the intermediate stages, and at the end of the time step, respectively. Written in terms of derivatives of $\phi$, and the stage order residuals $\sov{j}$, they read as
\begin{equation*}
	\vecipower{\lte}{n+1}
	= \!\!\sum_{j\geq q+1}\frac{\dt\,^j}{(j-1)!}\,\sov{j}\phi^{(j)}(t_n),\;\;
	\ltel^{n+1} = \!\!\sum_{j\geq p+1}\frac{\dt\,^j}{(j-1)!}
	\left(\vecipower{b}{T}\vecipower{c}{j-1}-\frac{1}{j}\right)\phi^{(j)}(t_n).
\end{equation*}
Using these expressions, we obtain the RK schemes's stiff ODE order as follows:
\begin{proposition}\label{WeakStageOrder}
Suppose the Runge-Kutta scheme ($A$, $\veci{b}$, $\vecc$), with $A$ invertible, has weak stage order $\tilde{q}$. Then in the stiff limit $\zeta \ll -1$, the local truncation error $\errl^{1}$ for equation \eqref{Stiff_ODE}, which is obtained by setting $\errl^0 = 0$, is of order $\min\{\tilde{q}+1,p+1\}$.
\end{proposition}
\begin{proof}
Setting $\errl^0 = 0$ in \eqref{TruncationErr}, and substituting the formula for $\vecipower{\lte}{n+1}$, we have
\begin{equation}\label{WeakStageOrderOneStepError}
	\errl^1 = \sum_{j \geq q+1} \frac{\dt\,^j}{(j-1)!}\,
	\phi^{(j)}(t_0) \Big( \zeta\vecipower{b}{T}(I-\zeta\,A)^{-1} \,\sov{j}\Big)
	+ \ltel^{1}.
\end{equation}
Weak stage order $\tilde{q}$ means that the stage order residuals $\sov{j}$ for $1\leq j \leq \tilde{q}$ lie in an $A$-invariant space $\mathcal{V}$ that is orthogonal to $\veci{b}$. Thus $\mathcal{V}$ is also $(I-\zeta\,A)^{-1}$-invariant, and $(I-\zeta\,A)^{-1}\sov{j} \in \mathcal{V}$, hence
\begin{equation*}
	\vecipower{b}{T}(I-\zeta\,A)^{-1} \,\sov{j} = 0,
	\quad \text{for}\; 1\leq j \leq \tilde{q}.
\end{equation*}
As a result, the first $j \leq \tilde{q}$ terms in \eqref{WeakStageOrderOneStepError} vanish, resulting in the sum to be over $j \geq \tilde{q}+1$. In the stiff ODE limit, i.e., $\dt \ll 1$ and $\zeta \ll -1$, the term $\zeta\vecipower{b}{T}(I-\zeta\,A)^{-1} \sov{j}$ can be bounded in terms of $\|A^{-1}\|$, $\|\sov{j}\|$, and an $O(1)$ constant. Hence the expression for $\errl^1$ in \eqref{WeakStageOrderOneStepError} is $O(\dt\,^{\tilde{q}+1})$, while $\delta^{1} = O(\dt\,^{p+1})$, which together yields $\errl^1 = O(\dt\,^{\min\{\tilde{q}+1,p+1\}})$. 	
\end{proof}
\begin{remark}
By construction, weak stage order satisfies $\tilde{q}\ge q$. Hence Proposition~\ref{WeakStageOrder} improves the stiff error bound given by the stage order $q$.
\TheoremEnd
\end{remark}
\begin{remark}
For stiff ODEs, the global truncation error is of order $\tilde{q}$. This is a difference to PDE IBVPs, for which the global error is of order $\tilde{q}+1$. Hence, to avoid OR for stiff ODEs, the RK scheme must have $\tilde{q} = p$. In contrast, for IBVPs the choice $\tilde{q} = p-1$ suffices.
\TheoremEnd
\end{remark}

\subsection{Order Reduction in non-Runge-Kutta Time-Stepping Methods}
\label{ssec:LMM_NoOR}
This paper shows that order reduction for IBVPs, due to boundary layers in the spatial error, arises rather generically in Runge-Kutta methods. It also occurs in other multistage schemes, or any scheme that can be recast as a multistage scheme, see \S \ref{ssec:PriorResearch}. In contrast, schemes that achieve high order via a single BVP solve per step are devoid of the order reduction mechanism. This includes linear multistep methods (LMMs) (see \cite{Lubich1991} and \cite[Chapter IV.15]{HairerWanner1999} for LMM error estimates), for example backward differentiation formula (BDF) methods.

To illustrate that LMM do not result in a singular perturbation problem, we use the framework introduced in \Srm\ref{sec:orderloss_gte}. A general $s$-step, implicit, LMM for solving the IBVP \eqref{eq:IBVP} with Dirichlet b.c.\ takes the form
\begin{equation}\label{eq:LMM}
	u^{n+s} = \sum_{j=0}^{s-1}\alpha_j u^{n+j}
	+ \dt \sum_{j=0}^s \beta_j (\diffop u^{n+j}\!+\!f^{n+j})
	\quad\text{with b.c.~} u^{n+s} = g(t_n + s\dt),
\end{equation}
where $\beta_s \neq 0$. The scheme \eqref{eq:LMM} defines a linear recursion relation for $u^{n+s}$, and can be written using matrix notation on the vector solution $(u^{n+s},\dots,u^{n+1})^T$, see \cite{HairerNorsettWanner1993}. To characterize the error, denote $\err^n = u^n - u^*(t_n)$, and let $\vecipower{\err}{n+1} = (\err^{n+s},\dots,\err^{n+1})^T$. One can then obtain an equation for the error vector $\vecipower{\err}{n+1}$, similar to \eqref{eq:ApproxErrIRK}, by considering time-periodic solutions $\vecipower{\err}{n} = z^n\,\vec{\err}(x)$ and $\vecipower{\lte}{n} = z^n\,\vec{\ltes}(x)$:
\begin{equation}\label{eq:LMM_EigVect}
	\veci{\err} - M\,\diffop\,\veci{\err}
	= N\,\veci{\ltes}\quad\text{with homogeneous b.c.\ for }\veci{\err}.
\end{equation}
Here $M = \frac{\dt}{z-1}AB + \dt B$ and $N = E+ \frac{1}{z-1}AE$ with
\begin{equation*}
	A = \begin{pmatrix}
		\alpha_{s-1} & \alpha_{s-2} & \cdots & \alpha_0 \\
		1 & 0  & & \\
		&\ddots & \ddots & \\
		& & 1 & 0
	\end{pmatrix},\quad
	B =
	\begin{pmatrix}
		\beta_s & \cdots & \beta_2 & \beta_1+\frac{1}{z}\beta_0\\
		0 & \cdots & 0 & 0 \\
		\vdots & & \vdots & \vdots\\
		0 & \cdots & 0 & 0
	\end{pmatrix},
\end{equation*}
$E = \veci{e}_1\vecipower{e}{T}_1$ and $\veci{e}_1 = (1,0,\dots,0)^T$. Note that $M$ has exactly $s-1$ zero eigenvalues with right eigenvectors corresponding to the $s-1$ dimensional null space of $B$. Moreover, $\frac{\dt}{z-1}AB$ is a rank-1 matrix with one $O(1)$ eigenvalue. As a result, $M$ has exactly one $O(1)$ eigenvalue and $s-1$ zero eigenvalues. Therefore, equation \eqref{eq:LMM_EigVect} does not incur a singularly perturbed BVP, and hence there are no BLs in the global solution $\vec{\err}(x)$. Another way to interpret these results is as follows. Every time-stepping scheme has one $O(1)$ eigenvalue in the error propagation matrix $M$, due to consistency. Time-stepping schemes that require only a single solve per time-step are devoid of order reduction.

We conclude by mentioning general linear methods (GLMs) \cite[Chapter 2]{Jackiewicz2009}, which are schemes with multiple steps and multiple stages. Although more complex, GLMs may inherit many of the good properties of Runge-Kutta and multistep methods, but also some of their drawbacks. Specifically, having multiple stages triggers the mechanism for boundary layers. Conversely, the added flexibility of GLMs allows for the construction of diagonally implicit schemes with desirable stability properties and high stage order \cite{ButcherJackiewicz1993, ZhangSanduBlaise2014}.

\section*{Acknowledgments}
The authors would like to thank David Ketcheson and Adrian Sandu for helpful conversations and suggestions.
This material is based upon work supported by the National Science Foundation under Grants DMS--1719637 (Rosales), DMS--1719640 (Zhou), DMS--2012271 and DMS--2309728 (both Seibold), DMS--2012268 and DMS--2309727) (both Shirokoff).
D. Shirokoff was supported by a grant from the Simons Foundation (\#359610).

\appendix

\section{Implicit Runge-Kutta Schemes Used in This Paper}
\label{app:listofRK}
All the DIRK schemes listed here are from \cite[Chapter IV.6]{WannerHairer1991}. Let $s$ be the number of stages, and $p$, $q$, and $\tilde{q}$ denote the order of the scheme, stage order, and WSO, respectively.

\noindent\textbf{Stiffly accurate DIRK with $s=2$, $p=2$, $q = 1$, $\tilde{q}=1$ (DIRK2)}:
	\begin{equation}\label{eq:DIRK2}
	\begin{array}{c|c c  }
	\gamma & \gamma &  \\
	1 & 1-\gamma & \gamma  \\
	\hline
	& 1-\gamma & \gamma
	\end{array}
	\quad \text{for}\quad \gamma = 1-\frac{\sqrt{2}}{2}
	\end{equation}
	
\noindent\textbf{Non stiffly accurate DIRK with $s=2$, $p=3$, $q=1$, $\tilde{q}=1$}:
	\begin{equation}\label{eq:DIRK3_2s}
	\begin{array}{c|c c  }
	\gamma & \gamma &  \\
	1-\gamma & 1-2\gamma & \gamma  \\
	\hline
	& \frac{1}{2} & \frac{1}{2}
	\end{array}
	\quad \text{for}\quad \gamma = \frac{3+\sqrt{3}}{6}
	\end{equation}
	
\noindent\textbf{Stiffly accurate DIRK with $s=3$, $p=3$, $q = 1$, $\tilde{q}=1$ (DIRK3)}:
	\begin{equation}\label{eq:DIRK3}
	\begin{array}{c|c c c}
	0.4358665215 & 0.4358665215 &  &   \\
	0.7179332608 & 0.2820667392 & 0.4358665215 &  \\
	1 & 1.208496649 & -0.644363171 & 0.4358665215\\
	\hline
	& 1.208496649 & -0.644363171 & 0.4358665215
	\end{array}
	\end{equation}
	
\noindent\textbf{Stiffly accurate DIRK with $s=5$, $p=4$, $q = 1$, $\tilde{q}=1$ (DIRK4)}:
	\begin{equation}\label{eq:DIRK4}
	\begin{array}{c|c c c c c }
	1/4 		& 1/4          & & & &\\
	3/4 		& 1/2		  & 1/4& &  &\\
	11/20 	& 17/50 	  & -1/25 & 1/4 &  &\\
	1/2 		& 371/1360 & -137/2720 & 15/544 & 1/4 &\\
	1 		& 25/24 & -49/48 & 125/16 & -85/12 & 1/4\\
	\hline
	& 25/24 & -49/48 & 125/16 & -85/12 & 1/4
	\end{array}
	\end{equation}

\section{Order Reduction and Periodic Solutions}
\label{app:periodicsol}
Here we outline a class of problems where OR originates through periodically
forced solutions. Assume the following applies to (\ref{eq:IBVP}) and
(\ref{eq:dirk_step_intermediatestages}--\ref{eq:dirk_step_finalstage}):
\begin{enumerate}[(i)]
\item
\emph{The operator $\diffop\/$, with homogeneous boundary conditions,
has a complete set of normal mode eigenfunctions, and corresponding
eigenvalues, $\{\lambda_{\ell}, \varphi_\ell(x) \}_{\ell=1}^\infty\/$}.
\item
\emph{The eigenvalues, as well as the corresponding scheme growth
factors $R_\ell = R(\lambda_\ell \dt)\/$, satisfy
$|e^{\lambda_\ell \dt}| \leq R_*\/$ and $|R_\ell| \leq R_*\/$, where
$R_* \leq 1\/$ is a constant.}\footnote{\,Example: heat equation in
    $0<x<\pi\/$, with Dirichlet b.c., and a stable scheme. Then
    $\varphi_\ell = \sin(\ell\,x)\/$ and $\lambda_\ell = -\ell^2\/$.}
\item
\emph{The scheme's growth factor satisfies
$D_g(\zeta) = |\frac{e^\zeta-R(\zeta)}{\zeta^{p+1}}| \leq B_*\/$ for
$\mbox{Re}(\zeta)\leq 0\/$, where $p\/$ is the scheme's order and
$B_*\/$ is a constant.}
\item
\emph{The initial conditions $u_{ic} = \sum \alpha_\ell\,\varphi_\ell\/$
are ``smooth enough'', in the sense that the terms in
$\sum \alpha_\ell\,\lambda_\ell^{\alpha}\,\varphi_\ell\/$ satisfy the Weierstrass M-test for all $0 \leq \alpha \leq p+1$ so that the series converges uniformly and absolutely to a continuous function (in $\overline{\Omega}$).}
\end{enumerate}
Then OR can occur only due to the ``periodic component'' (defined below)
of the solution.

First, we split the solution of the scheme equations into: (1)~a homogeneous
component $u_h^n\/$, which satisfies the initial conditions, with homogeneous
b.c.\ and no forcing; and (2)~a ``periodic component'' $u_\text{p}^n\/$, which
satisfies the scheme equations (with forcing and full b.c.) with
zero initial conditions (why we call this the ``periodic
component'' is explained below).

First we show that $u_h^n\/$ exhibits no OR; hence \emph{any OR that occurs must do so solely due to the periodic component.}  Clearly $u_h^n = \sum \alpha_\ell\,R_\ell^n\,\varphi_\ell\/$. The corresponding
PDE solution $u_h = \sum \alpha_\ell\,e^{\lambda_\ell\,t}\,\varphi_\ell\/$. It
follows that
\begin{equation}\label{eq:proof}
\left\|\frac{u_h(t_n) - u_h^n}{\dt^p}\right\|_\infty \leq
   n \, \dt \, R_{*}^{n-1}\,
   \sum_\ell |\alpha_\ell|\,D_g(\lambda_\ell\,\dt)|\lambda_\ell|^{p+1}\,
   \|\varphi_\ell\|_\infty\/ \leq C,
\end{equation}
where $C$ is a constant independent of $n$ and $\dt$. The second inequality in \eqref{eq:proof} follows since: $n \dt R_*^{n}$ is bounded independent of $n$; while  (iii-iv) imply the summation converges uniformly to a continuous function (on $\overline{\Omega}$). Hence,
$\left\|u_h(t_n) - u_h^n\right\| \leq C \dt^{p} = O(\dt^p)$.

We argue that the periodic component $u_\text{p}^n$ is a linear superposition of periodic solutions (hence the name). We can solve for $u_\text{p}^n$ using the Mellin/$z$-transform (correspondingly, the Laplace transform for the PDE). Because of (ii), the transform is analytic for $|z| \geq 1\/$, so the inverse Mellin transform can be written as an integral over the unit circle (correspondingly, the inverse Laplace transform can be written as an integral over the imaginary axis). However, because the initial data vanish, the integrands in these inverse transforms are actually periodic solutions to the scheme/equation, with forcing and b.c.\ provided by the transforms of the forcing and b.c.\ of the problem defining the periodic component.

Finally, if (iv) does not apply, then $u_h$ may exhibit OR. However, for dissipative PDEs, $u_h$ decays exponentially in time, and hence eventually, any OR that occurs will be dominated by the error in the periodic component $u_\text{p}$.

\section{Proof that the Square Bracket in \eqref{eq:ErrorFunctions} is $O(\dt^p)$}
\label{app:proof}
Here we compute the first bracket $B_1(x)$ in \eqref{eq:ErrorFunctions}, and show that $B_1(x) = O(\dt^p)$, where:
\begin{equation}\label{eq:B1}
	B_1(x) :=
	\underbrace{\left[\vecipower{\Erralpha}{T}\veci{r}_1 \, \ef_1(x) +
	\zeromode(x) +
	\sum_{i=2}^{s}\vecipower{\Erralpha}{T}\veci{r}_i \, \rsol_i(x)
	\right]}_{\textrm{Bracket } 1}.
\end{equation}
The intuitive underlying reason that the first square bracket $B_1(x)$ in \eqref{eq:ErrorFunctions} satisfies $B_1(x) = O(\dt^p)$ is because this bracket represents the error contribution from the regular part of the asymptotic expansion---which is a Taylor expansion in the small parameter $\dt\/$. Because the RK scheme order arises from satisfying the equation up to $O(\dt^p)$ upon expanding the solution in powers of $\dt$, this result should not be surprising, even though the proof is not immediate. Generally, one can expect (and prove, though we do not do it here) that, wherever a regular expansion for the numerical solution applies, OR does not occur.

In the calculation below we make use of the following formulas:
\begin{enumerate}
    \item[(ai)] From Proposition~\ref{prop:1stLeftEigenVector}, the term $\vecipower{\Erralpha}{T}\veci{r}_1 \, \ef_1(x) = O(\dt^p)$.
    \item[(aii)] From (\ref{eq:MatrixM}--\ref{eq:ErrorFunctions}) we have: $\zeromode = -\frac{1}{z}\vecipower{\Erralpha}{T}\vec{\lte} + O(\dt^p)$.
    \item[(aiii)] From \eqref{eq:ErrorFunctions} the function $\veci{h}(x) = \frac{1}{z} \veci{\lte} + O(\dt^p)$ (since $\ltel = O(\dt^p)$), hence:
        \begin{equation}\label{app:eq1}
            \ampk{i}\/U^*(x) := \vecipower{l}{T}_i \vec{h} = \frac{1}{z} \vecipower{\ell}{T}_i \vec{\lte} + O(\dt^p).
        \end{equation}
    \item[(aiv)] From the Definition in \eqref{eq:reg_sol}:
    	\begin{align*}
    	    \rsol_i(x)  &= \ampk{i}\left(
			U^* + \lambda_i \diffop U^* + \lambda_i^2 \diffop^2 U^* + \ldots + \lambda_i^{\mpow} \diffop^{\mpow} U^*
		\right), \quad\quad \textrm{for } 2 \leq i \leq s. \\
		    &= \left(
			I + \lambda_i \diffop  + \lambda_i^2 \diffop^2  + \ldots + \lambda_i^{\mpow} \diffop^{\mpow}
		\right) \frac{1}{z} \vecipower{\ell}{T}_i \vec{\lte} + O(\dt^p) \quad\quad \textrm{(using \eqref{app:eq1}})
	    \end{align*}
    \item[(av)]  From Proposition~\ref{prop:1stLeftEigenVector}, $\vecipower{\ell}{T}_1 \veci{\lte} = O(\dt^{p})$ so that
        \begin{equation}\label{app:eq2}
	        \frac{1}{z} \vecipower{\Erralpha}{T}\veci{r}_1 \left[ I + \lambda_1 \diffop + \lambda_1^2 \diffop^2 + \ldots + \diffop^m \right] \vecipower{\ell}{T}_1\lte = O(\dt^p).
        \end{equation}	
\end{enumerate}
Substituting the expressions from (ai), (aii) and (aiv) into equation \eqref{eq:groupterms} yields:
\begin{equation}\label{eq:step1}
    B_1(x) =
-\frac{1}{z} \vecipower{\Erralpha}{T} \sum_{i=1}^{s} \veci{r}_i \, \vecipower{\ell}{T}_i \veci{\lte} + \frac{1}{z} \vecipower{\Erralpha}{T} \sum_{i=2}^{s} \veci{r}_i \, \left( I + \lambda_i \diffop + \ldots + \lambda_i^m \diffop^m\right) \vecipower{\ell}{T}_i \veci{\lte} + O(\dt^p).
\end{equation}
In \eqref{eq:step1} we have inserted a resolution of identity $I = \sum_{i=1}^{s} \veci{r}_i \vecipower{\ell}{T}_i$ into the first term. Finally, adding (av) (which is an $O(\dt^p)$ correction) to \eqref{eq:step1} yields:
\begin{equation}\label{eq:step2}
\begin{split}
    B_1(x) &=
-\frac{1}{z} \vecipower{\Erralpha}{T} \sum_{i=1}^{s} \veci{r}_i \, \vecipower{\ell}{T}_i \veci{\lte} + \frac{1}{z} \vecipower{\Erralpha}{T} \sum_{i=2}^{s} \veci{r}_i \, \left( I + \lambda_i \diffop + \ldots + \lambda_i^m \diffop^m\right) \vecipower{\ell}{T}_i \veci{\lte} + O(\dt^p) \\
	&= \frac{1}{z} \vecipower{\Erralpha}{T} \left( \sum_{j = 1}^{m} \sum_{i=1}^{s} \diffop^{j} \lambda_i^{j} \veci{r}_i \vecipower{\ell}{T}_i \right) \veci{\lte} + O(\dt^{p})
	=  \frac{1}{z} \left( \sum_{j = 1}^{m} \diffop^{j} \vecipower{\Erralpha}{T} \, M^{j} \right) \veci{\lte} + O(\dt^{p})
\end{split}
\end{equation}
We now use the following two identities
\begin{align}\label{eq:iden1}
	&\vecipower{\Erralpha}{T} M^{j} \sov{i} = 0, \quad\quad \quad\quad \textrm{for } 2 \leq i + j \leq p. \\ \label{eq:iden2}
	& M^j = \sum_{i = 1}^{s} \lambda_i^j \vec{r}_j \vecipower{\ell}{T}_j =
	\lambda_1^j \veci{r}_1 \vecipower{\ell}{T}_1 + O( \dt^j).
\end{align}
Equation (\ref{eq:iden1}) follows from a direct calculation: $\vecipower{\Erralpha}{T} M^{j}$ is spanned by vectors of the form $\vecipower{b}{T} A^{v}$ for $0 \leq v \leq j-1$ so that Proposition~\ref{prop:StageOrderVectorOrthog} implies (\ref{eq:iden1}). The second identity (\ref{eq:iden2}) follows from a spectral expansion of $M$ with $\lambda_i = O(\dt)$ for $2\leq i \leq s$. Thus:
\begin{align*}
B_1(x)
&=\frac{1}{z} \sum_{i = 1}^{p-1} \sum_{j = 1}^{m} \frac{\diffop^j\, (\dt)^i}{(i-1)!}  \vecipower{\Erralpha}{T} \, M^{j} \sov{i} + O(\dt^{p}) \\
&=\frac{1}{z} \sum_{i = 1}^{p-1} \sum_{j = p - i+1}^{m} \frac{\diffop^j \, (\dt)^i}{(i-1)!}  \vecipower{\Erralpha}{T} \, M^{j} \sov{i} + O(\dt^{p})  \\
&=\frac{1}{z} \sum_{i = 1}^{p-1} \sum_{j = p - i+1}^{m} \frac{\diffop^j \, (\dt)^i}{(i-1)!}  \vecipower{\Erralpha}{T} \left( \lambda_1^j \veci{r}_1 \underbrace{\vecipower{\ell}{T}_1 \sov{i}}_{=O(\dt^{p-i})} + O(\dt^{p - i + 1}) \right) + O(\dt^{p})
= O(\dt^p).
\end{align*}
The identity $\vecipower{\ell}{T}_1 \sov{i} = O(\dt^{p-i})$ follows from \eqref{eq:ldottau} in Proposition ~\ref{prop:1stLeftEigenVector}.

\vspace{1.0em}
\bibliographystyle{plain}
\bibliography{references}

\vspace{1.5em}
\end{document}